\documentclass[11pt,a4paper]{amsart}
\usepackage[utf8]{inputenc}
\usepackage{amsmath}
\usepackage{amsfonts}
\usepackage{amssymb}
\usepackage{graphicx}
\usepackage{booktabs}
\usepackage{amsaddr}
\usepackage[left=2.2cm,right=2.2cm,top=2cm,bottom=2cm]{geometry}
\usepackage{mathtools}
\usepackage{mathrsfs}
\usepackage{tikz}
\usepackage{float}
\floatstyle{ruled}
\newfloat{algorithm}{tbp}{loa}
\floatname{algorithm}{Algorithm}
\newenvironment{algorithmic}{\small\begin{list}{}{\setlength{\leftmargin}{1em}\setlength{\itemsep}{3pt}\setlength{\parsep}{0pt}}}{\end{list}}
\newcommand{\STATE}{\item}
\newcommand{\FOR}[1]{\item\textbf{for} #1 \textbf{do}\begin{list}{}{\setlength{\leftmargin}{1em}\setlength{\itemsep}{3pt}\setlength{\parsep}{0pt}}}
\newcommand{\FORALL}[1]{\FOR{all #1}}
\newcommand{\ENDFOR}{\end{list}\item\textbf{end for}}

\usepackage{enumitem}
\usetikzlibrary{decorations.pathreplacing}
\usetikzlibrary{fadings}

\newcommand{\D}{{\mathop{}\!\mathrm{d}}} 

\newcommand{\R}{\mathbb{R}}
\newcommand{\Q}{\mathbb{Q}}

\newcommand{\N}{\mathbb{N}}
\newcommand{\Pc}{\mathcal{P}}
\newcommand{\EE}{\mathbb{E}}

\newcommand{\PP}{\mathbb{P}}

\newcommand{\E}{\mathbb{E}}

\newcommand{\X}{\mathcal{X}}

\newcommand{\one}{ 1 \hspace{-3pt} \mathrm{l}} %

\DeclareMathOperator*{\argmin}{arg\,min}
\DeclareMathOperator*{\Bin}{Bin}

\numberwithin{equation}{section}  
\newtheorem{defn}{Definition}[section]
\newtheorem{exa}[defn]{Example}
\newtheorem{rem}[defn]{Remark}
\newtheorem{thm}[defn]{Theorem}
\newtheorem{prop}[defn]{Proposition}
\newtheorem{cor}[defn]{Corollary}
\newtheorem{lem}[defn]{Lemma}

\newtheorem{asu}[defn]{Assumption}

\usepackage[authoryear,round]{natbib}
\usepackage{hyperref}
\providecommand{\doi}[1]{\url{https://doi.org/#1}}
\usepackage[textwidth=3.2cm]{todonotes}
\usepackage{xcolor}

\hypersetup{
    pdftitle={Q-Learning under Finite Model Uncertainty},
    pdfauthor={Julian Sester and Cecile Decker},
    pdfkeywords={Q-learning, Markov decision processes, model uncertainty, robust control, stochastic approximation},
    colorlinks,
    linkcolor={red!50!black},
    citecolor={blue!50!black},
    urlcolor={blue!80!black}
}
\title{Q-Learning under Finite Model Uncertainty}
\author[J. Sester and C. Decker]{Julian Sester$^{1,*}$, C\'{e}cile Decker$^{1}$}

\begin{document}
\maketitle

\begin{center}
\normalsize{\today} \\ \vspace{0.5cm}
\small\textit{
$^{1}$National University of Singapore, Department of Mathematics,\\ 21 Lower Kent Ridge Road, Singapore 119077, Singapore.}     
\\[2mm]
{$^*$Corresponding author, email:jul{$\_$}ses@nus.edu.sg}
\end{center}

\begin{abstract}
We propose a robust Q-learning algorithm for Markov decision processes under model uncertainty when each state--action pair is associated with a finite ambiguity set of candidate transition kernels. This finite-measure framework enables user-designed uncertainty models, including competing fitted models and stress scenarios. {The learner observes successor states and their deterministic rewards from every candidate simulator, without requiring explicit knowledge of the transition probabilities or the reward function.} We establish almost sure convergence of the learned Q-function to the robust optimum, and derive a time-uniform high-probability error certificate that separates an observable empirical Bellman residual from transition-kernel estimation error. We also give quantitative finite-ambiguity approximation results for Wasserstein and parametric uncertainty models on finite state spaces. Numerical examples illustrate the effect of scenario selection and compare the tabular algorithm with robust value iteration.

\medskip
\noindent\textbf{Keywords:} Q-learning; Markov decision processes; model uncertainty; robust control; stochastic approximation
\end{abstract}

\section{Introduction}
Markov decision processes (MDPs) are a fundamental framework for sequential decision making under uncertainty, and, in particular, for reinforcement learning. An agent repeatedly observes the state of a stochastic environment, selects an action, and receives a reward, with the goal of maximizing cumulative discounted rewards over an infinite time horizon. Central to this formulation is the specification of the state transition probabilities, which determine how actions influence the evolution of the system.

In practice, these transition probabilities are rarely known and must be estimated from data (\citealp{aguirregabiria2002swapping}, \citealp{rust1994structural}, \citealp{srisuma2012semiparametric}) or implicitly learned through interaction with the environment (see, e.g., \citealp{ccalicsir2019model}, \citealp{tarbouriech2020active}). Policies trained under one transition law may perform poorly when deployed under a different law, even when the learning algorithm does not explicitly estimate a transition model. This sensitivity to changes in the environment motivates robustness against alternative dynamics. This issue is particularly acute in non-stationary environments or in applications where data are scarce and multiple competing models are plausible.
A natural way to address this challenge is to incorporate model uncertainty directly into the training phase. In distributionally robust Markov decision problems \citep{neufeld2023markov,iyengar2005robust,ElGhaouiNilim2005robust,wiesemann2013robust,mannor2016robust}, the transition kernel is assumed to belong to an ambiguity set of probability measures, and the agent optimizes performance under the worst-case model within this set. This paradigm reflects Knightian uncertainty \citep{knight1921risk} and provides policies that are robust against adverse but plausible deviations of the true dynamics.

Many robust reinforcement learning approaches specify ambiguity sets as divergence or metric balls, such as Kullback--Leibler (\citealp{bauerle2021distributionally_2}, \citealp{liu2022distributionally}) or Wasserstein balls \citep{neufeld2022robust}, centered at a reference transition kernel. These constructions provide a neighbourhood-based description of uncertainty around a reference law.

In many applications, however, uncertainty is better described by a finite collection of competing models. These may arise from different estimation procedures, stress scenarios, regime-specific dynamics, or expert judgments. In such settings, it is often desirable to protect explicitly against a small number of adverse scenarios rather than against all distributions within a metric neighborhood of a reference model. This motivates the study of robust control problems with finite, user-specified ambiguity sets.

To illustrate the systematic difference, consider an agent repeatedly betting on the outcome of a coin toss game, where the state records recent outcomes and the agent chooses whether to bet on an increase or decrease of the next realization. Suppose the true data-generating process is unknown but believed to belong to a small set of plausible models, such as a fair or a moderately biased coin. A non-robust agent trained under a single reference model may perform well when this model is correct, but can suffer systematic losses under alternatives. Wasserstein-robust approaches protect against all distributions in a neighbourhood of the reference model, including laws outside the convex hull of the specified candidates. Protecting against these additional laws can lead to overly conservative policies. In contrast, finite ambiguity sets allow the decision maker to specify explicitly the competing scenarios and to optimize against their worst case. As shown in Section~\ref{sec:examples}, the different ambiguity sets yield different policies and performance across the considered models; no policy is uniformly best across all test distributions.

In this paper, we develop a robust $Q$-learning framework for Markov decision problems in which each state–action pair is associated with a finite set of candidate transition kernels. We allow the transition kernel selected by nature to vary over time, as long as it remains within the prescribed ambiguity set, leading to a time-inhomogeneous worst-case formulation. This setting preserves the dynamic programming principle while enabling highly flexible modeling of uncertainty.

{The information setting is simulation-based: each candidate simulator returns a successor state together with its realized deterministic reward. This means in particular that the learner need not know the reward function or evaluate it at unobserved transitions.} This is appropriate, for example, when competing models have been fitted to historical data and are then used to optimize a policy. The guarantees below concern the resulting specified ambiguity set; they do not assert that this set contains the unknown data-generating law. Observations from a single unknown environment alone do not provide the required access to all candidate models.

\subsection{Our contribution}
Our main contributions can be summarized as follows.
\begin{enumerate}
\item[1.)] We study distributionally robust MDPs with finite ambiguity sets of transition probabilities. This established scenario-based formulation allows the decision maker to specify explicitly which alternative models are deemed plausible, without restricting the input to a ball around one reference measure.
\item[2.)] We propose a simulation-based robust $Q$-learning algorithm for this setting. {Given observed transition--reward pairs from each candidate simulator, the algorithm estimates the candidate laws and their mean rewards, selects an empirically worst-case model, and updates the $Q$-function using a fresh transition and its observed reward from that model.}
\item[3.)] We establish almost sure convergence to the robust optimal $Q$-function and a time-uniform high-probability error bound. {The bound separates an observable empirical Bellman residual from transition-kernel estimation error. When both the candidate kernels and their expected one-step rewards are known and used exactly, only the residual term remains.}
\item[4.)] We illustrate the approach through numerical examples.
\end{enumerate}

\subsection{Related work}
Distributionally robust optimization and robust Markov decision problems have been studied extensively; see, e.g., \citet{bauerle2021distributionally_2,bauerle2021q,ElGhaouiNilim2005robust,mannor2016robust,neufeld2023markov,ning2021double,panaganti2022sample,si2020distributional,si2020distributionally,uugurlu2018robust,wang2022policy,wiesemann2013robust,xu2010distributionally,yang2022towards,zhou2021finite}. Finite-scenario uncertainty and robust dynamic programming already appear in \citet{ElGhaouiNilim2005robust}; see also \citet{iyengar2005robust,wiesemann2013robust}. For a recent review of ambiguity modelling and rectangularity, see \citet{ou2025sequential}. This literature has focused on establishing dynamic programming principles under model uncertainty, characterizing robust value functions, and developing tractable solution methods for specific classes of ambiguity sets. A central modelling choice in these works is the specification of uncertainty via ambiguity sets of transition kernels, typically designed to capture deviations from a reference model in a principled manner.

Classical Q-learning for Markov decision problems was introduced by Watkins (\citealp{watkins1989learning,watkins1992q}), and has since become one of the cornerstone algorithms in reinforcement learning. Comprehensive introductions can be found in \citet{clifton2020q} and \citep{jang2019q}, and a wide range of applications has been studied, including control, finance, and operations research (see, e.g., \citealp{angiuli2022reinforcement}, \citealp{angiuli2021reinforcement}, \citealp{cao2021deep}, \citealp{charpentier2021reinforcement}, \citealp{huang2020deep}, \citealp{kolm2020modern}, and \citealp{naghibi2006application}). Classical model-free learning does not require a specified transition matrix, but its usual guarantees concern the environment supplying the training transitions. Robustness to alternative deployment dynamics is a separate question.

Robust Q-learning methods explicitly account for uncertainty in the transition dynamics. Several approaches consider ambiguity sets defined as divergence or metric balls, including Kullback--Leibler balls (\citealp{bauerle2021q}, \citealp{liu2022distributionally}, and \citealp{wang2023finite}) and Wasserstein balls \citep{neufeld2023markov} around a reference transition kernel. These formulations exploit convex duality and the structure of the underlying metric to obtain tractable algorithms and theoretical guarantees. Finite-sample guarantees under generative access have been studied in \citet{panaganti2022sample,wang2023finite}. More recent results include variance-reduced robust $Q$-learning for Kullback--Leibler ambiguity sets \citep{wang2024variance} and empirical model-based methods for $L_p$ ambiguity sets \citep{clavier2024minimax}. These methods obtain samples from a nominal transition law, whereas our input consists of separate simulators for finitely many candidate laws. Online interaction with a single training environment involves different exploration requirements; see \citep{he2025online}. Accordingly, sample-complexity conclusions under these different information settings are not directly comparable.

In contrast, many practical applications call for scenario-based uncertainty modelling, where ambiguity is naturally described by a finite collection of competing transition models arising from different estimators, stress scenarios, regimes, or expert assessments. Finite ambiguity sets are also particularly natural in financial applications, where uncertainty is often driven by a small number of stress scenarios or regime-dependent models rather than diffuse perturbations of a reference measure, see, e.g. \citep{foglia2018stress}, \citep{kapinos2018stress}, and \citep{schuermann2014stress}.

Robust dynamic programming is also closely related to time-consistent dynamic risk measures, compare \citep{bauerle2021distributionally_2}. In the cost convention, $\rho_{x,a}(Z):=\max_k\mathbb E_{\PP^{(k)}(x,a)}[Z]$ is a coherent one-step risk mapping; our lower expectation is $-\rho_{x,a}(-Z)$. Rectangular composition gives the corresponding time-consistent formulation. \citet{huang2021stochastic} develop risk-aware $Q$-learning using stochastic saddle-point risk estimation and establish almost sure convergence and finite-time bounds. Their algorithm learns from observed transitions and costs for a prescribed risk mapping. {Our algorithm instead estimates separately sampled candidate kernels and their mean rewards, and selects a worst-case candidate by finite minimization, requiring transition--reward observations from all candidate simulators but not prior knowledge of the reward function.} Recent work on risk-aware value-function approximation includes \citep{ruszczynski2025risk}.

The present work focuses on this particular simulation-based algorithm, its almost sure convergence, and an a posteriori error certificate combining an empirical Bellman residual with a confidence bound for kernel estimation. Finite-scenario uncertainty, the robust Bellman equation, and risk-aware learning are established concepts. Moreover, the Bellman operator is unchanged by taking the convex hull of each finite ambiguity set, as explained below. Thus, the contribution concerns the algorithm and its guarantees under the stated information assumptions.

\subsection{Structure}

The remainder of the paper is organized as follows. Section~\ref{sec:setting} introduces the robust Markov decision problem with finite ambiguity sets and establishes the associated dynamic programming formulation. In Section~\ref{sec:algorithm}, we present the robust $Q$-learning algorithm and its convergence and finite-sample error guarantees. Section~\ref{sec:examples} provides numerical experiments illustrating the computational cost and qualitative behavior of the proposed approach. Appendix~\ref{sec:extensions} contains the approximation results for Wasserstein and parametric ambiguity sets, together with their proofs. Appendix~\ref{sec:continuous_state} outlines an extension to continuous state spaces using function approximation and deep $Q$-learning. Appendix~\ref{sec:proofs} collects the proofs and auxiliary results for the main text.

\section{Setting and specification of the problem} \label{sec:setting}
In this section we introduce the underlying framework which will be employed to establish a robust $Q$-learning
algorithm for finite ambiguity sets. 

\subsection{Setting} \label{subsec:setting}
Optimal control problems, such as Markov decision problems, are formulated using a state space comprising all possible states accessible by an underlying stochastic process. We model this state space by a  $d$-dimensional finite Euclidean subset 
$\X \subset \mathbb{R}^d$ and eventually aim at solving a robust control problem over an infinite time horizon. Hence, we define the state space over the entire time horizon via\footnote{Throughout the paper we use $\mathbb N=\{1,2,\ldots\}$ and $\mathbb N_0=\{0,1,\ldots\}$.}
$$\Omega :=  \X^{\mathbb{N}_0} = \X \times \X \times ... $$
equipped with the  corresponding $\sigma-$algebra
$\mathcal{F}:=2^\X\otimes 2^\X\otimes \cdots$. Next, let $(X_t)_{t\in\mathbb{N}_0}$ be the state process, i.e., the stochastic process on $\Omega$ describing the states attained over time.
We denote the finite set of possible actions  by $A \subset \mathbb{R}^m$, where $m\in\mathbb{N}$ represents the dimension of the action space. The set of admissible policies is then given by
\begin{align*}
    \mathcal{A} := &\{\textbf{a} = (a_t)_{t\in\mathbb{N}_0}\ |\ a_t:\Omega\rightarrow A\text{ for every }t\in\mathbb N_0;\ a_t \text{ is }\ \sigma(X_t)-\text{measurable}\ \text{for all } t \in \mathbb{N}_0\}\\
    = & \{(a_t(X_t))_{t\in\mathbb{N}_0}\ |\ a_t: \X \rightarrow A\ \text{Borel\ measurable } \text{for all } t \in \mathbb{N}_0\}.
\end{align*}
In contrast to classical non-robust Markov decision problems, we work under the paradigm that the precise state transition probability is not known but instead, to account for model uncertainty, contained in an ambiguity set of finitely many transition probabilities. In the following, let $\mathcal{M}_1(\X)$ denote the set of probability measures on $(\X,2^\X)$, {and let $\tau_0$ denote the topology of weak convergence\footnote{
For any $\mu\in \mathcal{M}_1(X)$, $(\mu_n)_{n\in\mathbb{N}}\subseteq \mathcal{M}_1(X)$ and $C_b(X,\mathbb{R})$ denoting the space of continuous and bounded functions mapping from the space $X$ to $\R$, we have:
\begin{equation*}
    \underset{n \rightarrow \infty}{\mu_n\longrightarrow^{\tau_0}\mu} \Leftrightarrow \lim\limits_{n \to \infty}\int gd\mu_n = \int gd\mu,\ \text{ for all } g \in C_b(X,\mathbb{R}).
\end{equation*}}.}
An ambiguity set of $N\in \N$ probability measures in dependence of a state-action pair is then modeled by a set-valued map given by 
\begin{equation}\label{eq:ambiguityset}
\begin{aligned}
    \X \times A & \rightarrow 2^{\mathcal{M}_1(\X)}\\
    (x,a) & \twoheadrightarrow \mathcal{P}(x,a):=\left\{\mathbb{P}^{(1)}(x,a),\cdots ,\mathbb{P}^{(N)}(x,a)\right\},
\end{aligned}
\end{equation}
where for all $k \in \{1,\dots,N\}$ and for all $(x,a) \in \X \times A$ we have that $\PP^{(k)}(x,a)$ is a probability measure, i.e., $\PP^{(k)}(x,a)\in \mathcal{M}_1(\X)$.
Note that in the degenerate case where $\mathcal{P}(x, a)$ contains only a single probability distribution, we are facing a classical Markov decision process (i.e.\ a non-robust one), compare, e.g., \citep{bauerle2011markov}.

The ambiguity set of admissible probability distributions on $\Omega$ depends on the initial state $x\in \X$ and the policy $\textbf{a}\in\mathcal{A}$. Let $\delta_x \in \mathcal{M}_1(\X)$ denote the Dirac measure at point $x\in \X$. Then, we define for every state-action pair $(x,\textbf{a})\in \X\times\mathcal{A}$ the underlying set of admissible probability measures of the stochastic process $(X_t)_{t\in\mathbb{N}_0}$ by
\begin{equation}
\begin{aligned}\label{probaspace}
\mathfrak{P}_{x,\textbf{a}} := \Bigl\{ \delta_{x} \otimes \mathbb{P}_0 \otimes \mathbb{P}_1 \otimes ... ~|~&\text{for all } t \in \mathbb{N}_0 : \mathbb{P}_t : \X  \rightarrow \mathcal{M}_1(\X)\ \text{Borel-measurable},\\ 
&\text{and } \mathbb{P}_t(x_t) \in \mathcal{P}(x_t,a_t(x_t)),\ \text{ for all } x_t \in \X \Bigl\},
\end{aligned}
\end{equation}
where $\mathbb{P}=\delta_x\otimes\mathbb{P}_0\otimes\mathbb{P}_1\otimes\cdots$ denotes the probability law obtained by successively applying these transition kernels. Equivalently, we have
\[
\mathbb P(X_0=x_0,\ldots,X_n=x_n)
=\delta_x(\{x_0\})\prod_{t=0}^{n-1}\mathbb P_t(x_t)(\{x_{t+1}\}).
\]
The construction of ambiguity sets $\mathfrak{P}_{x,\textbf{a}}$ provided in \eqref{eq:ambiguityset} and \eqref{probaspace} is in the related literature often referred to as a \emph{$(s,a)$-rectangular ambiguity set}, (compare \citealp{iyengar2005robust} and \citealp{wiesemann2013robust}) and enables to obtain a dynamic programming principle of the associated Markov decision problem which allows to derive tractable numerical solution methods, compare Section~\ref{subsec:defop}.

\subsection{Finite ambiguity sets and other ambiguity sets}
The ambiguity sets considered here have a different input representation from those used, e.g., in \citet{bauerle2021q}, \citep{liu2022distributionally}, and \citep{neufeld2022robust} where the corresponding ambiguity sets of probability measures are defined as balls around some reference measure. 

The $Q$-learning approaches presented in \citet{bauerle2021q}, \citep{liu2022distributionally}, and \citep{neufeld2022robust} provide solutions to account for a potential misspecification of a reference measure, by allowing deviations from it where the size of the deviations is measured by the respective distances (Kullback--Leibler divergence and Wasserstein distance). However, these approaches do not prescribe a finite list of candidate distributions but instead consider all measures that are in a certain sense close to the reference measure. If a decision maker is instead interested in controlling the distributions contained in the ambiguity set, for example, by allowing only for a specific type of parametric distributions with a finite set of possible parameters or by selecting particular stress scenarios, the decision maker can use the $Q$-learning approach presented in this paper. Another relevant situation is the consideration of multiple estimated models / probability measures obtained for example via different estimation methods. Wishing to take into account all of these estimations, a decision maker can construct an ambiguity set containing all estimated probability measures.

 Beyond these use cases, our approach also allows to combine different types of distributions, if desired.

The main difference therefore is that a decision maker using the approach presented in this paper can precisely control what types of distributions are deemed possible whereas the approaches from  \citep{bauerle2021q}, \citep{liu2022distributionally}, and \citep{neufeld2022robust} account for a misspecification of one single reference measure.

\begin{rem}[Convex hull and rectangularity]\label{rem:convex_hull}
For every function $g:\X\to\R$ and every $(x,a) \in \X \times A$,
\[
\min_{\PP\in\mathcal P(x,a)}\mathbb E_\PP[g]
=\min_{\PP\in\operatorname{co}\mathcal P(x,a)}\mathbb E_\PP[g],
\]
because the expectation under a convex mixture is a convex combination of the candidate expectations. Hence taking convex hulls $\operatorname{co}\mathcal P(x,a)$ leaves the robust Bellman operator, value function, and optimal controller policies unchanged. 

This shows that the  finite representation from \eqref{eq:ambiguityset} is useful for model specification and sampling, but does not provide additional generality merely through nonconvexity.
\end{rem}

Although our primary focus is on finite ambiguity sets specified directly by the decision maker, the framework also permits the approximation of more general ambiguity models. In Appendix~\ref{sec:extensions}, we show that Wasserstein ambiguity sets and compact parametric families on finite state spaces can be approximated by finite ambiguity sets, and establish corresponding quantitative approximation bounds. The computational cost of these approximations may, however, become substantial.

\subsection{Specification of the optimization problem} \label{subsec:specopt}

After execution of an action $a_t(X_t)$, the agent receives a feedback on the quality of the chosen action in terms of a reward $r(X_t,a_t(X_t),X_{t+1})$ where $r$ denotes some reward function $r:\X\times A\times \X \rightarrow \mathbb{R}$. The \textit{robust} optimization problem consists then, for every initial state $x \in \X$, in maximizing the expected value of $\sum_{t=0}^{\infty}\alpha^t r(X_t, a_t(X_t), X_{t+1})$ under the worst case measure from $\mathfrak{P}_{x,\textbf{a}}$ over all possible policies $\textbf{a} \in \mathcal{A}$, where $\alpha\in (0,1)$ is some discount factor accounting for time preferences of rewards. Throughout, write $C_r:=\max_{x,a,y}|r(x,a,y)|<\infty$. The function $r$ is deterministic and common to all candidate models, but is not assumed known to the learner. We only assume that its values are observed at sampled transitions.
Therefore, the value function
\begin{equation} \label{eq:valuefunc}
 \X \ni x \mapsto V(x) := \sup_{\mathbf{a} \in \mathcal{A}}\inf_{\mathbb{P}\in \mathfrak{P}_{x,\textbf{a}}} \mathbb{E}_{\mathbb{P}}\left[ \sum^{\infty}_{t=0}\alpha^t \cdot r(X_t,a_t(X_t),X_{t+1})\right]
\end{equation}
describes the expected value of $\sum_{t=0}^{\infty}\alpha^t r(X_t, a_t(X_t), X_{t+1})$ under the worst case measure from $\mathfrak{P}_{x,\textbf{a}}$ and when executing the optimal policy $\textbf{a} \in \mathcal{A}$ after having started in initial state $x\in \X$.

\subsection{Dynamic programming} \label{subsec:defop}
To solve \eqref{eq:valuefunc} directly is typically a non-feasible optimization problem as it means to find directly (infinite-dimensional) solutions over the whole time horizon. However, due to its time-homogeneous structure, solving the infinite time horizon problem can be simplified significantly to a one time step problem which turns out to be tractable. To this end, we consider the one time step optimization problem
\begin{equation}\label{eq:TVdefinition}
\X \ni x \mapsto    \mathcal{T}V(x) := \max_{a \in A}\min_{\mathbb{P}\in\mathcal{P}(x,a)}\mathbb{E}_{\mathbb{P}}\left[ r(x,a,X_1) + \alpha V(X_1)\right],
\end{equation}
where $\X \ni x  \mapsto V(x)$ is the value function defined in \eqref{eq:valuefunc}. By using the above definitions one can show (see Proposition~\ref{prop:DPP}) that the dynamic programming equation $\mathcal{T}V=V$ holds, implying that solving the fixed point equation $\mathcal{T}V=V$ is sufficient to compute the optimal value function. This fixed-point equation is also famously known in its non-robust formulation as \emph{Bellman equation}, compare, e.g. \citep[p.164]{dixit1990optimization} for a non-robust version.

To determine optimal actions, analogous to the definition of the non-robust $Q$-value function (see e.g.\,\citealp{watkins1989learning}), we further define the robust optimal $Q$-value function by
\begin{equation}\label{eq:Qvalue}
    \X \times A \ni (x,a)  \mapsto Q^*(x,a) := \min_{\mathbb{P}\in \mathcal{P}(x,a)} \mathbb{E}_{\mathbb{P}}\left[ r(x,a,X_1) + \alpha V(X_1)\right],
\end{equation}
allowing us to interpret $Q^*(x,a)$ as the \emph{quality} of executing action $a$ when in state $x$ (under the worst case measure) with the best possible action leading to $V(x)$ explaining the notion of quality. 
\begin{prop}\label{prop:DPP}
    Assume $0 < \alpha < 1$. Then, for all $x\in \X$ we have
    $$\max_{a\in A}Q^*(x,a)=\mathcal{T}V(x)=V(x).$$
\end{prop}

Using Proposition~\ref{prop:DPP}, the goal of our $Q$-learning algorithm presented in the next section will be to algorithmically determine $Q^*$ which in turn allows to infer optimal actions in every given state.

\section{Robust \textit{Q}-Learning Algorithm}\label{sec:algorithm}

In this section, we introduce Algorithm~\ref{alg:Q-L}, a simulation-based distributionally robust $Q$-learning algorithm. 
We establish sufficient conditions for its convergence in Theorem~\ref{thmcv} and derive a finite-sample error bound in Subsection~\ref{sec:rates}.

\subsection{The robust Q-Learning algorithm and its convergence}

To apply our proposed $Q$-learning algorithm (Algorithm~\ref{alg:Q-L}) we do not require explicit knowledge of the candidate transition kernels in $\mathcal P$ or of the reward function $r$. Instead, we only require \emph{generative access} to transition--reward pairs, i.e., we require that an applicant is able to simulate from each model and then observe the rewards. Formally, at each queried $(x,a)\in \X \times A$, one observes that candidate model $k\in \{1,\dots,N\}$ returns
\[
(Y^{(k)},\mathsf R^{(k)}), \qquad \text{ where }
Y^{(k)}\sim\PP^{(k)}(x,a),~
\mathsf R^{(k)}=r(x,a,Y^{(k)}).
\]
The reward $\mathsf R^{(k)}$ is observed with its sampled successor $Y^{(k)}$, evaluating $r$ at other triples is not required.

At each iteration, Algorithm~\ref{alg:Q-L} selects a state–action pair and uses past observations to identify the candidate with the smallest estimated reward–continuation value. It then obtains a fresh transition–reward observation from each candidate. All observations update the empirical model estimates, while the observation from the selected candidate is used for the Q-update.

$Q$-learning is off-policy learning, i.e., it learns the optimal policy from actions chosen by another policy (that we call for distinction \emph{behavior rule}). Thus, before running the algorithm, an applicant chooses a \emph{behavior rule} $\mathbf{a}=(a_t)_{t\in \N_0}$ for selecting the executed actions $a_t(X_t)$ when applying the algorithm. This behavior rule is not necessarily contained in $\mathcal{A}$, as we allow it to depend on the learning history and use independent randomization. In this way, we allow in particular the behavior rule to be an $\varepsilon$-greedy rule (see Remark~\ref{rem:implement}) which chooses a maximizing action for $Q_t(X_t,\cdot)$, as defined in \eqref{eq:updaterule}, with probability $1-\varepsilon$ and otherwise selects uniformly from $A$. Such a rule has full action support when $\varepsilon>0$, however state coverage must also be ensured, as discussed in Remark~\ref{rem:support}.

We count the visits in state--action pairs $(x,a)\in  \X \times A$  before the update at time $t$ via\footnote{We set an empty sum equal to zero.} $N_t(x,a):=\sum_{s=0}^{t-1}\one_{\{(X_s,a_s(X_s))=(x,a)\}}$. We also write $\mathsf R_{s+1}^{(k)}=r(X_s,a_s(X_s),X_{s+1}^{(k)})$ for the reward observed from candidate model $k$ at iteration $s$. For a positive visit count in a state--action pair $(x,a)$, we define its \emph{empirical mean reward} by
\begin{equation}\label{eq:empirical_reward_identity}
\begin{aligned}
\bar r_t^{(k)}(x,a)
:=\E_{\PP_t^{(k)}(x,a)}[r(x,a,X_1)]&=\sum_{y\in\X}\mathbb P_t^{(k)}(x,a)(y)\,r(x,a,y)\\
&=\frac{\sum_{s=0}^{t-1}\one_{\{(X_s,a_s(X_s))=(x,a)\}}\mathsf R_{s+1}^{(k)}}{N_t(x,a)},
\end{aligned}
\end{equation}
where $\PP_t^{(k)}$ denotes the empirical measure defined in \eqref{eq:sample_pt}, and where the last equality in \eqref{eq:empirical_reward_identity} follows directly from the definition of the empirical measure and the deterministic reward assumption. Thus under $\PP_t^{(k)}$, the expected reward can be computed exactly from observed rewards alone and doesn't require an additional estimation.

 For an unvisited pair, set $\bar r_t^{(k)}(x,a)=0$ and initialize all candidate laws at that pair with the same probability measure $\nu(x,a)$, for example the uniform law on $\X$. After a visit, \eqref{eq:empirical_reward_identity} makes the candidate scores identical to those obtained by evaluating the known reward function under the empirical kernels.

\begin{algorithm}[h!]
\caption{Robust $Q$-learning for finite ambiguity sets}\label{alg:Q-L}
\begin{algorithmic}

\STATE \textbf{Input:} 
\quad State space $\mathcal{X} \subseteq \mathbb{R}^d$; \quad Initial state $x_0 \in \mathcal{X}$; \quad Action space $A \subseteq \mathbb{R}^m$; {\quad Transition--reward simulators;} \quad Discount factor $\alpha \in (0,1)$; \quad Ambiguity set $\mathcal{P}(x,a)=\left\{\mathbb{P}^{(1)}(x,a),\ldots,\mathbb{P}^{(N)}(x,a)\right\}$;
\quad Policy $(a_t)_t$; \quad Sequence of learning rates $(\widetilde{\gamma_t})_{t \in \mathbb{N}_0} \subseteq [0,1]$;

\vspace{0.2cm}
\STATE Initialize $Q_0(x,a)$ for all $(x,a) \in \mathcal{X} \times A$ to an arbitrary real value;
\STATE Set $X_0 = x_0$;
\STATE {Initialize $\mathbb P_0^{(k)}(x,a)=\nu(x,a)$ with a common probability measure for all $k$, for example uniformly on $\X$;}
\STATE {Set $\bar r_0^{(k)}(x,a)=0$ for all $k$ and $(x,a)\in\X\times A$;}
\STATE Initialize $N_0(x,a) =0$ for all $(x,a) \in \X \times A$;
\FOR{$t=0,1,\ldots$}
    \STATE Select the action $a_t(X_t)$ using the policy;

    \STATE {Determine $k_t^*\in\{1,\ldots,N\}$ from past observations, breaking ties by the fixed rule:}
    \begin{equation}\label{eq_definition_k_t}
    k_t^*\in\argmin_{k=1,\ldots,N}\left\{
    \bar r_t^{(k)}(X_t,a_t(X_t))+
    \alpha\sum_{y\in\X}\mathbb P_t^{(k)}(X_t,a_t(X_t))(y)\max_{b\in A}Q_t(y,b)
    \right\};
    \end{equation}
      \STATE {Sample $X_{t+1}^{(k)}\sim\mathbb P^{(k)}(X_t,a_t(X_t))$ independently and observe $\mathsf R_{t+1}^{(k)}=r(X_t,a_t(X_t),X_{t+1}^{(k)})$ for $k=1,\ldots,N$;}
       \STATE {Set $X_{t+1}:=X_{t+1}^{(k_t^*)}$ and $\mathsf R_{t+1}:=\mathsf R_{t+1}^{(k_t^*)}$;}
     \STATE For pairs with positive denominator, set
    \begin{equation}   \label{eq:sample_pt}
    \mathbb{P}_{t+1}^{(k)}(x, a)(y) := \frac{\sum_{s=0}^{t}\one_{\{(X_s,a_s(X_s),X_{s+1}^{(k)})=(x,a,y)\}}}{\sum_{s=0}^{t}\one_{\{(X_s,a_s(X_s))=(x,a)\}}}     \text{ for all }k\in \{1,\dots,N\},~~ (x,a,y) \in \mathcal{X}\times A \times \X ;
    \end{equation}
    \STATE For an unvisited $(x,a)$, retain $\mathbb P_{t+1}^{(k)}(x,a)=\mathbb P_t^{(k)}(x,a)$ instead of dividing by zero;
    \STATE {For each $k$, update the empirical reward mean at $(x,a)=(X_t,a_t(X_t))$ by}
    \begin{equation}\label{eq:reward_update}   
    \bar r_{t+1}^{(k)}(x,a)
    :=\frac{N_t(x,a)\bar r_t^{(k)}(x,a)+\mathsf R_{t+1}^{(k)}}{N_t(x,a)+1},
    \end{equation}
    leaving the other empirical reward means unchanged;

    \FORALL{$(x,a) \in \mathcal{X} \times A$}
        \STATE Update $Q_t$ via
        {\small
        \begin{align}\label{eq:updaterule}
        Q_{t+1}(x,a) \leftarrow 
        \begin{cases}
        \begin{aligned}
        &(1-\widetilde{\gamma}_{N_t(x,a)})Q_t(x,a)\\[-1mm]
        &\quad+\widetilde{\gamma}_{N_t(x,a)}\left({\mathsf R_{t+1}}+\alpha\max_{b\in A}Q_t(X_{t+1},b)\right)
        \end{aligned} & \text{if }(x,a)=(X_t,a_t(X_t)), \\
        Q_t(x,a); & \text{else};
        \end{cases}
        \end{align}}
        and update the number of visit times $N_t$ via
        \begin{align*}
    N_{t+1}(x,a)\leftarrow \left\{
        \begin{array}{ll}
            N_{t}(x,a) + 1 \ \ &\text{if}\ (x,a)=\left(X_t,~a_t(X_t)\right), \\
            N_{t}(x,a)  \ \ &\text{else};
        \end{array}
    \right.
    \end{align*}
    \ENDFOR
\ENDFOR

\STATE \textbf{Output:} A sequence $(Q_t(x,a))_{t \in \mathbb{N}_0, x \in \mathcal{X}, a \in A}$;

\end{algorithmic}
\end{algorithm}

Given an initial state $x_0$ and a behavior rule $\mathbf a$, we write $\mathbb P_{x_0,\mathbf a}$ for the law of the complete learning experiment on an auxiliary probability space, including all simulator draws and action randomizations. Let $\mathcal G_t$ contain all observations and action randomizations before iteration $t$, together with the randomization for the current action. Then
\begin{equation}\label{eq_definition_P_measure}
\mathbb P_{x_0,\mathbf a}(X_{t+1}=y\mid\mathcal G_t)
=\mathbb P^{(k_t^*)}(X_t,a_t(X_t))(y).
\end{equation}
The index $k_t^*$ is chosen before the fresh samples are drawn. Our main result establishes that the sequence $(Q_t)_{t\geq0}$ produced by Algorithm~\ref{alg:Q-L} converges pointwise, $\mathbb P_{x_0,\mathbf a}$-almost surely, to the optimal robust $Q$-value function $Q^*$ defined in \eqref{eq:Qvalue}.

\begin{thm}\label{thmcv}
Assume the initialization specified in Algorithm~\ref{alg:Q-L}. Let $(\widetilde{\gamma_t})_{t \in \mathbb{N}_0} \subseteq [0,1]$, and define for all $(x,a) \in \X \times A  $ and for all $t\in \N_0$
 \begin{equation}
\gamma_t(x,a,X_t) := \one_{\left\{(x,a)=\left(X_t,~a_t(X_t)\right) \right\}}\cdot \widetilde{\gamma}_{N_t(x,a)}.
\end{equation}
where
\[
N_t(x,a):=\sum_{s=0}^{t-1}\mathbf{1}_{\{(X_s,a_s(X_s))=(x,a)\}}
\]
is the number of visits to $(x,a)\in \X \times A$ before time $t$.  Moreover, let $0 < \alpha < 1$ and let $x_0\in\X$ and a behavior rule $\mathbf a$ be such that
    \begin{equation}\label{eq:gammat}
    \sum^{\infty}_{t=0}\gamma_t(x,a,X_t) = \infty,\ \sum^{\infty}_{t=0}\gamma_t^2(x,a,X_t) < \infty,\ \text{ for all } (x,a)\in \X \times A,\ \mathbb{P}_{x_0,\textbf{a}}\ -\text{almost\ surely}.
    \end{equation}
 Then, we have $\text{ for all } (x,a) \in \X\times A$ that
    \begin{equation}
        \lim_{t\rightarrow\infty}Q_t(x,a) = Q^*(x,a)\quad  \mathbb{P}_{x_0,\textbf{a}}\ -\text{almost\ surely}.
    \end{equation}
\end{thm}

The proof, given in Appendix~\ref{subsec:proofs}, combines consistency of the empirical candidate kernels with a stochastic approximation argument (that is used in a slightly simpler form for the proof of classical Q-learning). Since the minimizing candidate model is selected before drawing the fresh transition, the conditional mean of the update can be compared with the robust Bellman operator. The discrepancy caused by estimating the candidate kernels vanishes almost surely. {Identity~\eqref{eq:empirical_reward_identity} ensures that observing deterministic rewards instead of knowing $r$ introduces no additional estimation term.}

\begin{rem}
{If both the candidate kernels $\PP^{(k)}(x,a)$ and their expected one-step rewards are explicitly known (in particular, if the kernels and the reward function $r$ are known), then \eqref{eq_definition_k_t} can be replaced by}
    \begin{equation}\label{eq_definition_k_t_explicit}
    k_t^* \in  \argmin_{k = 1,\ldots, N} \left\{ \sum_{y \in \X} \left(r(X_t, a_t(X_t),y) + \alpha \max_{b \in A} Q_t(y,b)\right)\mathbb{P}^{(k)}(X_t, a_t(X_t))(y)\right\};
    \end{equation}
    {and there is no need to estimate the kernels or their mean rewards, or to sample the unselected models. Under this additional information, robust value iteration (\cite{iyengar2005robust}) is also a natural baseline.}
\end{rem}

\begin{rem}[On the implementation]~\label{rem:implement}
\begin{itemize}
\item[(a)]
The idea behind  replacing $\widetilde{\gamma}_t$ with $\widetilde{\gamma}_{N_t(x,a)}$ is to account for the exploration-exploitation trade off by reducing the learning rate the more often we have visited a state and hence the more confident the algorithm is w.r.t.\, choosing the correct optimal action, and conversely to apply a larger learning rate if the state-action pair was visited less often until time $t$. Visit-dependent learning rates are common in reinforcement learning; compare, e.g., \citep{martin2017count}. A choice of $\widetilde{\gamma}$ that satisfies \eqref{eq:gammat} is given by the family $\widetilde{\gamma}_t = (n+1)^{-\omega}$ for $\frac{1}{2}<\omega \leq 1$.
\item[(b)]
Algorithm~\ref{alg:Q-L} requires a specified exploration rule, which may adapt to the learned $Q$-function. Implied from usual practice in non-robust $Q$-learning (\citealp[Chapter~9]{dixon2020machine}, \citealp{mnih2015human}, or \citealp{tokic2011value}) and robust $Q$-learning \citep{neufeld2022robust}, a reasonable choice seems to be the $ \varepsilon$-greedy policy defined by
\[
\mathcal{X} \ni x \mapsto a_t(x) :=\begin{cases}
\operatorname{argmax}_{b \in A} Q_t(x,b)&\text{with probability } 1- \epsilon_{\operatorname{greedy}}, \\
a \sim \mathcal{U}(A) &\text{with probability } \epsilon_{\operatorname{greedy}} ,
\end{cases}
\]
for some $\epsilon_{\operatorname{greedy}}\in(0,1]$, with ties broken by a fixed rule, and
where $a \sim \mathcal{U}(A)$ means that a random action $a$ is  chosen uniformly at random from the finite set $A$.
\end{itemize}

\end{rem}
\begin{rem}[Exploration and the number of visits]\label{rem:support}
A policy with full support in $A$ ensures exploration of actions conditional on the current state; visiting every state additionally requires a reachability condition. One sufficient condition is the existence of constants $p_0,a_0>0$ such that
\[
\mathbb P^{(k)}(x,a)(y)\geq p_0\quad\text{for all }x,y,a,k,
\qquad
\mathbb P(a_t(X_t)=a\mid\text{past observations},X_t)\geq a_0
\quad\text{for all }a,t.
\]
The second condition holds for an $\varepsilon$-greedy policy (Remark~3.3(b)) with $\varepsilon\geq\varepsilon_0>0$. For each fixed $T\geq2$ and $\delta\in(0,1)$, an application of the Azuma--Hoeffding inequality (Lemma~\ref{lem:azu}) to the centered state-action visit indicators, followed by a union bound, gives with probability at least $1-\delta$ that
\[
\min_{x,a}N_T(x,a) \geq \frac{p_0\varepsilon_0}{|A|}(T-1)-\sqrt{2(T-1)\log\bigl(|\X||A|/\delta\bigr)}.
\]
Indeed, at each time $t\geq1$ the conditional probability of visiting any prescribed pair is at least $p_0\varepsilon_0/|A|$, and the centered indicators are martingale differences bounded in absolute value by one.
\end{rem}

\begin{rem}[Independent generative queries]\label{rem:generative_queries}
The same update can also be implemented without following a single trajectory. At iteration $t$, draw a state-action pair $(Z_t,B_t)$ uniformly from $\X\times A$, independently of the past. {Replace $(X_t,a_t(X_t))$ in Algorithm~\ref{alg:Q-L} by $(Z_t,B_t)$ and select $k_t^*$ from the empirical kernels and reward means before drawing fresh independent successors $Y_t^{(k)}\sim\mathbb P^{(k)}(Z_t,B_t)$ with their observed rewards $r(Z_t,B_t,Y_t^{(k)})$. Use the selected successor and its observed reward in the $Q$-update. All $N$ transition--reward observations update the empirical kernels and reward means, as before.} The next queried pair is sampled independently; it is not set equal to the selected successor. Here
\[
N_T(x,a):=\sum_{t=0}^{T-1}\one_{\{(Z_t,B_t)=(x,a)\}}
\]
counts queries rather than trajectory visits. This implementation uses precisely the generative access to every state-action pair already assumed above, and is useful when the candidate kernels do not ensure exploration of all states along a trajectory.

The conclusions of Theorems~\ref{thmcv} and~\ref{thm:finite_time_rate_alg1} remain valid with these substitutions.
\end{rem}

\begin{rem}[Computational cost]\label{rem:cost}
{Using stored transition counts, empirical reward sums, and a cached vector $\max_{b\in A} Q_t(y,b)$,} one iteration uses $N$ simulator calls and $O(N|\X|+|A|)$ arithmetic operations for model selection and updating the visited pair. The transition tables require $O(N|\X|^2|A|)$ storage. {The reward sums require only $O(N|\X||A|)$ additional storage and $O(N)$ operations per update.} Evaluating the full empirical Bellman residual costs $O(N|\X|^2|A|)$ and can be done periodically. Thus enumeration is practical for small candidate sets, but its cost grows with the number of models and states. {For explicitly known kernels and expected rewards, a robust value-iteration sweep also costs $O(N|\X|^2|A|)$ and requires no simulator calls.}
\end{rem}

\subsection{Finite-sample error bounds}\label{sec:rates}
In addition to the almost sure convergence guarantee from Theorem~\ref{thmcv}, we provide a non-asymptotic error bound that separates the error in solving the empirical Bellman equation from the error in estimating the candidate kernels. The first term is an observable residual; we do not assert a universal iteration rate for the asynchronous recursion under the assumptions of Theorem~\ref{thmcv} alone.

For $T\geq0$, define the observable empirical Bellman operator $\widehat{\mathcal H}_T$ and its residual $R_T$ by
\begin{equation}\label{eq:observable_empirical_bellman}
\begin{aligned}
\widehat{\mathcal H}_Tq(x,a)
&:=\min_{k=1,\ldots,N}\left\{
\bar r_T^{(k)}(x,a)+\alpha\sum_{y\in\X}\mathbb P_T^{(k)}(x,a)(y)\max_{b\in A}q(y,b)
\right\},\\
R_T&:=\|Q_T-\widehat{\mathcal H}_TQ_T\|_\infty.
\end{aligned}
\end{equation}
This operator is an $\alpha$-contraction. If $N_T(x,a)>0$, identity~\eqref{eq:empirical_reward_identity} gives
\[
\widehat{\mathcal H}_Tq(x,a)
=\min_{k=1,\dots,N}\sum_{y\in\X}\mathbb P_T^{(k)}(x,a)(y)
\left(r(x,a,y)+\alpha\max_{b\in A}q(y,b)\right).
\]
At unvisited pairs, the zero reward initialization is used. Thus $R_T$ is computable at every iteration without evaluating $r$ at unobserved transitions.
\begin{thm}\label{thm:finite_time_rate_alg1}
Let $|\X|\geq2$, $d=|\X||A|$, $C_r=\max_{(x,a)\in\X\times A,\ y\in\X}|r(x,a,y)|$, and let $N_T(x,a)$ denote the number of visits to $(x,a)\in \X \times A$ before time $T$, as in Algorithm~\ref{alg:Q-L}. For $\delta\in(0,1)$, set
\[
b_\delta(0):=2,\qquad
b_\delta(n):=\min\left\{2,\sqrt{\frac{2}{n}\log\left(\frac{\pi^2Nd(2^{|\X|}-2)n^2}{6\delta}\right)}\right\},\quad n\geq1,
\]
and $e_T(\delta):=\max_{(x,a)\in\X\times A}b_\delta(N_T(x,a))$.
Then, with probability at least $1-\delta$, simultaneously for all $T\geq0$,
\begin{equation}\label{eq:finite_sample_certificate}
\|Q_T-Q^*\|_\infty
\leq\frac{R_T}{1-\alpha}+\frac{C_r}{(1-\alpha)^2}e_T(\delta).
\end{equation}
In particular, if $T\geq1$ and $N_T^*:=\min_{x,a}N_T(x,a)\geq1$, then on the same event
\begin{equation}\label{eq:finite_sample_visit_bound}
\|Q_T-Q^*\|_\infty
\leq\frac{R_T}{1-\alpha}
+\frac{C_r}{(1-\alpha)^2}
\min\left\{2,\sqrt{\frac{2}{N_T^*}\log\left(\frac{\pi^2Nd(2^{|\X|}-2)T^2}{6\delta}\right)}\right\}.
\end{equation}
{If $|\X|=1$, \eqref{eq:finite_sample_certificate} holds with $e_T(\delta)=2\one_{\{N_T^*=0\}}$; after every pair has been sampled, its deterministic reward is known exactly and this term vanishes. For any $|\X|$, if both the candidate kernels and their expected rewards are known and used exactly in the Bellman operator, the certificate holds with $e_T(\delta)=0$ and the corresponding exact Bellman residual.}
\end{thm}

The residual $R_T=\|Q_T-\widehat{\mathcal H}_TQ_T\|_\infty$ in Theorem~\ref{thm:finite_time_rate_alg1} is computable from the learned $Q$-function, empirical kernels, and observed reward means. To evaluate the numerical bound, one additionally needs a known finite upper bound $\overline C_r\geq C_r$, which may replace $C_r$ throughout \eqref{eq:finite_sample_certificate} and \eqref{eq:finite_sample_visit_bound}. Neither the algorithm nor its almost sure convergence requires knowledge of this bound. The residual measures the remaining error in solving the empirical control problem; the second term controls estimation of the deterministic reward--continuation expectation through the empirical transition law. For an unvisited pair, $b_\delta(0)=2$ also covers the zero reward initialization. No separate reward-estimation confidence term is needed because of \eqref{eq:empirical_reward_identity}.

Under the assumptions of Theorem~\ref{thmcv}, every state–action pair is sampled infinitely often, and \(e_T(\delta)\to0\) almost surely. Moreover, Theorem~\ref{thmcv} and consistency of the empirical Bellman operator imply \(R_T\to0\) almost surely. Hence the right-hand side of \eqref{eq:finite_sample_certificate} vanishes asymptotically. The bound separates the observable empirical Bellman residual from the statistical error in estimating the candidate models.

\section{Numerical Experiments}\label{sec:examples}
In this section we provide three numerical examples illustrating Algorithm~\ref{alg:Q-L} and comparing finite-scenario policies with other uncertainty models, non-robust policies, and empirical robust planning. The examples also illustrate the differences between these approaches. We use a discount factor $\alpha=0.95$ and, unless stated otherwise, the learning rate $(1+n)^{-0.6}$ at the $(n+1)$st update of a state--action pair. The coin toss experiments use trajectory sampling, whereas the portfolio and three-state examples use generative queries as described in Remark~\ref{rem:generative_queries}. The three-state example also considers the harmonic learning rate. Further details of the implementation can be found under \href{https://github.com/juliansester/FiniteQLearning}{https://github.com/juliansester/FiniteQLearning}. 

\subsection{Coin toss}\label{subsec:CoinToss}
Inspired by \citet[Example 4.1]{neufeld2023markov}, to illustrate the differences between different approaches to model ambiguity, we consider an agent playing a simple coin toss game.

At each time step $t \in \mathbb{N}_0$ the agent observes the result of $10$ coins where an outcome of heads corresponds to $1$, and tails corresponds to  $0$. The state $X_t$ at time $t \in \mathbb{N}_0$ is then given by the sum of the $10$ coins value, i.e., we have $\X := \{0,\cdots, 10\}$. 

At each time step $t$ the agent can make a bet whether the sum of the next throw strictly exceeds the previous sum (i.e., $X_{t+1} > X_t$), or whether it is strictly smaller (i.e., $X_{t+1} < X_t$). If the agent is correct, she gets $1\$$, in contrast if the agent's bet is wrong she has to pay $1\$$. The agent also has the possibility not to play. We model this by considering the following reward function:
\begin{equation}
  \X \times A \times \X  \ni (x, a, x') \mapsto r(x, a, x') := a\one_{\{x<x'\}}- a\one_{\{x>x'\}}-|a|\one_{\{x=x'\}}, 
\end{equation}
where the possible actions are given by $A := \{-1, 0, 1\}$ with $a = 1$ corresponding to betting $X_{t+1} > X_t$, $a=0$ to not playing, and $a=-1$ to betting $X_{t+1}<X_t$.
{This formula specifies the rewards returned by the simulators; the empirical learner observes its value only at sampled transitions. We use the explicit reward function for the exact reference calculations and policy evaluation, but not to evaluate unobserved transitions during empirical learning.}
In a next step, we define two different ambiguity sets via\footnote{We denote by $\Bin(n,p)$ a binomial distribution with $n$ number of trials and $p$ being the probability of success in a single trial.}
\begin{align}
    \mathcal{P}^1(x,a) := & \{\Bin(10, 0.5),~ \Bin(10, 0.6)\}, \label{P1}\\
    \mathcal{P}^2(x,a) := & \{\Bin(10, 0.5), ~\Bin(10, 0.3)\}, \label{P2}
\end{align}
and we aim at comparing policies for these finite sets with policies for the Wasserstein ambiguity sets considered in \citet{neufeld2023markov}. To build comparative ambiguity sets with Wasserstein uncertainty we define ambiguity sets such that all probability measures from $\mathcal{P}^1 $ and $\mathcal{P}^2$ are contained, respectively. This leads to\footnote{Here $W_1$ denotes the Wasserstein-$1$ distance, compare for  more details, e.g., \citep{villani2008optimal}.} 
\begin{align}
    \mathcal{P}^3(x,a) := & \left\{ \PP \in \mathcal{M}_1(\X)~\middle|~W_1(\PP, \Bin(10, 0.5)) \leq 1 \right\}, \label{P3}\\
    \mathcal{P}^4(x,a) := & \left\{ \PP \in \mathcal{M}_1(\X)~\middle|~W_1(\PP, \Bin(10, 0.5)) \leq 2 \right\}, \label{P4}
\end{align}
which are Wasserstein-balls centered at the \emph{reference probability} $\Bin(10, 0.5)$ with radii $1$ and $2$, respectively. As one can see, e.g., from \citet[Equation (4.2)] {neufeld2022robust}, $\mathcal{P}^3$ is the smallest Wasserstein-ball  around the reference measure that contains $\Bin(10, 0.6)$, whereas $\mathcal{P}^4$  is the smallest Wasserstein-ball containing $\Bin(10, 0.3)$, but, by construction they contain many additional distributions including distributions outside the convex hull of the respective candidates. For example, $\mathcal P^3$ also contains $\Bin(10,0.4)$.

We determine reference policies for $\mathcal{P}^1,\ldots,\mathcal{P}^4$ by the corresponding $Q$-learning algorithms and verify them by robust value iteration. Table~\ref{tbl_trained_actions_exa1} depicts the resulting policies, together with the non-robust policy for $\operatorname{Bin}(10,0.5)$. The finite-set policies also provide reference solutions for evaluating Algorithm~\ref{alg:Q-L}.

\begin{table}[ht]
\centering
\small
\begin{tabular}{@{}lrrrrrrrrrrr@{}}
\toprule
State $x$ & 0 & 1 & 2 & 3 & 4 & 5 & 6 & 7 & 8 & 9 & 10 \\
\midrule
Robust, $\mathcal{P}^1$ & $1$ & $1$ & $1$ & $1$ & $1$ & $0$ & $0$ & $-1$ & $-1$ & $-1$ & $-1$ \\
Robust, $\mathcal{P}^2$ & $1$ & $1$ & $1$ & $0$ & $0$ & $0$ & $-1$ & $-1$ & $-1$ & $-1$ & $-1$ \\
Robust, $\mathcal{P}^3$ & $1$ & $1$ & $0$ & $0$ & $0$ & $0$ & $0$ & $0$ & $0$ & $-1$ & $-1$ \\
Robust, $\mathcal{P}^4$ & $0$ & $0$ & $0$ & $0$ & $0$ & $0$ & $0$ & $0$ & $0$ & $0$ & $0$ \\
Non-robust & $1$ & $1$ & $1$ & $1$ & $1$ & $0$ & $-1$ & $-1$ & $-1$ & $-1$ & $-1$ \\
\bottomrule
\end{tabular}
\caption{Reference policies obtained by robust value iteration for the indicated ambiguity sets, and the non-robust reference policy.}
\label{tbl_trained_actions_exa1}
\end{table}

We test the policies under independent successive outcomes with distribution $\pi_p=\operatorname{Bin}(10,p_{\rm true})$. The long-run mean reward per step can be evaluated directly as
\[
\overline r_p(a)=\sum_{x,y\in\X}\pi_p(x)\pi_p(y)r(x,a(x),y).
\]
Table~\ref{tbl_cumulative_reward_exa1} and Figure~\ref{fig:avg_profit} report these expected rewards. This deterministic evaluation avoids sampling variability in the comparison of the prescribed ambiguity sets.

\begin{table}[ht]
\centering
\small
\begin{tabular}{@{}lrrrrrrrrr@{}}
\toprule
$p_{\rm true}$ & 0.1 & 0.2 & 0.3 & 0.4 & 0.5 & 0.6 & 0.7 & 0.8 & 0.9 \\
\midrule
Robust, $\mathcal{P}^1$ & $-0.3110$ & $-0.1905$ & $-0.0431$ & $0.1567$ & $0.3036$ & $\mathbf{0.2979}$ & $\mathbf{0.1301}$ & $-0.1028$ & $-0.2997$ \\
Robust, $\mathcal{P}^2$ & $-0.2438$ & $0.0499$ & $\mathbf{0.2100}$ & $\mathbf{0.2473}$ & $0.2267$ & $0.1189$ & $-0.0519$ & $-0.1913$ & $-0.3110$ \\
Robust, $\mathcal{P}^3$ & $-0.0774$ & $\mathbf{0.1510}$ & $0.1117$ & $0.0442$ & $0.0211$ & $0.0442$ & $0.1117$ & $\mathbf{0.1510}$ & $-0.0774$ \\
Robust, $\mathcal{P}^4$ & $\mathbf{0.0000}$ & $0.0000$ & $0.0000$ & $0.0000$ & $0.0000$ & $0.0000$ & $0.0000$ & $0.0000$ & $\mathbf{0.0000}$ \\
Non-robust & $-0.3108$ & $-0.1851$ & $-0.0098$ & $0.2311$ & $\mathbf{0.3541}$ & $0.2311$ & $-0.0098$ & $-0.1851$ & $-0.3108$ \\
\bottomrule
\end{tabular}
\caption{Expected reward per step under the reference policies in Table~\ref{tbl_trained_actions_exa1}. Bold entries indicate the largest reward among the five policies at each test parameter.}
\label{tbl_cumulative_reward_exa1}
\end{table}

\begin{figure}[ht]
\centering
\includegraphics[width=0.75\linewidth]{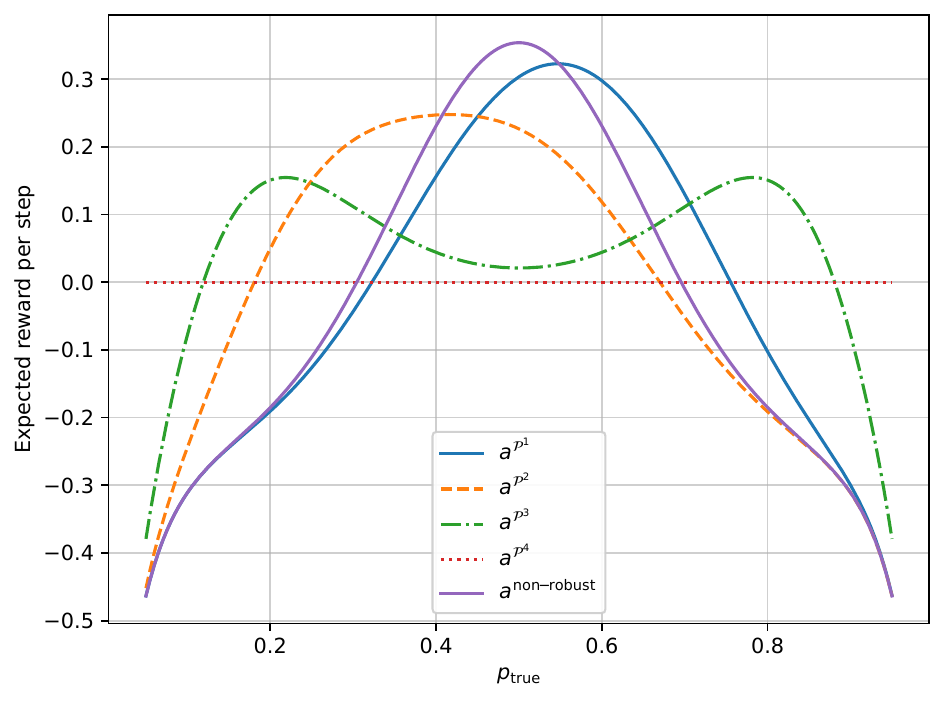}
\caption{Expected reward per step of the reference policies as a function of the binomial test parameter $p_{\rm true}$. The values are computed directly from the binomial probabilities.}
\label{fig:avg_profit}
\end{figure}

For the reference model $p_{\rm true}=0.5$, the non-robust policy has the highest expected reward. The policy for $\mathcal P^1$ has a higher reward than the policy for $\mathcal P^3$ at both $p_{\rm true}=0.5$ and $0.6$; likewise, the policy for $\mathcal P^2$ has a higher reward than the policy for $\mathcal P^4$ at $p_{\rm true}=0.3$ and $0.5$. Outside these prescribed scenarios the ranking can change. In particular, the policy for $\mathcal P^3$ performs best among the five policies at $p_{\rm true}=0.2$ and $0.8$.

The Wasserstein policy for $\mathcal P^3$ has a wider abstention region. For the larger ball $\mathcal P^4$, abstention is optimal at every state and the robust value is zero. These differences are caused by the additional transition laws in the Wasserstein balls, including laws outside the convex hull of the finite candidate sets. The example thus illustrates how the choice of the ambiguity set affects the policy and its performance, it does not imply uniform superiority of one uncertainty model.

To complement the preceding illustrations, we conduct an experiment with the same reward function and discount factor. We consider
$N\in\{2,5,11\}$ candidate kernels $\operatorname{Bin}(10,p_k)$, where
$p_k=0.5+0.1(k-1)/(N-1)$, $k=1,\ldots,N$. For each candidate set, we compute
a reference solution by robust value iteration, stopping when consecutive
$Q$-iterates differ by at most $10^{-11}$ in the supremum norm. The contraction
property bounds the error of this reference by $2\cdot10^{-10}$.
We then run Algorithm~\ref{alg:Q-L} for $T=10^6$ iterations, initial state $5$, $Q_0=0$, learning rate
$(1+n)^{-0.6}$ at the $(n+1)$st visit, and
$\varepsilon_t=1-0.95t/(T-1)$ for $t=0,\ldots,T-1$.
The empirical kernels are initialized uniformly on $\X$, identically for all candidates, and all empirical reward means $\bar r_0^{(k)}(x,a)$ are initialized at zero. At each iteration, we select the first minimizing candidate using the past reward means and empirical continuation values in \eqref{eq_definition_k_t}, before drawing any fresh samples. Each candidate is then queried independently at the same current state--action pair $(x,a)$, returning a successor and its observed reward $\mathsf R_{t+1}^{(k)}$. With $n=N_t(x,a)$ previous visits, in line with \eqref{eq:reward_update} we update the mean reward estimate at the pair $(x,a)$ via
\[
\bar r_{t+1}^{(k)}(x,a)
=\frac{n\bar r_t^{(k)}(x,a)+\mathsf R_{t+1}^{(k)}}{n+1},
\]
leaving reward means at other pairs unchanged. The $Q$-update uses the selected candidate's fresh observed reward, not its estimated mean. Thus empirical learning requires no reward evaluations at unobserved transitions.

We also consider the variant with known kernels and exact candidate reward expectations. It selects the minimizing model exactly and draws only the selected successor. This comparison therefore distinguishes empirical estimation of both transition laws and reward expectations from exact evaluation of these quantities.

\begin{table}[ht]
\centering
\small
\begin{tabular}{@{}crrrr@{}}
\toprule
& \multicolumn{2}{c}{Known kernels} & \multicolumn{2}{c}{Empirical kernels}\\
$N$ & $\|Q_T-Q^*\|_\infty$ & Time (s) & $\|Q_T-Q^*\|_\infty$ & Time (s)\\
\midrule
2 & $0.0905\pm0.0368$ & 0.160 & $0.0937\pm0.0392$ & 0.271\\
5 & $0.0766\pm0.0248$ & 0.178 & $0.0743\pm0.0253$ & 0.692\\
11 & $0.0758\pm0.0207$ & 0.186 & $0.0725\pm0.0301$ & 1.328\\
\bottomrule
\end{tabular}
\caption{Comparison on the coin toss example. Errors are means
$\pm$ sample standard deviations over five seeds; learning times are means.
{The empirical variant estimates candidate reward expectations from observed rewards, whereas the known-kernel variant uses exact reward expectations.}
The known-kernel variant uses $10^6$ simulator calls,
whereas the empirical-kernel variant uses $N\cdot10^6$ calls; thus the table
compares a fixed number of updates, not equal sampling budgets.}
\label{tbl:validation_scaling}
\end{table}

Because the rewards are deterministic, identity~\eqref{eq:empirical_reward_identity} implies that the observed-reward and known-reward empirical formulations have the same mathematical iterates when they use the same common initial laws, tie-breaking rule, and simulator draws. No additional reward-estimation error term is needed.

The greedy policies agree with the reference policy at all states in all
30 runs. The approximately
linear growth in empirical-kernel sampling cost illustrates the limitation of
enumerating large candidate sets.

Figure~\ref{fig:convergence_coin_toss} depicts the supremum-norm error for the two ambiguity sets $\mathcal P^1$ and $\mathcal P^2$, using the same learning and exploration settings as in the preceding comparison. {The empirical variant uses the same observed-reward estimation procedure described above.} The reference values are computed by robust value iteration. The curves show the mean over five seeds, with the full observed range shaded. Both known-kernel and empirical-kernel variants approach the reference solutions in these runs; their relative errors vary with the iteration and seed.

\begin{figure}[ht]
\centering
\includegraphics[width=\linewidth]{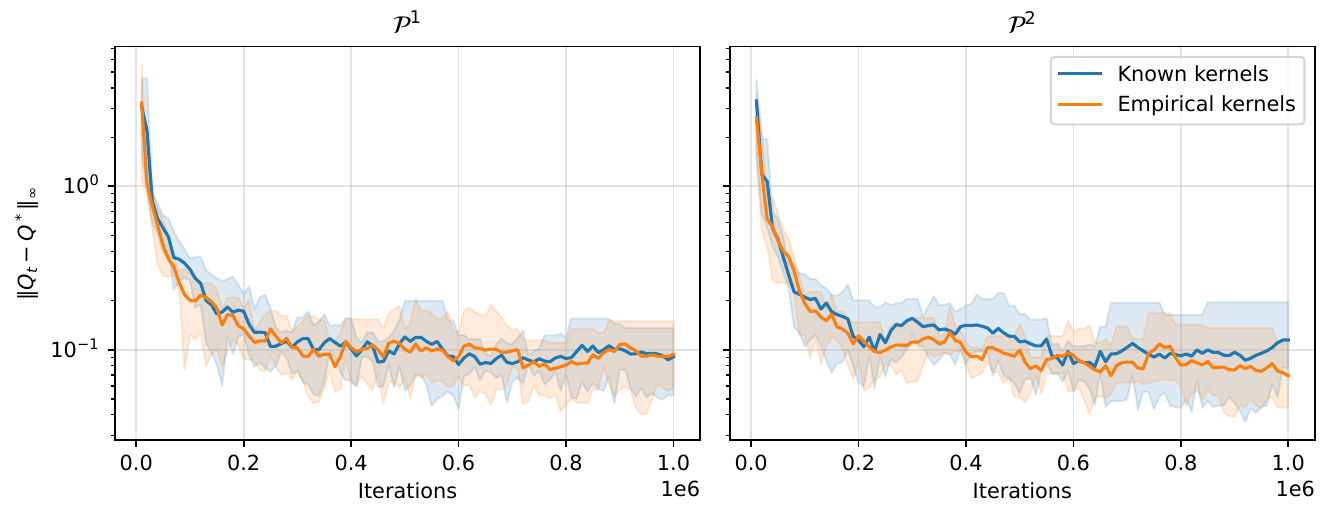}
\caption{Supremum-norm errors for $\mathcal P^1$ (left) and $\mathcal P^2$ (right), {comparing known kernels and exact reward expectations with empirical kernels and observed reward means}. Lines indicate means over five seeds, and shaded regions show the corresponding minima and maxima.}
\label{fig:convergence_coin_toss}
\end{figure}

\subsection{A comparison of robust $Q$-learning with robust planning}\label{subsec:three_state}

In this example we compare Algorithm~\ref{alg:Q-L} with robust planning based on the same simulator observations and show that robust $Q$-learning can be more efficient in terms of policy regret. 

To this end, we consider a decision between a safe and a risky payoff. The state space is $\X=\{s,g,b\}$, where $s$ is the initial decision state and $g,b$ are absorbing states, and the action space is $A=\{a_{\rm safe},a_{\rm risk}\}$. For $k=1,2$, the candidate kernels at $s$ are
\[
\PP^{(k)}(s,a_{\rm safe})=\delta_g,\qquad
\PP^{(k)}(s,a_{\rm risk})=p_k\delta_g+(1-p_k)\delta_b.
\]
The safe action gives the deterministic reward $c$, whereas the risky action gives reward $1$ on reaching $g$ and $-1$ on reaching $b$. All rewards after absorption are zero. Hence, for every discount factor $\alpha\in(0,1)$, we have
\begin{align*}
Q^*(s,a_{\rm safe})&=\min_{k={1,2}} \mathbb{E}_{\PP^{(k)}(s,a_{\rm safe})}\left[ r(s,a_{\rm safe},X_1) + \alpha V(X_1)\right]= r(s,a_{\rm safe},g) + \alpha V(g)=c,\\
Q^*(s,a_{\rm risk})&=\min_{k={1,2}} \mathbb{E}_{\PP^{(k)}(s,a_{\rm risk})}\left[ r(s,a_{\rm risk},X_1) + \alpha V(X_1)\right]\\
&= \min_{k={1,2}}\left\{p_k+(1-p_k)(-1)\right\} =2\min\{p_1,p_2\}-1.
\end{align*}
We set $p_1=0.60$, $p_2=0.61$, and $c=0.10$, which means that the risky action is robust-optimal with value $0.20$, and choosing the safe action is suboptimal (as it only gives a value $0.1$), and thus incurs robust policy regret $0.10$. The probabilities $p_1,p_2$ are used to generate observations and to evaluate the learned policies, but are not known to the learner. We first describe the comparison with a known reward function, successor supports, and zero continuation values; below we additionally consider learning both rewards and kernels from transition--reward observations.

We use the generative access described in Remark~\ref{rem:generative_queries}, querying $(s,a_{\rm risk})$ repeatedly. Deterministic safe-action and absorbing-state updates are evaluated exactly: with $Q_0=0$ and first learning rate equal to one, the safe value equals $c$ after its first update and the absorbing values remain zero. 

Let $B_m^{(k)}$ be the success indicator of the $m$th fresh observation from candidate $k\in \{1,2\}$:
 with all observations independent, i.e.,
 \[
 B_m^{(k)}= \begin{cases} 1, &\text{ if the sampled next state is } g ~(\text{ reward } +1),\\
 0, &\text{ if the sampled next state is } b ~(\text{ reward } -1).\\
 \end{cases}
 \]
 
 and put $C_n^{(k)}:=\sum_{m=1}^n B_m^{(k)}$. At the $m$th risky-action update, Algorithm~\ref{alg:Q-L} selects the model $k_m \in \{1,2\}$ according to
\[
k_m=\begin{cases}
1,&C_{m-1}^{(1)}\leq C_{m-1}^{(2)},\\
2,&C_{m-1}^{(1)}>C_{m-1}^{(2)},
\end{cases}
\qquad\text{before drawing }B_m^{(1)},B_m^{(2)}.
\]
and uses $B_m^{(k_m)}$ in the $Q$-update.
This corresponds to identical initial empirical laws and ties between candidate models being resolved in favour of model~1. With learning rate $\widetilde\gamma_{m-1}=1/m$, the estimated risky-action value is therefore
\begin{equation}\label{eq:three_state_q}
\widehat q_n^{\rm Q}=\frac{2S_n}{n}-1,
\qquad S_n:=\sum_{m=1}^n B_m^{(k_m)}.
\end{equation}

The empirical robust planner solves the estimated MDP exactly. Its estimated risky-action value is
\begin{equation}\label{eq:three_state_plan}
\widehat q_n^{\rm P}=2\min_{k=1,2}\frac{C_n^{(k)}}{n}-1.
\end{equation}
 Each learned policy chooses the risky action if its estimated value is strictly larger than $c$; action ties are resolved in favour of the safe action. Thus, the comparison concerns empirical robust planning under an equal sampling budget. Planning supplied with the true candidate probabilities already selects a robust-optimal action.

For an action $a$ at $s$, denote its true robust return by $V^a(s)$ and define its regret (i.e. the \emph{loss} from being suboptimal) by $L(a):=V(s)-V^a(s)$. The following result establishes a strict comparison in expectation for the harmonic learning rate.

\begin{prop}\label{prop:three_state}
Suppose that $0<p_1\leq p_2<1$ and $-1<c<2p_1-1$. Let $a_n^{\rm Q}$ and $a_n^{\rm P}$ denote the actions obtained from \eqref{eq:three_state_q} and \eqref{eq:three_state_plan}. Set $\Delta:=2p_1-1-c>0$, $\theta:=(1+c)/2$, and
\[
A_k(n):=\PP\bigl(\operatorname{Bin}(n,p_k)>n\theta\bigr),\qquad k=1,2.
\]
Then, for every integer $n\geq1$,
\begin{equation}\label{eq:three_state_regret}
\EE[L(a_n^{\rm Q})]
\leq\Delta\bigl(1-A_1(n)\bigr)
<\Delta\bigl(1-A_1(n)A_2(n)\bigr)
=\EE[L(a_n^{\rm P})].
\end{equation}
The expectations are over the training observations. Moreover, on every shared sample path,
\begin{equation}\label{eq:three_state_identity}
\widehat q_n^{\rm Q}=\widehat q_n^{\rm P}+\frac{2T_n}{n},\qquad
T_n:=\sum_{m=1}^n\one_{\{C_{m-1}^{(1)}=C_{m-1}^{(2)},\ B_m^{(1)}=1,\ B_m^{(2)}=0\}}.
\end{equation}
\end{prop}

 The minimum of the empirical candidate means in \eqref{eq:three_state_plan} can underestimate the robust risky-action value. In contrast, the selected fresh reward in Algorithm~\ref{alg:Q-L} has conditional mean $2p_{k_m}-1\geq2p_1-1$. Identity~\eqref{eq:three_state_identity} shows more precisely that $Q$-learning chooses the risky action whenever empirical robust planning does. This favours the learned policy when the risky action is truly optimal whereas it can have the opposite effect when the safe action is optimal.

We now additionally suppose that neither method is given the reward function. Each candidate instead supplies a transition--reward pair, as in Algorithm~\ref{alg:Q-L}, and its reward mean is initialized at zero. The successor supports and zero rewards after absorption remain known. Querying the safe action once from each candidate reveals its deterministic successor $g$ and reward $c$; the empirical safe value and its $Q$-estimate are then exact because the first learning rate is one. All methods receive these same two safe-action observations, and no further safe queries are needed. For the risky action, write $\mathsf R_m^{(k)}:=2B_m^{(k)}-1$ for the observed reward. After $n\geq1$ queries per candidate, the estimated reward mean and transition law satisfy
\begin{equation}\label{eq:three_state_observed_rewards}
\begin{aligned}
\bar r_n^{(k)}(s,a_{\rm risk})
&=\frac{1}{n}\sum_{m=1}^n\mathsf R_m^{(k)}
=\frac{2C_n^{(k)}}{n}-1,\\
\PP_n^{(k)}(s,a_{\rm risk})
&=\frac{C_n^{(k)}}{n}\delta_g+
\left(1-\frac{C_n^{(k)}}{n}\right)\delta_b.
\end{aligned}
\end{equation}
These quantities use only sampled transitions and their observed rewards, without evaluating the reward function at unobserved transitions.

Since the continuation values are zero, candidate selection in \eqref{eq_definition_k_t} minimizes the \emph{past} reward means. By \eqref{eq:three_state_observed_rewards}, this gives the same $k_m$ as ordering the past success counts, including the first-query tie between the zero-initialized means. The $Q$-update uses the selected candidate's fresh reward, whereas empirical robust planning uses $\min_k\bar r_n^{(k)}(s,a_{\rm risk})$, which equals \eqref{eq:three_state_plan}. Consequently, both learning rates give the same iterates and policies as in the known-reward formulation on every shared sample path. Proposition~\ref{prop:three_state} and the exact regret comparison below therefore apply unchanged when both rewards and kernels are learned. The common budget in this formulation is $2n+2$ simulator calls: $2n$ stochastic risky-action calls and two deterministic safe-action calls. 

For the main numerical comparison, we use $200{,}000$ independent training repetitions and the nine nested budgets in Table~\ref{tbl:three_state}. Within each repetition all methods share the same observations. In addition to the two learning rates and empirical robust planning, we include two further estimators. Sample-split planning divides each candidate's observations into odd and even indices, selects the empirical worst candidate on one part, evaluates it on the other part, and averages the two estimates obtained by exchanging the parts. Pooled nominal planning uses the average of the two empirical candidate laws. The latter solves a different, non-robust control problem. Every resulting policy is evaluated using the true robust action values, without additional simulation noise. Identity~\eqref{eq:three_state_observed_rewards} also holds within each sample-split fold and after pooling, so the policies of these additional estimators are likewise unchanged when rewards are learned.

\begin{table}[htbp]
\centering
\small
\begin{tabular}{@{}rrrrrr@{}}
\toprule
& \multicolumn{2}{c}{Robust $Q$-learning} & Empirical & Sample-split & Pooled nominal\\
$n$ & $1/n$ & $n^{-0.6}$ & robust planning & planning & planning\\
\midrule
2 & 0.063709 & 0.039677 & 0.086618 & 0.058029 & 0.051647 \\
5 & 0.031361 & 0.040422 & 0.052322 & 0.035593 & 0.035536 \\
10 & 0.035933 & 0.038940 & 0.058389 & 0.038611 & 0.038559 \\
20 & 0.039101 & 0.035712 & 0.062242 & 0.041117 & 0.028810 \\
50 & 0.021770 & 0.031331 & 0.037958 & 0.026105 & 0.015235 \\
100 & 0.016013 & 0.027543 & 0.028535 & 0.020230 & 0.006525 \\
200 & 0.007142 & 0.022852 & 0.013074 & 0.010206 & 0.001422 \\
500 & 0.000937 & 0.016525 & 0.001650 & 0.001521 & 0.000021 \\
1000 & 0.000037 & 0.011614 & 0.000071 & 0.000069 & 0.000000 \\
\bottomrule
\end{tabular}
\caption{Mean true robust policy regret over $200{,}000$ repetitions for $p_1=0.60$, $p_2=0.61$, and $c=0.10$. The budget $n$ is the number of observations per candidate; all methods use the same $2n$ stochastic simulator calls. The same policies, hence the same entries, result when rewards are learned, with two additional deterministic safe-action calls. Smaller values are better. A zero entry means that no incorrect actions were observed.}
\label{tbl:three_state}
\end{table}

At $n=20$, an exact calculation gives
\[
\EE[L(a_{20}^{\rm Q})]=0.0391822410,\qquad
\EE[L(a_{20}^{\rm P})]=0.0624061535,
\]
up to the displayed rounding, a reduction of approximately $37.2\%$. These values are obtained by propagating the joint distribution of $(C_n^{(1)},C_n^{(2)},S_n)$ through the four possible fresh outcome pairs; an independent rational-arithmetic calculation verifies them. The corresponding Monte Carlo means are $0.0391005$ and $0.0622420$. The paired mean regret difference is $-0.0231415$, with a pointwise $95\%$ Hoeffding confidence interval $[-0.023749,-0.022534]$, using the independence of training repetitions and the range $[-0.10,0.10]$ of each paired difference. All budgets and both learning rates are retained. In particular, the learning rate $n^{-0.6}$ has larger mean regret than empirical robust planning at $n=200,500,1000$. The nonmonotonicity at some smaller budgets is caused by the integer-valued decision threshold and is also present in the exact calculations.

\begin{figure}[htbp!]
\centering
\includegraphics[width=\linewidth]{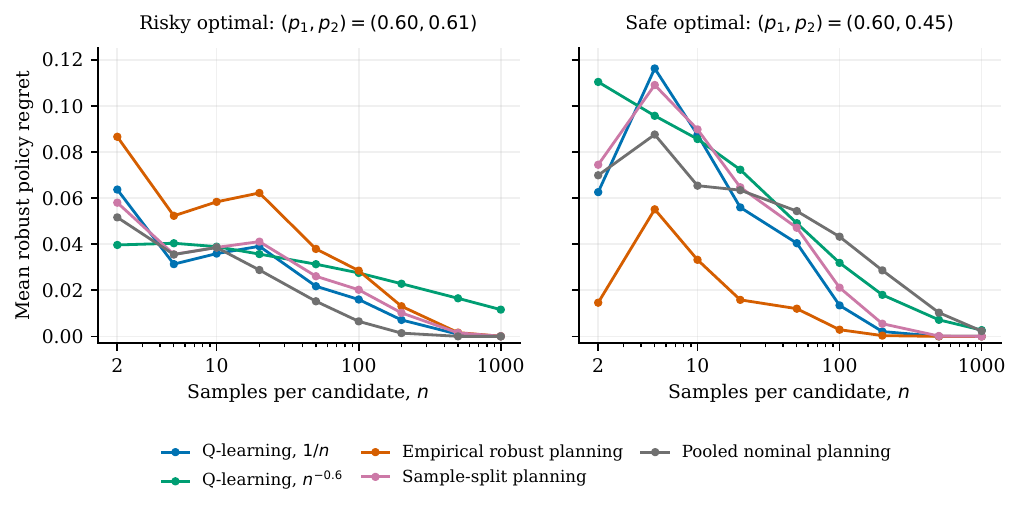}
\caption{Mean robust policy regret in the three-state example (left, $200{,}000$ repetitions) and the control with $p_2=0.45$ (right, $100{,}000$ repetitions), keeping $p_1=0.60$ and $c=0.10$. The risky action is robust-optimal on the left and the safe action on the right. The budgets and sampling conventions are the same in both panels.}
\label{fig:three_state}
\end{figure}

To examine the scope of the comparison, we change only $p_2$ to $0.45$ and run $100{,}000$ independent repetitions. The robust risky-action value is then $-0.10$, so the safe action is optimal. At $n=20$, the mean regrets are $0.055972$ for harmonic $Q$-learning and $0.015858$ for empirical robust planning. Figure~\ref{fig:three_state} reports all budgets in both cases. The additional estimators also delimit the claim: pooled nominal planning performs better in the primary example, but can incur higher robust regret in the control. Moreover, at $n=20$ in the primary example, the mean squared errors of the risky-action value estimates are $0.048358$ for harmonic $Q$-learning and $0.046773$ for empirical robust planning. Thus, the advantage concerns expected robust policy regret in the stated family.what d

\subsection{Portfolio allocation}\label{subsec:StockInvesting}

We conclude the numerical section with an example which is a variant of \citet[Example 4.3]{neufeld2023markov} illustrating tailored scenario sets in comparison with the reference-ball formulations considered in \citet{bauerle2021q}, \citep{liu2022distributionally}, and \citep{wang2023finite}.

We consider the problem of optimally investing in the stock market, and, to this end, encode the return of a stock by $2$ numeric values: either the return is nonnegative ($1$) or negative ($-1$).
The space of numerically encoded returns is therefore given by $T:=\{-1, 1\}$.
At each time step $t$ the decision maker can choose to buy the stock or not, or to short sell the stock. The action $a_t$ represent this investment decision, where buying the stock is encoded by $1$ and not buying it by $0$, whereas short-selling corresponds to $-1$, i.e., we have $A:=\{-1,0,1\}$.

The agent's investment decision should depend not only on the most recent return but naturally will depend on  the current \emph{market situation} reflecting autocorrelation of financial time series. Therefore the agent relies her investment decision on the last $h=5$ returns and on the previous position. Hence, we consider the state space
$$\mathcal{X}:=T^h \times A =\{-1, 1\}^h \times \{-1,0,1\}.$$
 The reward is given by the function:
\begin{align*}
 r:\X\times A\times\X &\rightarrow \R\\
\bigg((x_{-h+1},\cdots,x_0,a_0),a_1,(x_{-h+2},\cdots,x_1,a_1)\bigg)  &\mapsto  a_1\cdot x_1-\eta \cdot |a_1-a_0|,
\end{align*}
i.e., we reward a correct investment decision and penalize an incorrect investment decision. Additionally, the term $\eta \cdot |a_1-a_0|$ penalizes position changes, accounting for transaction costs. In our experiments we set $\eta = 0.1$.

Now, to train the agent, we construct market simulators in the following way. We consider the historical evolution of the (numerically encoded by $-1$ and $1$) daily returns of the underlying stock. This time series is denoted by $(R_j )_{j=1,...,N} \subset T^N$ for some $N \in \mathbb{N}$ labeling the length of the time series. Given some $x=(x_{-h+1},\cdots,x_0,a_0) \in \X$ we then sample uniformly the next return from all historical returns sequences with coinciding history\footnote{If no such sequence exist we sample uniformly from $T$.} $(x_{-h+1},\cdots,x_0)$. Then, the state transition of the reference measure is partially deterministic ($h-1$ components of subsequent states coincide) and the next state is of the form 
\[
(x_{-h+2},\cdots,x_0,y,a_1) \in \X
\]
where $y$ is the sampled return and $a_1$ the executed action.

To apply this construction to real data, we consider daily return signs for \emph{Nvidia, Apple, Microsoft, Amazon, Alphabet, Broadcom, and JP Morgan} (tickers: NVDA, AAPL, MSFT, AMZN, GOOGL, AVGO, JPM). The  data uses adjusted closing prices from January $2010$ through December $2025$ obtained from Yahoo Finance. The calculations below use the same $4000$ signs per stock as Figure~\ref{fig:train_test}.
\begin{figure}[ht]
\centering
\includegraphics[width=0.95\textwidth]{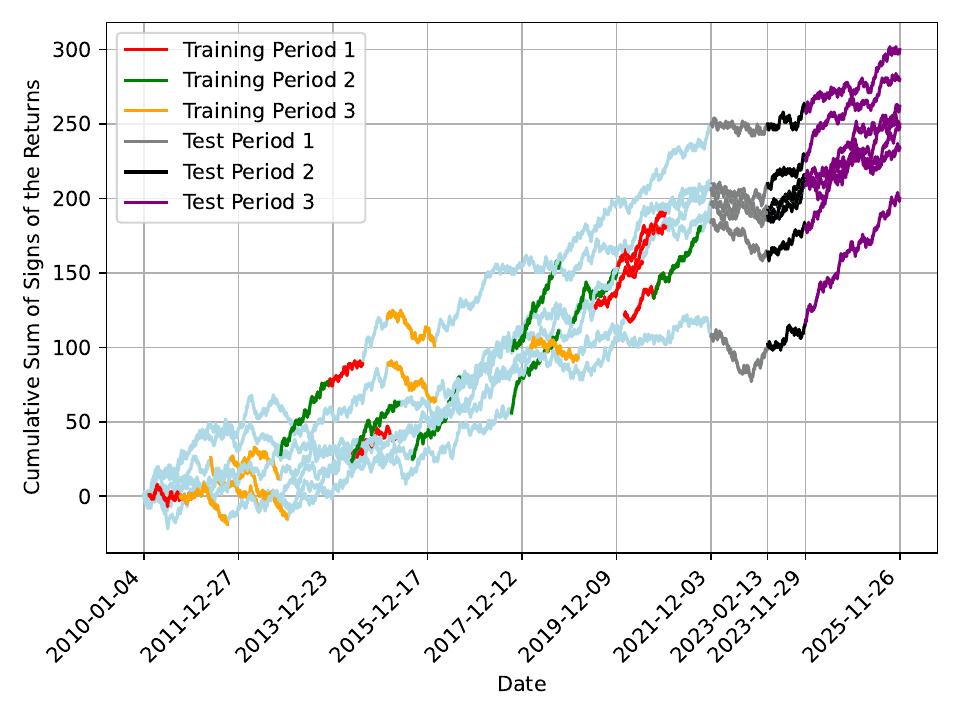}
\caption{Cumulative sums of return signs and the original training and testing windows. The numerical evaluations use the observation-index intervals specified in the text.}
\label{fig:train_test}
\end{figure}

We use the first $3000$ observations (January 2010 until December 2021) for training. The three test periods range from January 2022 to February 2023, from February 2023 until January 2024, and from January 2024 until November 2025, with lengths $299$, $199$, and $500$, respectively. To construct an ambiguity set $\mathcal P$, we isolate $250$-observation subperiods of the training sample with the largest mean absolute change between consecutive signs (Training Period 1), the largest cumulative sign sum (Training Period 2), and the smallest cumulative sign sum (Training Period 3). The first criterion measures sign alternation rather than the standard deviation of price returns. These periods vary from stock to stock. We denote the corresponding transition kernels by $\PP_1,\PP_2,\PP_3$, respectively, and construct $\PP_4$ from all training observations. Including these kernels in $\mathcal P:=\{\PP_1,\PP_2,\PP_3,\PP_4\}$ represents protection against the selected historical scenarios. The training-period choices are depicted in Figure~\ref{fig:train_test}.

We train agents for the ambiguity set $\mathcal P$ and for each single kernel $\PP_i$, $i=1,2,3,4$. In this empirical example the finite transition probabilities can also be obtained directly from historical counts; using simulators illustrates the information setting rather than an advantage over exact planning. The empirical laws are initialized with equal probability on the two admissible successor states, with deterministic tie-breaking between candidates.

For this example, we use the generative-query variant of
Algorithm~\ref{alg:Q-L} (Remark~\ref{rem:generative_queries}): at each update the state--action pair is drawn
independently and uniformly from $\X\times A$. This ensures exploration
without requiring the fitted historical kernels to generate an irreducible
trajectory. We use $h=5$, $\eta=0.1$, $\alpha=0.95$, $Q_0=0$,
the learning rate $(1+n)^{-0.6}$ at the $(n+1)$st visit, and $10^6$ updates
for each stock, candidate set, and seed $0,\ldots,4$. Candidate kernels are
estimated separately for each stock. In every update, all candidate
successors use the same queried state and action, and the minimizing model
is selected before the fresh samples are drawn.

We evaluate each trained greedy policy on every observation of the three
test periods, using the preceding five observed signs for the initial
history and an initial position of zero. The reward charges the change
from the previous position to the new position; the opening transaction
is included, and no forced liquidation is imposed at the end of a period.
The buy-and-hold benchmark uses the same convention.

\begin{table}[ht]
\centering
\small
\begin{tabular}{@{}lrrr@{}}
\toprule
& Test Period 1 & Test Period 2 & Test Period 3\\
\midrule
Days & 299 & 199 & 500\\
Negative returns & 1078 & 635 & 1573\\
Positive returns & 1015 & 758 & 1927\\
\midrule
Robust ($\mathcal P$) & $\mathbf{-0.0048}\pm0.0029$ & $-0.0035\pm0.0059$ & $-0.0024\pm0.0027$\\
Non-robust: Period 1 ($\PP_1$) & $-0.1071\pm0.0069$ & $-0.0438\pm0.0071$ & $-0.0383\pm0.0050$\\
Non-robust: Period 2 ($\PP_2$) & $-0.0639\pm0.0057$ & $-0.0052\pm0.0048$ & $0.0121\pm0.0060$\\
Non-robust: Period 3 ($\PP_3$) & $-0.0404\pm0.0070$ & $-0.0785\pm0.0105$ & $-0.0795\pm0.0040$\\
Non-robust: All data ($\PP_4$) & $-0.0346\pm0.0065$ & $0.0566\pm0.0066$ & $0.0687\pm0.0096$\\
Buy and hold & $-0.0304$ & $\mathbf{0.0878}$ & $\mathbf{0.1009}$\\
\bottomrule
\end{tabular}
\caption{Mean directional reward per time step across seven stocks. The
learning rows report means $\pm$ sample standard deviations across five
seeds, after averaging over stocks within each seed. Bold entries indicate the highest mean in each period.}
\label{tbl:portfolio_rewards}
\end{table}

\begin{figure}[htb!]
\includegraphics[width = 0.95\linewidth]{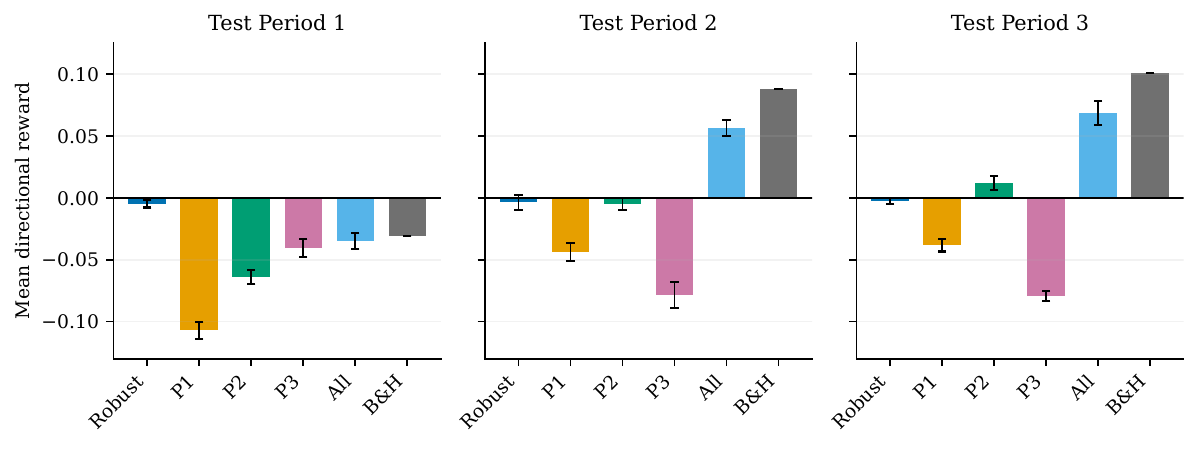}
\caption{The histograms illustrate the results from Table~\ref{tbl:portfolio_rewards}.}\label{fig:portfolio_results}
\end{figure}

Table~\ref{tbl:portfolio_rewards} and Figure~\ref{fig:portfolio_results} illustrate the conservatism induced by
the selected stress scenarios. The robust policy has a smaller average
loss in the first period than the single-model policies, while giving up
most of the gains achieved by buy and hold in the later periods. Its
average reward remains slightly negative in all three periods, and the
zero-position benchmark performs better. Thus, this example
illustrates how historical subsamples can be used to construct the
candidate simulators. In particular, the sign-based reward is a directional score
and does not measure monetary investment returns.

\appendix
{
\section{Approximation of other Ambiguity sets with finite ambiguity sets}\label{sec:extensions}

In this appendix we show that using finite ambiguity sets - and hence our proposed $Q$-learning algorithm - provides a versatile framework for ambiguity modeling, as it cannot only be used to model ambiguity w.r.t.\,a finite pool of candidate models but also to approximate a wide range of
other classes of infinitely large ambiguity sets.

\subsection{Approximation of infinite ambiguity sets} \label{sec:infinite_ambiguity}
Consider some possibly infinitely large ambiguity set of probability measures
\begin{equation*}
\begin{aligned}
    \X \times A \ni (x,a)  \twoheadrightarrow \mathcal{P}(x,a)
\end{aligned} \subseteq \mathcal{M}_1(\X)
\end{equation*}
for which we want to compute its $Q$-value function $Q^*(x,a) = \inf_{\mathbb{P}\in \mathcal{P}(x,a)} \mathbb{E}_{\mathbb{P}}\left[ r(x,a,X_1) + \alpha V(X_1)\right]$. The state and action spaces remain finite throughout this appendix. We impose the following assumption on $\mathcal P$.
\begin{asu}\label{asu_p}~
The set-valued map 
\begin{align*}
\X \times A &\rightarrow\mathcal{M}_1(\X)\\
(x,a) &\twoheadrightarrow  \mathcal{P}(x,a)
\end{align*}
is assumed to be nonempty, compact-valued, and continuous in the weak topology.
\end{asu}
The above assumption ensures that minimizers of the corresponding value function exist (compare \citealp[Theorem 2.7]{neufeld2023markov}). Under these assumptions, we can derive the following Lemma~\ref{lem:Q_convergence} that allows to approximate the corresponding $Q$-value function by the $Q$-value function of an approximating sequence of finite ambiguity sets as defined in \eqref{eq:ambiguityset}, justifying in return to apply Algorithm~\ref{alg:Q-L} to compute an approximation of the $Q$-value function associated to $\mathcal{P}(x,a)$.

\begin{lem}\label{lem:Q_convergence}
Assume Assumption~\ref{asu_p} holds true. If there exists a sequence of finite ambiguity sets $\left(\mathcal{P}^{(n)}(x,a)\right)_{n \geq1}$ of nonempty finite subsets of $\mathcal{P}(x,a)$ such that for all $(x,a) \in \X \times A$ and all $\PP \in \mathcal{P}(x,a)$ there exists a sequence $(\PP^{(n)})_{n \in \N}$  with $\PP^{(n)} \in \mathcal{P}^{(n)}(x,a)$ for all $n\in \N$ such that\footnote{Note that the conditions in \eqref{eq:tau1convergence} are equivalent to convergence in the Wasserstein-$1$-distance, compare, e.g. \citep[Definition 6.8~(i)]{villani2008optimal}, and  Appendix~\ref{sec:wasserstein}.}
\begin{equation}\label{eq:tau1convergence}
\PP^{(n)} \rightarrow \PP \text{ weakly, and }\int_\X |x| \D \PP^{(n)}(x) \rightarrow \int_\X |x| \D \PP ( x)\text{ as }n \rightarrow \infty,
\end{equation} then we also have
\begin{align*}
\lim_{n \rightarrow \infty} {Q}^{(n)^*}(x,a)&:=\lim_{n \rightarrow \infty} \inf_{\mathbb{P}\in \mathcal{P}^{(n)}(x,a)} \mathbb{E}_{\mathbb{P}}\left[ r(x,a,X_1) + \alpha V^{(n)}(X_1)\right] \\
&= \inf_{\mathbb{P}\in \mathcal{P}(x,a)} \mathbb{E}_{\mathbb{P}}\left[ r(x,a,X_1) + \alpha V(X_1)\right]=:Q^*(x,a).
\end{align*}
where $V^{(n)}$ denotes the value function as defined in \eqref{eq:valuefunc} associated to the finite ambiguity set $\mathcal{P}^{(n)}(x,a)$.
\end{lem}

\subsection{Approximation of Wasserstein ambiguity sets}\label{sec:wasserstein}

Recently, \citet{neufeld2022robust} introduced a $Q$-learning algorithm that incorporates model uncertainty by allowing for distributional ambiguity within a Wasserstein-ball around a reference measure. We demonstrate that its robust Bellman problem can be approximated by designing a suitable sequence of finite ambiguity sets, and then applying Algorithm~\ref{alg:Q-L}. Solving the robust $Q$-learning problem for these finite sets then yields an approximation of the optimal $Q$-function corresponding to Wasserstein ambiguity.

For any $q \in \mathbb{N}$ and any probability measures $\PP_1,\PP_2 \in \mathcal{M}_1(\mathcal{X})$, the $q$-Wasserstein distance is defined by  
\[
W_q(\PP_1,\PP_2)
:= \left( \inf_{\pi \in \Pi(\PP_1,\PP_2)} 
       \int_{\mathcal{X}\times \mathcal{X}} \|x-y\|^q \, \mathrm{d}\pi(x,y)
   \right)^{1/q},
\]
where $\|\cdot\|$ denotes the Euclidean norm on $\R^d$, and where $\Pi(\PP_1,\PP_2) \subset \mathcal{M}_1(\mathcal{X}\times\mathcal{X})$ denotes the set of all couplings of $\PP_1$ and $\PP_2$.

Since we work on a finite state space, each measure $\PP_i$ can be represented as  
\[
\PP_i = \sum_{x\in \mathcal{X}} a_{i,x} \delta_x,\qquad
\sum_{x\in\mathcal{X}} a_{i,x} = 1,\quad a_{i,x}\ge0,\quad i=1,2,
\]
where $\delta_x$ denotes the Dirac mass at $x \in \mathcal{X}$.  
In this case, the Wasserstein distance admits the equivalent discrete formulation  
\[
W_q(\PP_1,\PP_2)
= \left( 
      \min_{\pi_{x,y} \in \widetilde{\Pi}(\PP_1,\PP_2)} 
      \sum_{x,y \in \mathcal{X}} \|x-y\|^q\, \pi_{x,y}
   \right)^{1/q},
\]
where  
\[
\widetilde{\Pi}(\PP_1,\PP_2)
:= \left\{
    (\pi_{x,y})_{x,y\in\mathcal{X}} \subseteq [0,1]
    \ \middle| \
    \sum_{x'\in\mathcal{X}} \pi_{x',y} = a_{2,y},\ 
    \sum_{y'\in\mathcal{X}} \pi_{x,y'} = a_{1,x},\
    \forall\, x,y\in\mathcal{X}
   \right\}.
\]

Fix $\varepsilon>0$ and $q\in\mathbb{N}$. For each $(x,a)\in \mathcal{X}\times A$, we define the ambiguity set  
\begin{equation}\label{eq_def_P1_rewritten}
\mathcal{P}^{(q,\varepsilon)}(x,a)
:= \left\{
      \PP\in \mathcal{M}_1(\mathcal{X})
      \,\middle|\,
      W_q(\PP,\widehat{\PP}(x,a)) \le \varepsilon
   \right\},
\end{equation}
i.e., the closed $q$-Wasserstein ball of radius $\varepsilon$ centered at some reference distribution $\widehat{\PP}(x,a) \in \mathcal{M}_1(\X)$. We impose the following  assumption on the reference kernel.
\begin{asu}\label{asu_wasserstein_q}~
The  map 
\begin{align*}
\X \times A &\rightarrow\mathcal{M}_1(\X)\\
(x,a) &\mapsto  \widehat{\PP}(x,a) 
\end{align*}
is assumed to be continuous in the Wasserstein-$q$ topology.
\end{asu}
Thus, for any state--action pair $(x,a)$, the set $\mathcal{P}^{(q,\varepsilon)}(x,a)$ consists of all probability measures on $\mathcal{X}$ that are within Wasserstein distance $\varepsilon$ of the reference measure $\widehat{\PP}(x,a)$. 

\begin{lem}\label{lem:q_wasserstein}
Let Assumption~\ref{asu_wasserstein_q} hold true.
Define for $n\geq1$ the set of probability measures on $\X$ with probabilities defined on an equidistant $1/n$-spaced grid via
\[
G^{(n)}:=\left\{\PP \in \mathcal{M}_1(\X)~\middle|~\PP(y)=\frac{k_y}{n}\text{ for some } k_y \in \{0,1,\dots,n\} \text{ for all } y \in \X \right\},
\]
and define for all $n\geq1$ the ambiguity set
\begin{equation}\label{eq:defn_Pn}
\mathcal{X} \times A \ni (x,a)\twoheadrightarrow \mathcal{P}^{(n)}(x,a):= \mathcal{P}^{(q,\varepsilon)}(x,a) \cap G^{(n)}= \left\{\PP\in G^{(n)}~\middle|~W_q(\PP,\widehat{\PP}(x,a)) \leq  \varepsilon \right\}.
\end{equation}
\begin{itemize}
\item[(i)]
Let $(x,a) \in \X \times A$, and $\PP \in \mathcal{P}^{(q,\varepsilon)}(x,a)$. Then, for all $p\geq1$ and for all $n\geq1$ there exists some $\PP^{(n)} \in G^{(n)}$ with 
\begin{equation}\label{eq:proof_C_n1}
W_p(\PP^{(n)}, \PP) \leq \frac{\max_{y,y' \in \X} \|y-y'\| \cdot |\X|^{1/p} }{n^{1/p}},
\end{equation}
for $|\X|$ denoting the number of states of $\X$.
\item[(ii)]
For all sufficiently large $n$, the sets $\mathcal P^{(n)}(x,a)$ are nonempty, and we define for $(x,a)\in\X\times A$
\begin{align*}
Q^{(n)^*}(x,a)&:= \inf_{\mathbb{P}\in \mathcal{P}^{(n)}(x,a)} \mathbb{E}_{\mathbb{P}}\left[ r(x,a,X_1) + \alpha V^{(n)}(X_1)\right], \\
Q^*(x,a)&:= \inf_{\mathbb{P}\in\mathcal{P}^{(q,\varepsilon)}(x,a)} \mathbb{E}_{\mathbb{P}}\left[ r(x,a,X_1) + \alpha V(X_1)\right],
\end{align*}
where $V^{(n)}$ denotes the value function as defined in \eqref{eq:valuefunc} associated to the finite ambiguity set $\mathcal{P}^{(n)}(x,a)$ defined in \eqref{eq:defn_Pn}. Then, there exists some $n_0$ such that we have for all $ n \geq n_0$
\begin{align*}
\sup_{(x,a) \in \X \times A} |Q^*(x,a)-Q^{(n)^*}(x,a)| \leq \frac{C}{ n^{1/q}},
\end{align*}
for 
\[
C:=\frac{(L_r+\alpha L_V)(q+1)\max_{y,y'\in\X}\|y-y'\|\,|\X|^{1/q}}{1-\alpha}<\infty,
\]
and where \[
L_r:=\max_{(x,a)\in\X\times A}\max_{\substack{y,y'\in\X\\y\neq y'}}\frac{|r(x,a,y)-r(x,a,y')|}{\|y-y'\|}<\infty,
\qquad
L_V:=\max_{\substack{y,y'\in\X\\y\neq y'}}\frac{|V(y)-V(y')|}{\|y-y'\|}<\infty.
\]
\end{itemize}
\end{lem}
For a singleton state space, set $L_r=L_V=0$. The bound quantifies the required grid resolution, but the full grid contains $\binom{n+m-1}{m-1}$ distributions, where $m=|\X|$. Testing membership in the Wasserstein ball generally requires a transport optimization with $m^2$ variables for each candidate. Thus this construction proves approximability and does not establish an efficient replacement for specialized Wasserstein solvers in high dimensions. The results in this appendix concern finite state spaces; they do not supply the sampling or function-approximation guarantees discussed in Appendix~\ref{sec:continuous_state}.

\subsection{Approximation of parametric ambiguity sets}
We now turn to parametric uncertainty, i.e., to the case where one specifies a class of potential conditional transition kernels up to the choice of the model parameter $\theta$ which is subject to ambiguity. We define for each state-action pair $(x,a)\in \mathcal{X}\times A$ the ambiguity set  
\begin{equation}\label{eq_def_P1_parametric}
\mathcal{P}^{\Theta}(x,a)
:= \left\{
      \widehat{\PP}(x,a,\theta)\in \mathcal{M}_1(\mathcal{X})
      \,\middle|\,
      \theta \in \Theta(x,a)
   \right\},
\end{equation}
for some state-action dependent parameter set $\Theta(x,a) \subset \R^{d_\theta}$  and some parameterized probability measure $\widehat{\PP}(x,a,\theta)\in \mathcal{M}_1(\mathcal{X})$. 
\begin{lem}\label{lem:q_parametric}
Let $\Theta(x,a)\subset\R^{d_\theta}$ be nonempty and compact for all $(x,a) \in \X \times A$. Assume also that for all $(x,a) \in \X \times A$ the map $\Theta(x,a) \ni \theta \mapsto \widehat{\PP}(x,a,\theta)$ is continuous in the Wasserstein-1 topology.
Then for all $n\in \N$ there exists a set-valued map 
\[
\X \times A \ni (x,a) \twoheadrightarrow \Theta^{(n)}(x,a)\subset \Theta(x,a)
\]
which contains finitely many elements, 
based on which we define for all $n\geq1$ the ambiguity set
\begin{equation}\label{eq:defn_Pn2}
\mathcal{X} \times A \ni (x,a)\twoheadrightarrow \mathcal{P}^{(n)}(x,a):= \mathcal{P}^{\Theta^{(n)}}(x,a) = \left\{
      \widehat{\PP}(x,a,\theta)\in \mathcal{M}_1(\mathcal{X})
      \,\middle|\,
      \theta \in \Theta^{(n)}(x,a)
   \right\},
\end{equation}
such that the assumptions of Lemma~\ref{lem:Q_convergence} are fulfilled, and we have
\begin{align*}
\lim_{n \rightarrow \infty} Q^{(n)^*}(x,a)&:=\lim_{n \rightarrow \infty} \inf_{\mathbb{P}\in \mathcal{P}^{(n)}(x,a)} \mathbb{E}_{\mathbb{P}}\left[ r(x,a,X_1) + \alpha V^{(n)}(X_1)\right] \\
&= \inf_{\mathbb{P}\in\mathcal{P}^{\Theta}(x,a)} \mathbb{E}_{\mathbb{P}}\left[ r(x,a,X_1) + \alpha V(X_1)\right]=:Q^*(x,a),
\end{align*}
where $V^{(n)}$ denotes the value function as defined in \eqref{eq:valuefunc} associated to the finite ambiguity set $\mathcal{P}^{(n)}(x,a)$ defined in \eqref{eq:defn_Pn2}.
\end{lem}

\begin{exa}\label{exa:binomial_approximation}
Consider the state space $\X = \{0,\dots,N\}$ for some $N \in \N$. Then, for any two functions $\underline{p}:\X \rightarrow \R_+, \overline{p}:\X \rightarrow \R_+$ with $0\leq\underline{p}(\cdot)\leq\overline{p}(\cdot)\leq1$, the ambiguity set of binomial distributions 
\[
  \X \times A \ni (x,a) \twoheadrightarrow \mathcal{P}(x,a):= \{ \operatorname{Bin}(N,p), p \in [\underline{p}(x), ~\overline{p}(x)]\}
\]
can be approximated in the sense of Lemma~\ref{lem:q_parametric} by
\[
\mathcal{P}^{(n)}(x,a):= \left\{\operatorname{Bin}(N,p), p \in \{p_0,\dots,p_n\}\text{ with } p_i = \underline{p}(x) + \frac{i}{n} \left(\overline{p}(x)-\underline{p}(x)\right),~~ i =0,\dots,n\right\}.
\]
The proof is reported in Appendix~\ref{subsec:proof_approximation}. This means, by Lemma~\ref{lem:q_parametric}, to approximate the $Q$-value function associated to ambiguity set $\mathcal{P}(x,a)$, we can compute the $Q$-value function associated to  ambiguity set $\mathcal{P}^{(n)}(x,a)$ using Algorithm~\ref{alg:Q-L} for a sufficiently large $n$.
\end{exa}
Under the additional assumption of Lipschitz-continuity of the parameter $\theta$ in the Wasserstein-distance (see Appendix~\ref{sec:wasserstein}), we are able to also derive explicit approximation rates allowing in particular to determine the needed granularity of the covering radius of $\Theta^{(n)}$ for a fixed precision.
\begin{cor}\label{cor:parametric}
Let $(x,a)\in\X\times A$, and assume that $\Theta(x,a)\subset\R^{d_\theta}$ is nonempty and compact for all $(x,a)$ and that there exists
a constant $L_\theta<\infty$ such that the parametrization is uniformly Lipschitz in $W_1$, i.e.
\begin{equation}
\label{eq:W1-Lipschitz-param}
W_1\!\big(\widehat{\PP}(x,a,\theta),\widehat{\PP}(x,a,\theta')\big)\le L_\theta\|\theta-\theta'\|
\qquad\text{for all }(x,a)\in\X\times A,\;\theta,\theta'\in\Theta(x,a).
\end{equation}
Let $\Theta^{(n)}(x,a)\subset\Theta(x,a)$ be any nonempty finite subset and define the finite ambiguity set
\[
\mathcal{P}^{(n)}(x,a):=\mathcal{P}^{\Theta^{(n)}}(x,a)=\{\widehat{\PP}(x,a,\theta)\,:\,\theta\in\Theta^{(n)}(x,a)\}.
\]
Define the \emph{covering radius}
\begin{equation}
\label{eq:covering-radius}
\varepsilon_n:=\sup_{(x,a)\in\X\times A}\;\sup_{\theta\in\Theta(x,a)}\;\inf_{\theta'\in\Theta^{(n)}(x,a)}\|\theta-\theta'\|.
\end{equation}

Then, for every $n\geq1$,
\begin{equation}
\label{eq:parametric-rate-Q}
\sup_{(x,a)\in\X\times A}\big|Q^\ast(x,a)- Q^{(n)^*}(x,a)\big|
\;\le\;
\frac{(L_r+\alpha L_V)\,L_\theta}{1-\alpha}\;\varepsilon_n
\end{equation}
for 
\[
L_r:=\max_{(x,a)\in\X\times A}\max_{\substack{y,y'\in\X\\y\neq y'}}\frac{|r(x,a,y)-r(x,a,y')|}{\|y-y'\|}<\infty,
\qquad
L_V:=\max_{\substack{y,y'\in\X\\y\neq y'}}\frac{|V(y)-V(y')|}{\|y-y'\|}<\infty.
\]

For a singleton state space, set $L_r=L_V=0$.
\end{cor}



}

\subsection{Proofs of the approximation results}\label{subsec:proof_approximation}

{
\begin{proof}[Proof of Lemma~\ref{lem:Q_convergence}]
Let $\mathcal H^{(n)}$ be the Bellman operator associated with $\mathcal P^{(n)}$, and let $\mathcal H$ correspond to $\mathcal P$. Fix $(x,a)\in\X\times A$. Compactness of $\mathcal P(x,a)$ and finiteness of $\X$ imply that a minimizer $\mathbb P^*\in\mathcal P(x,a)$ exists for the expectation of
\[
g_{x,a}(y):=r(x,a,y)+\alpha\max_{b\in A}Q^*(y,b).
\]
Choose the approximating measures $\mathbb P^{(n)}\in\mathcal P^{(n)}(x,a)$ from the assumption. Since $g_{x,a}$ is bounded and continuous on the finite state space, weak convergence gives
\[
0\leq\mathcal H^{(n)}Q^*(x,a)-\mathcal HQ^*(x,a)
\leq\mathbb E_{\mathbb P^{(n)}}[g_{x,a}]-\mathbb E_{\mathbb P^*}[g_{x,a}]
\longrightarrow0.
\]
The first inequality uses $\mathcal P^{(n)}(x,a)\subseteq\mathcal P(x,a)$. As $\X\times A$ is finite, convergence is uniform. Both Bellman operators are $\alpha$-contractions, and their fixed points therefore satisfy
\[
\|Q^{(n)^*}-Q^*\|_\infty
\leq\alpha\|Q^{(n)^*}-Q^*\|_\infty+
\|\mathcal H^{(n)}Q^*-\mathcal HQ^*\|_\infty.
\]
Dividing by $1-\alpha$ proves the assertion.
\end{proof}

\begin{proof}[Proof of Lemma~\ref{lem:q_wasserstein}]
Put $D:=\max_{y,y'\in\X}\|y-y'\|$ and $m:=|\X|$. If $D=0$, the state space is a singleton and the assertions are immediate. Thus assume $D>0$.
\begin{itemize}
\item[(i)]
For $\mathbb P\in\mathcal M_1(\X)$, round $n\mathbb P(y)$ down to its integer part and then add one to $n-\sum_y\lfloor n\mathbb P(y)\rfloor$ coordinates with the largest fractional parts. The resulting integers $k_y$ sum to $n$ and satisfy $|k_y/n-\mathbb P(y)|\leq1/n$. Hence the law $\mathbb P^{(n)}(y)=k_y/n$ belongs to $G^{(n)}$ and
\[
\|\mathbb P^{(n)}-\mathbb P\|_1\leq\frac{m}{n}.
\]
Couple the common mass of the two measures at identical states, and transport the remaining mass over distances at most $D$. This gives, for every $p\geq1$,
\[
W_p(\mathbb P^{(n)},\mathbb P)^p
\leq\frac{D^p}{2}\|\mathbb P^{(n)}-\mathbb P\|_1
\leq\frac{D^p m}{n},
\]
which proves \eqref{eq:proof_C_n1}.

\item[(ii)]
Write $d_n:=Dm^{1/q}n^{-1/q}$, and choose $n_0$ such that $d_n<\varepsilon$ for every $n\geq n_0$. Fix such an $n$, a state-action pair $(x,a)$, and any $\mathbb P\in\mathcal P^{(q,\varepsilon)}(x,a)$. Define
\[
\lambda_n:=1-(1-d_n/\varepsilon)^q,\qquad
\mathbb R_n:=(1-\lambda_n)\mathbb P+\lambda_n\widehat{\mathbb P}(x,a).
\]
Convexity of the transport cost $W_q^q$ in its first marginal implies
\[
W_q(\mathbb R_n,\widehat{\mathbb P}(x,a))
\leq(1-\lambda_n)^{1/q}W_q(\mathbb P,\widehat{\mathbb P}(x,a))
\leq\varepsilon-d_n.
\]
By part~(i), choose $\mathbb P_n\in G^{(n)}$ with $W_q(\mathbb P_n,\mathbb R_n)\leq d_n$. The triangle inequality now gives $W_q(\mathbb P_n,\widehat{\mathbb P}(x,a))\leq\varepsilon$, so $\mathbb P_n\in\mathcal P^{(n)}(x,a)$. In particular, these finite ambiguity sets are nonempty for all $n\geq n_0$.
Moreover, $\lambda_n\leq qd_n/\varepsilon$ and $W_1\leq W_q$ give
\begin{align*}
W_1(\mathbb P,\mathbb P_n)
&\leq W_1(\mathbb P,\mathbb R_n)+W_1(\mathbb R_n,\mathbb P_n)\\
&\leq\lambda_n W_1(\mathbb P,\widehat{\mathbb P}(x,a))+d_n
\leq(q+1)d_n.
\end{align*}
The second inequality follows by mixing a coupling between $\mathbb P$ and $\widehat{\mathbb P}(x,a)$ with the diagonal coupling of $\mathbb P$.

Let $\mathcal H^{(n)}$ and $\mathcal H$ denote the corresponding Bellman operators. Apply the preceding construction to a minimizing measure for $\mathcal HQ^*(x,a)$. Since $y\mapsto r(x,a,y)+\alpha V(y)$ is $(L_r+\alpha L_V)$-Lipschitz, the coupling inequality for $W_1$ yields
\[
0\leq\mathcal H^{(n)}Q^*(x,a)-\mathcal HQ^*(x,a)
\leq(L_r+\alpha L_V)(q+1)d_n.
\]
The contraction argument gives
\[
\|Q^{(n)^*}-Q^*\|_\infty
\leq\frac{\|\mathcal H^{(n)}Q^*-\mathcal HQ^*\|_\infty}{1-\alpha}
\leq\frac{(L_r+\alpha L_V)(q+1)Dm^{1/q}}{(1-\alpha)n^{1/q}},
\]
as asserted.
\end{itemize}
\end{proof}

\begin{proof}[Proof of Lemma~\ref{lem:q_parametric}]
Let $(x,a) \in \X \times A$. Let $n \in \N$ and consider the open cover of $\Theta(x,a)$.
$$\left\{B_{1/n}(\theta)~|~\theta \in \Theta(x,a)\right\} \supseteq \Theta(x,a),$$
where $B_{1/n}(\theta)$ denotes (open) balls centered at $\theta$ with radius $1/n$.
By compactness, there are $\theta_1^{(n)},\ldots,\theta_{m_n}^{(n)}\in\Theta(x,a)$ such that
\[
\Theta(x,a)\subseteq\bigcup_{k=1}^{m_n}B_{1/n}(\theta_k^{(n)}).
\]
Then, we define
\[
\Theta^{(n)}(x,a):= \{\theta_1^{(n)}, \dots, \theta_{m_n}^{(n)}\} \subseteq \Theta(x,a)
\]
and $\mathcal{P}^{(n)}(x,a):= \mathcal{P}^{\Theta^{(n)}}(x,a)$. To show that Lemma~\ref{lem:Q_convergence} applies, we first note that Assumption~\ref{asu_p} is implied by the assumed continuity of  $\theta \mapsto \widehat{\PP}(x,a, \theta)$, see \citep[Proposition 3.2]{neufeld2023markov}.  We pick some $\PP \in \mathcal{P}(x,a)$ which, by definition of $\mathcal{P}$, can be represented as $\PP = \widehat{\PP}(x,a, \theta)$ for some $\theta \in \Theta(x,a)$. By construction of $\Theta^{(n)}$, there exists some $\theta^{(n)} \in \Theta^{(n)}(x,a)$ with $\|\theta-\theta^{(n)}\|\leq1/n$, and the assertion of  Lemma~\ref{lem:Q_convergence} follows by defining $\PP^{(n)}=\widehat{\PP}(x,a, \theta^{(n)})$ and since $\widehat{\PP}(x,a, \theta^{(n)}) \rightarrow \widehat{\PP}(x,a, \theta)=\PP$ in the Wasserstein-1 topology as $n \rightarrow \infty$ with the assumed continuity of  $\theta \mapsto \widehat{\PP}(x,a, \theta)$.
\end{proof}

\begin{proof}[Proof of Example~\ref{exa:binomial_approximation}]
For every $(x,a)$, the interval $[\underline p(x),\overline p(x)]$ is nonempty and compact. Moreover, for every $y\in\{0,\dots,N\}$ the function
\[
p\longmapsto\operatorname{Bin}(N,p)(y)=\binom Ny p^y(1-p)^{N-y}
\]
is continuous on $[0,1]$, with the usual endpoint conventions. Thus $p\mapsto\operatorname{Bin}(N,p)$ is continuous in the Wasserstein-$1$ topology on the finite state space. For any $p\in[\underline p(x),\overline p(x)]$, a nearest point $p_{i_n}$ of the stated grid satisfies
\[
|p-p_{i_n}|\leq\frac{\overline p(x)-\underline p(x)}{n}\longrightarrow0.
\]
Consequently $\operatorname{Bin}(N,p_{i_n})\to\operatorname{Bin}(N,p)$ in Wasserstein-$1$ distance, and Lemma~\ref{lem:q_parametric} applies.
\end{proof}
}
\begin{proof}[Proof of Corollary~\ref{cor:parametric}]
Let $\mathcal H^{(n)}$ and $\mathcal H$ denote the Bellman operators associated with $\mathcal P^{(n)}$ and $\mathcal P^{\Theta}$, respectively. Fix $(x,a)$. By compactness and continuity, there is a minimizing parameter $\theta^*\in\Theta(x,a)$ for $\mathcal HQ^*(x,a)$. Choose $\theta^{(n)}\in\Theta^{(n)}(x,a)$ with $\|\theta^{(n)}-\theta^*\|\leq\varepsilon_n$. Since $y\mapsto r(x,a,y)+\alpha V(y)$ is $(L_r+\alpha L_V)$-Lipschitz, the coupling inequality for $W_1$ gives
\begin{align*}
0&\leq\mathcal H^{(n)}Q^*(x,a)-\mathcal HQ^*(x,a)\\
&\leq(L_r+\alpha L_V)
W_1\bigl(\widehat{\mathbb P}(x,a,\theta^{(n)}),\widehat{\mathbb P}(x,a,\theta^*)\bigr)\\
&\leq(L_r+\alpha L_V)L_\theta\varepsilon_n.
\end{align*}
Here the first inequality follows from $\mathcal P^{(n)}(x,a)\subseteq\mathcal P^{\Theta}(x,a)$. As both Bellman operators are $\alpha$-contractions,
\[
\|Q^{(n)^*}-Q^*\|_\infty
\leq\frac{\|\mathcal H^{(n)}Q^*-\mathcal HQ^*\|_\infty}{1-\alpha}
\leq\frac{(L_r+\alpha L_V)L_\theta}{1-\alpha}\varepsilon_n.
\]
\end{proof}

\section{Extension to continuous state spaces}
\label{sec:continuous_state}

In many applications of interest the state space is naturally continuous and the assumption of a finite
state space imposed in Section~\ref{sec:setting} is restrictive. In this section we outline how the
robust Bellman formulation can be used on continuous state spaces
with a separate fitted $Q$-iteration procedure, in the spirit of the classical deep $Q$-learning literature (\citealp{baird1995residual}, \citealp{melo2008analysis},  \citealp{mnih2015human} and \citealp{van2016deep}).

\subsection{Robust Markov decision problems on continuous domains}
\label{subsec:continuous-setting}

Let $\X\subset\R^d$ be a compact state space equipped with its Borel
$\sigma$-algebra, and let $A$ be a finite action set. We consider a bounded Borel-measurable reward function
$r:\X\times A\times\X\to\R$ and a discount factor $\alpha\in(0,1)$.
As in Section~\ref{sec:setting}, we assume an $(x,a)$-rectangular ambiguity structure; however, in
contrast to Section~\ref{sec:setting} we now allow $\X$ to be continuous.

For each state-action pair $(x,a)\in\X\times A$ we are given a finite set of candidate transition
kernels
\begin{equation}
\label{eq:continuous-finite-ambiguity}
\Pc(x,a)=\Big\{\PP^{(1)}(x,a),\dots,\PP^{(N)}(x,a)\Big\}\subset\mathcal M_1(\X).
\end{equation}
We interpret $\Pc(x,a)$ as the set of admissible conditional distributions of the next state $X_1$
given $(X_0,A_0)=(x,a)$ and consider the corresponding robust control objective (over an infinite
time horizon) under the worst-case model chosen from these kernels.

Assume each $\PP^{(k)}$ is a Borel transition kernel. Given a bounded Borel-measurable function $Q:\X\times A\to\R$, define the robust Bellman operator
$\mathcal H$ by
\begin{equation}
\label{eq:robust-bellman-operator-cont}
(\mathcal H Q)(x,a)
:=
\min_{k\in\{1,\dots,N\}}
\EE_{\PP^{(k)}(x,a)}
\Big[
r(x,a,X_1)+\alpha\max_{b\in A}Q(X_1,b)
\Big].
\end{equation}
The robust optimal $Q$-function $Q^\ast$ is characterized as the unique bounded Borel fixed point
$Q^\ast=\mathcal H Q^\ast$ and the robust value function satisfies $V^\ast(x)=\max_{a\in A}Q^\ast(x,a)$.

\medskip
\begin{rem}
Under standard boundedness assumptions (see \citealp[Chapter 7]{bauerle2011markov}), $\mathcal H$ is a contraction w.r.t.\, $\|\cdot\|_\infty$
with modulus $\alpha$ (the proof is identical to the finite-state case), which motivates iterative
methods based on approximating $\mathcal H$. This contraction does not imply convergence after a nonlinear regression step.
\end{rem}
\subsection{Function approximation and a robust deep $Q$-learning objective}
\label{subsec:function-approx}

When $\X$ is continuous, tabular representations of $Q$ are infeasible and we employ a parametric
function class
\[
Q_\theta:\X\times A\to\R,\qquad \theta\in\R^p,
\]
for instance a neural network (\citealp{lecun2015deep}, \citealp{lu2025distributionally}). The goal is to fit $Q_\theta$ such that it approximately satisfies
the robust Bellman equation \eqref{eq:robust-bellman-operator-cont}.

\medskip

Similarly to Algorithm~\ref{alg:Q-L} , we assume access to $N$ simulators that, given $(x,a)$, produce independent successor samples $X_1^{(k)}\sim\PP^{(k)}(x,a)$ together with their observed deterministic rewards $\mathsf R_1^{(k)}=r(x,a,X_1^{(k)})$. Prior knowledge of $r$ or its evaluation away from sampled transitions is not required.

Let $\theta^-$ denote a (possibly delayed) target parameter. For a given $(x,a)$ define the robust
TD target
\begin{equation}
\label{eq:robust-td-target}
Y_{\theta^-}(x,a)
:=
\min_{k\in\{1,\dots,N\}}
\EE_{X_1\sim \PP^{(k)}(x,a)}
\Big[
r(x,a,X_1)+\alpha\max_{b\in A}Q_{\theta^-}(X_1,b)
\Big].
\end{equation}

In practice, the inner expectation is approximated using $M\in\N$ sampled transition--reward pairs, with $X_{1,j}^{(k)}\sim\PP^{(k)}(x,a)$ and observed $\mathsf R_{1,j}^{(k)}=r(x,a,X_{1,j}^{(k)})$:
\begin{equation}
\label{eq:robust-td-target-mc}
\widehat Y_{\theta^-}(x,a)
:=\min_{k\in\{1,\dots,N\}}\frac{1}{M}\sum_{j=1}^M
\Big[\mathsf R_{1,j}^{(k)}+\alpha\max_{b\in A}Q_{\theta^-}(X_{1,j}^{(k)},b)\Big].
\end{equation}

\medskip
Let $\mu$ be a sampling distribution on $\X\times A$ (e.g.\ induced by an exploration policy with
experience replay \citep{foerster2017stabilising}). For fixed target parameters, we propose to minimize the squared temporal-difference regression loss:
\begin{equation}
\label{eq:robust-dqn-loss}
{
\ \mathcal L(\theta)
:=
\EE_{(x,a)\sim\mu,\,\mathrm{MC}}
\Big[
\big(
Q_\theta(x,a)-\widehat Y_{\theta^-}(x,a)
\big)^2
\Big].
\ }
\end{equation}
Here the expectation also includes the Monte-Carlo samples used in the target. Stochastic gradient descent on \eqref{eq:robust-dqn-loss} yields a robust analogue of deep $Q$-learning (DQN)
for finite ambiguity sets, similar to the approach presented in \citet{lu2025distributionally} for Wasserstein ambiguity set.

\medskip
\begin{rem}[On the implementation]
\begin{itemize}

\item[(a)] As in standard DQN (\citealp{mnih2015human}, \citealp{van2016deep}), a slowly updated target parameter $\theta^-$
is commonly used to stabilize training; the robust modification affects only the target computation through the
$\min_{k=1,\dots,N}$ in \eqref{eq:robust-td-target-mc}.

\item[(b)] One may reduce overestimation bias by replacing the target
$\max_{b\in A}Q_{\theta^-}(X_{1,j}^{(k)},b)$ with a Double-DQN target, i.e.\ selecting the maximizing
action using the online network and evaluating it using the target network, see \citep{van2016deep}.

\item[(c)]  Relative to non-robust DQN (\citealp{mnih2015human}, \citealp{van2016deep}), the additional cost is linear in $N$, since for
each training pair $(x,a)$ one must evaluate $N$ Monte-Carlo averages and take their minimum.

\item[(d)]  {If transitions are collected from the real environment under a single unknown kernel, one may still evaluate \eqref{eq:robust-td-target-mc} using the available simulator for each candidate model $k$, provided it supplies both $X_{1,j}^{(k)}$ and its deterministic reward $\mathsf R_{1,j}^{(k)}$. Transition samples without the associated rewards are not sufficient under the present information assumption.}
\end{itemize}
\end{rem}

This fitted procedure differs from Algorithm~\ref{alg:Q-L}: it estimates continuation expectations directly using samples from every candidate, rather than estimating transition tables and updating along the empirically selected model. It can also be implemented with a tabular function class on finite state spaces, but it remains a different recursion. Because a minimum is concave, $\mathbb E[\widehat Y_{\theta^-}(x,a)]\leq Y_{\theta^-}(x,a)$ at finite $M$; increasing $M$ reduces this sampling effect. The approximation results of Appendix~\ref{sec:extensions} concern the ambiguity sets and do not control this bias or regression error. Establishing convergence and rates for the fitted procedure, particularly with neural networks, is beyond the scope of this paper. {Both implementations require transition--reward observations from each candidate simulator, rather than prior knowledge of the reward function or only a trajectory from one unknown environment.}

\section{Proofs and auxiliary results for the main text}\label{sec:proofs}
\subsection{Auxiliary results}\label{subsec:aux}
We start this appendix by introducing necessary notation. For any function $f : \X \times A \rightarrow  \R$, we write
\begin{equation}
    \|f \|_{\infty} := \sup_{x\in \X}\sup_{a\in A} |f (x, a)|,
\end{equation}
and we define an estimation of the ambiguity set after $t$ steps of Algorithm~\ref{alg:Q-L} via the set-valued map
\begin{equation}\label{eq:ambiguityset_3}
\begin{aligned}
     \X \times A \ni   (x,a) & \twoheadrightarrow \mathcal{P}_t(x,a):=\left\{\mathbb{P}_t^{(1)}(x,a),\cdots ,\mathbb{P}_t^{(N)}(x,a)\right\} \subset \mathcal{M}_1(\X).
\end{aligned}
\end{equation}
where $\PP_t^{(k)}$ is defined in \eqref{eq:sample_pt}.
Next, we introduce the operators $\mathcal{H}_t$ and $\mathcal{H}$ operating on a function $q:\X \times A \rightarrow \R$ being defined by
\begin{equation}\label{eq:opH}
\X\times A \ni (x,a) \mapsto \mathcal{H}_tq(x,a) := \min_{\mathbb{P}\in\mathcal{P}_t(x,a)} \mathbb{E}_{\mathbb{P}}\left[r(x,a,X_1) + \alpha \max_{b\in A}q(X_1,b)\right]\in \R,
\end{equation}
as well as 
\begin{equation}\label{eq:opH2}
\X\times A \ni (x,a) \mapsto \mathcal{H} q(x,a) := \min_{\mathbb{P}\in\mathcal{P}(x,a)} \mathbb{E}_{\mathbb{P}}\left[r(x,a,X_1) + \alpha \max_{b\in A}q(X_1,b)\right]\in \R.
\end{equation}

Here $\mathcal H_t$ retains the true reward function for analysis, and differs from the observable operator $\widehat{\mathcal H}_t$ in \eqref{eq:observable_empirical_bellman} only at unvisited pairs. Identity~\eqref{eq:empirical_reward_identity} implies $\widehat{\mathcal H}_tq(x,a)=\mathcal H_tq(x,a)$ whenever $N_t(x,a)>0$. At an unvisited pair, their candidate scores differ by the common term $\sum_y\nu(x,a)(y)r(x,a,y)$, so they still have the same minimizing candidate indices. Neither $\mathcal H_t$ nor the true reward function must be evaluated by the learner.

Moreover, analogously to \eqref{eq:Qvalue}, we define the optimal $Q$-value function with respect to the ambiguity set $\mathcal{P}_t$ via
\begin{equation}\label{eq:Qtvalue}
    \X \times A \ni (x,a)  \mapsto Q_t^*(x,a) := \min_{\mathbb{P}\in \mathcal{P}_t(x,a)} \mathbb{E}_{\mathbb{P}}\left[ r(x,a,X_1) + \alpha V(X_1)\right],
\end{equation}
where - by slight abuse of notation - $V$ is defined as in \eqref{eq:valuefunc} with using the ambiguity set $\mathcal{P}_t$.
\begin{lem}\label{lem:HQstar}
Let $0 < \alpha<1$, and let the ambiguity set $\mathcal{P}_t$ be defined in \eqref{eq:ambiguityset_3}. Then the following fixed-point equations hold true for the optimal $Q$-value functions defined in \eqref{eq:Qvalue} and \eqref{eq:Qtvalue}, respectively
\begin{align*}
\mathcal{H}_tQ_t^*(x, a) &= Q_t^*(x, a) \text{ for all } (x, a) \in \X \times A, \text{ for all }t \in \N_0, \\
\mathcal{H}Q^*(x, a) &= Q^*(x, a) \text{ for all } (x, a) \in \X \times A.
\end{align*}

\end{lem}

\begin{proof}
This follows directly by definition of $Q_t^*$ and by Proposition \ref{prop:DPP}, compare also \citep[Proof of Lemma 17]{neufeld2022robust}.
%
\end{proof}

\begin{lem}\label{lem:mapsq}
For any maps $q_i:\X\times A \rightarrow \mathbb{R},\ i=1,2,$ we have: 
\begin{align*}
\|\mathcal{H}_tq_1-\mathcal{H}_tq_2\|_{\infty} &\leq \alpha\|q_1-q_2\|_{\infty} \text{ for all } t \in \N_0, \\
\|\mathcal{H}q_1-\mathcal{H}q_2\|_{\infty} &\leq \alpha\|q_1-q_2\|_{\infty}.
\end{align*}
\end{lem}

\begin{proof}
The proof follows analogously to  \citep[Proof of Lemma 18]{neufeld2022robust}.
\end{proof}
Next, recall the definition for $N_t(x,a)>0$ (and retain the initialized probability measure when $N_t(x,a)=0$)     
\begin{equation*}
    \mathbb{P}_t^{(k)}(x, a)(y) = \frac{\sum_{s=0}^{t-1}\one_{\{(X_s,a_s(X_s),X_{s+1}^{(k)})=(x,a,y)\}}}{\sum_{s=0}^{t-1} \one_{\{(X_s,a_s(X_s))=(x,a)\}}}.
    \end{equation*}
\begin{lem}\label{lem:P_t}
Let $x_0\in\X$ and a behavior rule $\mathbf a$ be such that
    \begin{equation}\label{eq:gammat_lem}
    \sum^{\infty}_{t=0}\one_{\{(X_t,a_t(X_t))=(x,a)\}} = \infty, \text{ for all } (x,a)\in \X \times A,\ \mathbb{P}_{x_0,\textbf{a}}\ -\text{almost\ surely}.
    \end{equation}
 Then, we have that 
\[
\mathbb{P}_t^{(k)}(x,a) \rightarrow \mathbb{P}^{(k)}(x,a)  \text{ for all } (x,a)\in \X \times A, k \in \{1,\dots,N\} \qquad \mathbb{P}_{x_0,\textbf{a}}\ -\text{almost\ surely}.
\]
as $t \rightarrow \infty$. 
\end{lem}
\begin{proof}
Let $k \in \{1,\dots,N\}$ and let $(x,a) \in \X \times A$. First note that the condition in \eqref{eq:gammat_lem} ensures that 
\begin{equation}\label{eq:proof_step1}
N_t(x,a):=\sum_{s=0}^{t-1} \one_{\{(X_s,a_s(X_s))=(x,a)\}} \rightarrow \infty ~~ \mathbb{P}_{x_0,\textbf{a}}\ -\text{almost\ surely as }  t \rightarrow \infty.
\end{equation}
Let $\tau_j(x,a)$ be the time of the $j$-th visit of $\left((X_t,a_t(X_t)\right)_{t\in \N_0}$ to $(x,a)$. Then, we draw $Y_j^{(k)}:=X_{\tau_j(x,a)+1}^{(k)} \sim \mathbb{P}^{(k)}(x,a)$. By construction $(Y_j^{(k)})_{j \in \N}$ is an i.i.d. sequence of random variables with law $\mathbb{P}^{(k)}(x,a)$. For a fixed next state $y \in \X$, define 
\[
\xi_j^{(y)}:= \one_{\{Y_j^{(k)}=y\}},
\]
and, for $n \in \N$, let $S_n^{(y)}:= \sum_{j=1}^n \xi_j^{(y)}$.
By the strong law of large numbers, we get
\begin{equation}\label{eq:proof_step2}
\frac{S_n^{(y)}}{n} \rightarrow \E\left[\xi_j^{(y)}\right] =\mathbb{P}^{(k)}(x,a)(y)~~\qquad \mathbb{P}_{x_0,\textbf{a}}-\text{almost\ surely as } n \rightarrow \infty.
\end{equation}
Now, note that, whenever $N_t(x,a)>0$,
\[
\mathbb{P}_t^{(k)}(x, a)(y)  = \frac{\sum_{s=0}^{t-1}\one_{\{(X_s,a_s(X_s),X_{s+1}^{(k)})=(x,a,y)\}}}{\sum_{s=0}^{t-1} \one_{\{(X_s,a_s(X_s))=(x,a)\}}} = \frac{S_{N_t(x,a)}^{(y)}}{{N_t(x,a)}} 
\]
and the assertion follows with \eqref{eq:proof_step1} and \eqref{eq:proof_step2}.
\end{proof}

\begin{lem}\label{lem:Q_t}
Let $0 < \alpha < 1$, and let $x_0\in\X$ and a behavior rule $\mathbf a$ be such that
    \begin{equation}\label{eq:gammat_lem2}
    \sum^{\infty}_{t=0}\one_{\{(X_t,a_t(X_t))=(x,a)\}} = \infty, \text{ for all } (x,a)\in \X \times A,\ \mathbb{P}_{x_0,\textbf{a}}\ -\text{almost\ surely}.
    \end{equation}
 Then, we have that 
 \[
\|Q_t^*-Q^*\|_{\infty} \leq \frac{\|\mathcal{H}_tQ^*-\mathcal{H}Q^*\|_{\infty}}{1-\alpha} \qquad \mathbb{P}_{x_0,\textbf{a}}\ -\text{almost\ surely}.
\]
which implies
\[
{Q}_t^{*}(x,a) \rightarrow Q^{*}(x,a)  \text{ for all } (x,a)\in \X \times A,  \qquad \mathbb{P}_{x_0,\textbf{a}}\ -\text{almost\ surely}.
\]
as $t \rightarrow \infty$. 
\end{lem}
\begin{proof}

We define $C_r:=\max_{(x,a,x')\in \X \times A \times \X} |r(x,a,x')| < \infty$, and then observe that for all $x\in \X$:
\[
V(x) = \sup_{\mathbf{a} \in \mathcal{A}}\inf_{\mathbb{P}\in \mathfrak{P}_{x,\textbf{a}}} \mathbb{E}_{\mathbb{P}}\left[ \sum^{\infty}_{t=0}\alpha^t \cdot r(X_t,a_t(X_t),X_{t+1})\right] \leq  \sum^{\infty}_{t=0}\alpha^t \cdot C_r= \frac{C_r}{1-\alpha}.
\]
By Proposition~\ref{prop:DPP},  for all $(x,a) \in \X \times A$ we then obtain 
$$
Q^*(x,a) \leq \max_{a \in A} Q^*(x,a) =V(x) \leq \frac{C_r}{1-\alpha}
$$
Applying the same bound to the negative rewards gives $\|Q^*\|_\infty\leq C_r/(1-\alpha)<\infty$.
Next, for all  $(x,a) \in \X \times A$ we compute

\begin{equation}\label{eq:htqminushq}
\begin{aligned}
&|\mathcal{H}_tQ^*(x,a)-\mathcal{H}Q^*(x,a)| \\
&=\left|\min_{\mathbb{P}\in\mathcal{P}_t(x,a)} \mathbb{E}_{\mathbb{P}}\left[r(x,a,X_1) + \alpha \max_{b\in A}Q^*(X_1,b)\right]-\min_{\mathbb{P}\in\mathcal{P}(x,a)} \mathbb{E}_{\mathbb{P}}\left[r(x,a,X_1) + \alpha \max_{b\in A}Q^*(X_1,b)\right]\right| \\
&=\bigg|\min_{k=1,\dots,N} \sum_{y \in \X} \left(r(x,a,y) + \alpha \max_{b\in A}Q^*(y,b)\right)\PP_t^{(k)}(x,a)(y)\\
&\hspace{6cm}-\min_{k=1,\dots,N} \sum_{y \in \X} \left(r(x,a,y) + \alpha \max_{b\in A}Q^*(y,b)\right)\PP^{(k)}(x,a)(y)\bigg|\\
&\leq \max_{k=1,\dots,N} \bigg| \sum_{y \in \X} \left(r(x,a,y) + \alpha \max_{b\in A}Q^*(y,b)\right)\left(\PP_t^{(k)}(x,a)(y)-\PP^{(k)}(x,a)(y)\right)\bigg|\\
&\leq  \left(C_r+ \alpha \|Q^*\|_{\infty}\right)\max_{k=1,\dots,N}\sum_{y\in\X}\left|\PP_t^{(k)}(x,a)(y)-\PP^{(k)}(x,a)(y)\right| \rightarrow 0 \text{ almost surely, as } t \rightarrow \infty
\end{aligned}
\end{equation}
where the convergence follows by Lemma~\ref{lem:P_t} and since $\|Q^*\|_{\infty} < \infty$.
Next, we get by using Lemma~\ref{lem:HQstar} and Lemma~\ref{lem:mapsq} that
\begin{align*}
\|Q_t^*-Q^*\|_{\infty}&=\|\mathcal{H}_tQ_t^*-\mathcal{H}Q^*\|_{\infty}\\
&=\|\mathcal{H}_tQ_t^*-\mathcal{H}_tQ^*+\mathcal{H}_tQ^*-\mathcal{H}Q^*\|_{\infty}\\
&\leq \alpha \|Q_t^*-Q^*\|_{\infty}+\|\mathcal{H}_tQ^*-\mathcal{H}Q^*\|_{\infty}
\end{align*}
implying 
\[
\|Q_t^*-Q^*\|_{\infty} \leq \frac{\|\mathcal{H}_tQ^*-\mathcal{H}Q^*\|_{\infty}}{1-\alpha}
\]
which by \eqref{eq:htqminushq} converges against $0$ almost surely as $t \rightarrow \infty$.
\end{proof}

The following lemma is a result from
stochastic approximation theory compare also
\citep[Lemma 1]{singh2000convergence}, \citep{dvoretzky1956stochastic} \citep[Theorem 1]{jaakkola1994convergence}, \citep[Lemma 12]{szepesvari1996generalized}, and  \citep[Lemma 3]{van2007convergence} for further reference.
\begin{lem}[\citep{singh2000convergence}, Lemma 1] \label{lem_convergence}
Let $\PP_0$ be a probability measure on an arbitrary probability space $(\Omega_0,\mathcal F_0)$, and consider a family of stochastic processes $(\gamma_t(x,a),F_t(x,a),\Delta_t(x,a))_{t\in \N_0}$, $(x,a)\in \X \times A$, satisfying for all $t \in \N_0$
\[
\Delta_{t+1}(x,a)=\left(1-\gamma_t(x,a)\right){\Delta}_t(x,a)+\gamma_t(x,a)F_t(x,a) \qquad \PP_0\text{-almost surely for all } (x,a) \in \X \times A.
\]
Let $(\mathcal{G}_t)_{t\in \N_0}\subseteq \mathcal{F}_0$ be a sequence of increasing $\sigma$-algebras such that for all $(x,a) \in \X \times A$ the random variables $\Delta_0(x,a)$ and $\gamma_0(x,a)$ are $\mathcal{G}_0$-measurable and such that
$\Delta_t(x,a)$, $\gamma_t(x,a)$, and $F_{t-1}(x,a)$ are $\mathcal{G}_t$-measurable for all $t\in \N$.
Further assume that the following conditions hold.
\begin{itemize}
\item[(i)] $0 \leq \gamma_t(x,a ) \leq 1$, $\sum_{t=0}^\infty \gamma_t(x,a )= \infty$, $\sum_{t=0}^\infty \gamma_t^2(x,a) <\infty$ $\PP_0$-almost surely for all $(x,a) \in \X \times A$, $ t\in \N_0$.
\item[(ii)] There exists $\delta \in (0,1)$ such that $\|\E_{\PP_0}\left[F_t (\cdot ,\ \cdot )~\middle|~\mathcal{G}_t\right]\|_\infty\leq \delta \|\Delta_t\|_\infty+c_t$ ~$\PP_0$-almost surely for all $t\in \N_0$, where $c_t$ converges to zero ~$\PP_0$-almost surely.
\item[(iii)] There exists $C>0$ such that $ \left\|\operatorname{Var}_{\PP_0}\left(F_t(\cdot ,\ \cdot  )~\middle|~\mathcal{G}_t\right) \right\|_{\infty}\leq C(1+\|\Delta_t\|_\infty)^2$ $\PP_0$-almost surely for all $ t\in \N_0$.
\end{itemize}
Then, $\lim_{t \rightarrow \infty} \Delta_t(x,a) =0 \qquad \PP_0$-almost surely for all $(x,a) \in \X \times A$.
\end{lem}
To keep the paper self-contained, we further state the so called \emph{Azuma--Hoeffding Inequality}, see, e.g., \citep{mcdiarmid1989method} or \citep{mcdiarmid1998concentration}.
\begin{lem}[\citealp{azuma1967weighted}]\label{lem:azu}
Let $(M_t,\mathcal{F}_t)_{t\in \N_0}$ be a martingale with respect to the filtration
$(\mathcal{F}_t)_{t\in \N_0}$.  
Assume that the martingale differences are almost surely bounded, i.e.,
there exist constants $c_t > 0$ such that
\[
  |M_t - M_{t-1}| \le c_t \quad \text{a.s. for all } t\in \N
\]
Then for any $\varepsilon > 0$, $n \in \N$ we have
\[
  \mathbb{P}\!\left(
    |M_n - M_0| \ge \varepsilon
  \right)
  \le
  2 \exp\!\left(
    - \frac{\varepsilon^2}{2 \sum_{t=1}^n c_t^2}
  \right).
\]
\end{lem}
Moreover, we make use of the following concentration inequality for empirical distributions.
\begin{lem}[\citep{weissman2003inequalities}, Theorem 2.1]\label{lem:concentration}
Let $\Q$ be a probability distribution on $\X$.
Let $\widehat{\Q}_n$ denote the empirical distribution of $n$ samples from $\Q$. Then for all $\varepsilon>0$
\[
\PP\left(\|\widehat{\Q}_n-\Q\|_1\geq \varepsilon\right) \leq \left(2^{|\X|}-2\right) \exp  \left(-\frac{n \varepsilon^2}{2}\right).
\]
\end{lem}

\subsection{Proofs of the results from Sections~\ref{sec:setting} and~\ref{sec:algorithm}}\label{subsec:proofs}
In this subsection we provide the proofs of Section \ref{subsec:defop} and Section \ref{sec:algorithm}.

\begin{proof}[Proof of Proposition~\ref{prop:DPP}]
Note that the equality $\max_{a\in A}Q^*(x,a)=\mathcal{T}V(x)$ directly follows by definition of the operator $\mathcal{T}$, and it only remains to show $\mathcal{T}V = V$. To this end, we aim at applying \citep[Theorem 3.1]{neufeld2023markov} with $p=0$ and verify its assumptions, i.e., we verify \citep[Assumption 2.2]{neufeld2023markov} and \citep[Assumption 2.4]{neufeld2023markov} by showing the following four properties.
\begin{itemize}
\item[(i)] The set-valued map \begin{equation*}
\begin{aligned}
    \X \times A & \rightarrow 2^{\mathcal{M}_1(\X)}\\
    (x,a) & \twoheadrightarrow \mathcal{P}(x,a):=\left\{\mathbb{P}^{(1)}(x,a),\cdots ,\mathbb{P}^{(N)}(x,a)\right\}
\end{aligned}
\end{equation*}
is nonempty, compact-valued and continuous\footnote{Continuity of a set-valued map means the map is both upper hemicontinuous and lower hemicontinuous.}.
\item[(ii)] There exists $C_P \geq 1$ such that for all $(x,a) \in \X\times A, \mathbb{P}\in\mathcal{P}(x,a)$ we have
    \begin{equation}\label{existCp}
     \int_X (1+\|y\|^p)\mathbb{P}(dy) \le C_P(1+\|x\|^p)
    \end{equation}
    \item[(iii)] The reward function $r:\X\times A\times \X \rightarrow \mathbb{R}$ is Lipschitz-continuous in its first two components.
    \item[(iv)] We have $0 < \alpha < \frac{1}{C_P}$, where $C_P$ is defined in (ii).
\end{itemize}

To see that (i) holds, we note that $\mathcal{P}$ is by definition non-empty. It is also compact for all $(x,a) \in \X \times A$ since the set $\mathcal{P}(x,a)$ is finite. 
To show the upper hemicontinuity, we consider some $(x,a) \in \mathcal{X}\times A$ and we let $(x_n,a_n,\PP_n)_{n \in \N} \subseteq \operatorname{Gr}(\mathcal{P}):= \{(x,a,\PP)~|~(x,a) \in \mathcal{X}\times A, \PP \in \mathcal{P}(x,a)\}$ be a sequence such that $(x_n,a_n) \rightarrow (x,a) \in \mathcal{X}\times A $ as $n \rightarrow \infty$. Since both $\mathcal{X}$ and $A$ are finite, we then have $(x_n,a_n)=(x,a)$ for $n$ large enough. Hence, $\PP_n \in \mathcal{P}(x,a)$ for $n$ large enough. As $\mathcal{P}(x,a)$ is finite we also have that $\PP_n = \PP^{(i)}(x,a)$ for infinitely many indices $n$ for some $i \in \{1,\dots,N\}$. Hence, we can find a subsequence $(\PP_{n_k})_{k\in \N}$ such that $\PP_{n_k} \rightarrow \PP^{(i)}(x,a) \in \mathcal{P}(x,a)$ weakly as $k \rightarrow \infty$ and the upper hemicontinuity of $\mathcal{P}$ follows by the characterization provided in \citet[Theorem 17.20]{Aliprantis}.

To show the lower hemicontinuity let $(x,a) \in \X \times A$, $\PP\in \mathcal{P}(x,a)$ and consider a sequence $(x_n,a_n)_{n \in \N} \subseteq \X \times A$ with $(x_n,a_n) \rightarrow (x,a)$ as $n \rightarrow \infty$. By definition of $\mathcal{P}$, there exists some $i_{\PP} \in \{1,\dots,N\}$ such that $\PP^{(i_{\PP})}(x,a) = \PP$. We use this observation to define 
\[
\PP_n:= \PP^{(i_{\PP})}(x_n,a_n) \in \mathcal{P}(x_n,a_n),~n \in \N,
\] and we obtain, since $(x_n,a_n)= (x,a)$ for $n$ large enough that
$$
\PP_n=\PP^{(i_{\PP})}(x_n,a_n)= \PP^{(i_{\PP})}(x,a) = \PP
$$  for $n$ large enough, and hence $\PP_n \rightarrow \PP$ weakly which shows
the lower hemicontinuity of $\mathcal{P}$ according to \citep[Theorem 17.21]{Aliprantis}.

To show (ii), note that $p=0$, and define $C_P: = 1$, then we have
\[
        \int_\X (1+\|y\|^0)\mathbb{P}(dy)  = 2 = C_P \cdot(1+\|x\|^0)
\]
for all $(x,a) \in \X \times A$.
Next, (iii) follows since  for all $x_0,x_0',x_1 \in \mathcal{X}$ and $a,a' \in A$ we have
\[
|r(x_0,a,x_1)-r(x_0',a',x_1)|\leq \left(\max_{\substack{y_0,y_0'\in \mathcal{X},~b,b'\in A\\(y_0,b) \neq (y_0',b')}} \frac{|r(y_0,b,x_1)-r(y_0',b',x_1)|}{\|y_0-y_0'\|+\|b-b'\|}\right) \cdot \left(\|x_0-x_0'\|+\|a-a'\|\right).
\]
The maximum is taken as zero if $\X\times A$ is a singleton. Finally, (iv) follows since $C_P=1$ and as we have $0< \alpha<1$ by assumption.
\end{proof}

\begin{proof}[Proof of Theorem~\ref{thmcv}]
To establish convergence of the $Q$-learning algorithm from Algorithm~\ref{alg:Q-L}, we apply Lemma~\ref{lem_convergence} and verify its three assumptions. We use the fixed target $Q^*$ in the stochastic-approximation recursion, and the empirical fixed points $Q_t^*$ in the estimate of its conditional mean. Since the observed reward satisfies $\mathsf R_{t+1}=r(X_t,a_t(X_t),X_{t+1})$ exactly, we may write the update using $r$ without assuming that this function is known to the learner.

First, for each $t\in\N_0$ and $(x,a)\in\X\times A$, the update rule \eqref{eq:updaterule} can be written equivalently as
\begin{align*}
Q_{t+1}(x,a)
&=Q_t(x,a)+\gamma_t(x,a,X_t)
\left(r(x,a,X_{t+1})+\alpha\max_{b\in A}Q_t(X_{t+1},b)-Q_t(x,a)\right),
\end{align*}
where we recall that
\[
\gamma_t(x,a,X_t)
:=\one_{\{(x,a)=(X_t,a_t(X_t))\}}\,\widetilde\gamma_{N_t(x,a)}.
\]
Define
\begin{align*}
\Delta_t(x,a)&:=Q_t(x,a)-Q^*(x,a),\\
F_t(x,a)&:=\one_{\{(x,a)=(X_t,a_t(X_t))\}}
\left(r(x,a,X_{t+1})+\alpha\max_{b\in A}Q_t(X_{t+1},b)-Q^*(x,a)\right).
\end{align*}
Then, for all $(x,a)\in\X\times A$,
\begin{align*}
\Delta_{t+1}(x,a)
&=Q_t(x,a)-Q^*(x,a)\\
&\quad+\gamma_t(x,a,X_t)
\left(r(x,a,X_{t+1})+\alpha\max_{b\in A}Q_t(X_{t+1},b)-Q_t(x,a)\right)\\
&=\Delta_t(x,a)+\gamma_t(x,a,X_t)
\left(r(x,a,X_{t+1})+\alpha\max_{b\in A}Q_t(X_{t+1},b)-Q^*(x,a)-\Delta_t(x,a)\right)\\
&=(1-\gamma_t(x,a,X_t))\Delta_t(x,a)+\gamma_t(x,a,X_t)F_t(x,a).
\end{align*}
In the last equality, we used that the visit indicator occurs in both $\gamma_t$ and $F_t$, and that its square equals itself. In particular, the equality also holds at pairs that are not updated at time $t$.

Let $(\mathcal G_t)_{t\in\N_0}$ be the filtration described before Theorem~\ref{thmcv}: $\mathcal G_t$ contains the initial data, all simulator observations and action randomizations before iteration $t$, and the randomization used to select the current action, but not the fresh simulator draws of iteration $t$. Thus $X_t$, $a_t(X_t)$, $Q_t$, the empirical kernels and reward means, and $k_t^*$ are $\mathcal G_t$-measurable. Consequently, $\Delta_0(x,a)$ and $\gamma_0(x,a,X_0)$ are $\mathcal G_0$-measurable, and $\Delta_t(x,a)$, $\gamma_t(x,a,X_t)$, and $F_{t-1}(x,a)$ are $\mathcal G_t$-measurable for $t\geq1$.
Condition~(i) of Lemma~\ref{lem_convergence} follows from $0\leq\widetilde\gamma_n\leq1$ and \eqref{eq:gammat}.

We next recall $C_r:=\max_{x,a,y}|r(x,a,y)|<\infty$ and set
\[
B:=\max\left\{\|Q_0\|_\infty,\frac{C_r}{1-\alpha}\right\}.
\]
This yields a uniform bound for the iterates, including arbitrary initialization. Indeed, $C_r+\alpha B\leq B$, and if $\|Q_t\|_\infty\leq B$, then
\[
\left|r(x,a,y)+\alpha\max_{b\in A}Q_t(y,b)\right|
\leq C_r+\alpha\|Q_t\|_\infty\leq C_r+\alpha B\leq B.
\]
Since the update is a convex combination with weights $1-\gamma_t(x,a,X_t)$ and $\gamma_t(x,a,X_t)$, it follows that
\begin{align*}
|Q_{t+1}(x,a)|
&\leq(1-\gamma_t(x,a,X_t))|Q_t(x,a)|\\
&\quad+\gamma_t(x,a,X_t)
\left|r(x,a,X_{t+1})+\alpha\max_{b\in A}Q_t(X_{t+1},b)\right|\\
&\leq(1-\gamma_t(x,a,X_t))B+\gamma_t(x,a,X_t)B=B.
\end{align*}
Induction therefore gives $\|Q_t\|_\infty\leq B$ for all $t\in\N_0$. Moreover, bounded rewards imply $\|V\|_\infty\leq C_r/(1-\alpha)$ and hence
\[
\|Q^*\|_\infty\leq C_r+\alpha\|V\|_\infty
\leq\frac{C_r}{1-\alpha}\leq B.
\]
The same argument for the empirical ambiguity set gives $\|Q_t^*\|_\infty\leq C_r/(1-\alpha)$ for every $t$.

To verify condition~(ii), define the random operator
\begin{equation}\label{eq:opHtilde}
\widetilde{\mathcal H}_tq(x,a)
:=\mathbb E_{\mathbb P^{(k_t^*)}(x,a)}
\left[r(x,a,X_1)+\alpha\max_{b\in A}q(X_1,b)\right],
\qquad (x,a)\in\X\times A.
\end{equation}
Here $k_t^*$ is the index selected by Algorithm~\ref{alg:Q-L}, and the expectation uses its true transition kernel. The minimizing property of this index is used only at the currently visited pair $(X_t,a_t(X_t))$.
At that pair, if $N_t(x,a)>0$, identity~\eqref{eq:empirical_reward_identity} gives
\[
k_t^*\in\argmin_{k=1,\ldots,N}
\mathbb E_{\mathbb P_t^{(k)}(x,a)}
\left[r(x,a,X_1)+\alpha\max_{b\in A}Q_t(X_1,b)\right].
\]
If $N_t(x,a)=0$, the empirical laws all equal $\nu(x,a)$ and the empirical reward means all equal zero. All candidates then tie both in the observable criterion \eqref{eq_definition_k_t} and in the expectation above. Thus the same minimizing property holds at the first visit as well.

Write
\[
\eta_t:=\max_{(x,a)\in\X\times A}\max_{k=1,\ldots,N}
\sum_{y\in\X}\left|\mathbb P_t^{(k)}(x,a)(y)-\mathbb P^{(k)}(x,a)(y)\right|.
\]
Because $\gamma_t(x,a,X_t)\leq\one_{\{(x,a)=(X_t,a_t(X_t))\}}$, assumption~\eqref{eq:gammat} implies infinitely many visits to every state--action pair. Lemma~\ref{lem:P_t}, together with finiteness of $\X$, $A$, and the candidate family, therefore yields $\eta_t\to0$ almost surely. The same visit condition allows us to apply Lemma~\ref{lem:Q_t}, which gives
\[
\|Q_t^*-Q^*\|_\infty\longrightarrow0
\qquad\mathbb P_{x_0,\mathbf a}\text{-almost surely}.
\]

Now fix $(x,a)\in\X\times A$. On the event $\{(x,a)=(X_t,a_t(X_t))\}$, the minimizing property established above implies
\begin{align*}
&\left|\widetilde{\mathcal H}_tQ_t(x,a)-\mathcal H_tQ_t(x,a)\right|\\
&=\left|\sum_{y\in\X}
\left(\mathbb P^{(k_t^*)}(x,a)(y)-\mathbb P_t^{(k_t^*)}(x,a)(y)\right)
\left(r(x,a,y)+\alpha\max_{b\in A}Q_t(y,b)\right)\right|\\
&\leq\max_{k=1,\ldots,N}\sum_{y\in\X}
\left|\mathbb P^{(k)}(x,a)(y)-\mathbb P_t^{(k)}(x,a)(y)\right|
\left|r(x,a,y)+\alpha\max_{b\in A}Q_t(y,b)\right|\\
&\leq\left(C_r+\alpha\|Q_t\|_\infty\right)
\max_{k=1,\ldots,N}\sum_{y\in\X}
\left|\mathbb P^{(k)}(x,a)(y)-\mathbb P_t^{(k)}(x,a)(y)\right|\\
&\leq(C_r+\alpha B)\eta_t=:e_t.
\end{align*}
In particular, $e_t\to0$ almost surely.
Since $k_t^*$ is chosen before the fresh simulator observations, the conditional law \eqref{eq_definition_P_measure} gives
\begin{align*}
&\left|\mathbb E_{\mathbb P_{x_0,\mathbf a}}[F_t(x,a)\mid\mathcal G_t]\right|\\
&=\one_{\{(x,a)=(X_t,a_t(X_t))\}}
\left|\mathbb E_{\mathbb P_{x_0,\mathbf a}}
\left[r(x,a,X_{t+1})+\alpha\max_{b\in A}Q_t(X_{t+1},b)-Q^*(x,a)\mid\mathcal G_t\right]\right|\\
&=\one_{\{(x,a)=(X_t,a_t(X_t))\}}
\left|\widetilde{\mathcal H}_tQ_t(x,a)-Q^*(x,a)\right|\\
&\leq\one_{\{(x,a)=(X_t,a_t(X_t))\}}
\left(\left|\mathcal H_tQ_t(x,a)-Q^*(x,a)\right|+e_t\right).
\end{align*}
Using the empirical fixed-point identity $\mathcal H_tQ_t^*=Q_t^*$ from Lemma~\ref{lem:HQstar}, we further obtain
\begin{align*}
\left|\mathcal H_tQ_t(x,a)-Q^*(x,a)\right|
&\leq\left|\mathcal H_tQ_t(x,a)-\mathcal H_tQ_t^*(x,a)\right|
+\left|Q_t^*(x,a)-Q^*(x,a)\right|.
\end{align*}
Taking the maximum over $(x,a)$, applying the contraction estimate from Lemma~\ref{lem:mapsq}, and then using the triangle inequality yields
\begin{equation}\label{formexpectedvalue}
\begin{aligned}
\left\|\mathbb E_{\mathbb P_{x_0,\mathbf a}}[F_t\mid\mathcal G_t]\right\|_\infty
&\leq\|\mathcal H_tQ_t-\mathcal H_tQ_t^*\|_\infty
+\|Q_t^*-Q^*\|_\infty+e_t\\
&\leq\alpha\|Q_t-Q_t^*\|_\infty+\|Q_t^*-Q^*\|_\infty+e_t\\
&\leq\alpha\|Q_t-Q^*\|_\infty+(1+\alpha)\|Q_t^*-Q^*\|_\infty+e_t\\
&=\alpha\|\Delta_t\|_\infty+c_t,
\end{aligned}
\end{equation}
where
\[
c_t:=(C_r+\alpha B)\eta_t+(1+\alpha)\|Q_t^*-Q^*\|_\infty
\longrightarrow0\qquad\mathbb P_{x_0,\mathbf a}\text{-almost surely}.
\]
This verifies condition~(ii) of Lemma~\ref{lem_convergence} with $\delta=\alpha\in(0,1)$.

It remains to verify condition~(iii). For brevity, write
\[
v_t(y):=\max_{b\in A}Q_t(y,b),\qquad
v^*(y):=\max_{b\in A}Q^*(y,b).
\]
For every $y\in\X$, we have
\[
|v_t(y)-v^*(y)|
\leq\max_{b\in A}|Q_t(y,b)-Q^*(y,b)|
\leq\|\Delta_t\|_\infty.
\]
The visit indicator is $\mathcal G_t$-measurable. Adding a constant or changing the sign leaves the variance unchanged. Hence, with
\[
Z_t(x,a,y):=-r(x,a,y)-\alpha v_t(y)+\alpha\min_{z\in\X}v^*(z),
\]
we obtain
\begin{equation}\label{eq_Var_ineq}
\begin{aligned}
&\operatorname{Var}_{\mathbb P_{x_0,\mathbf a}}(F_t(x,a)\mid\mathcal G_t)\\
&=\one_{\{(x,a)=(X_t,a_t(X_t))\}}
\operatorname{Var}_{\mathbb P^{(k_t^*)}(x,a)}
\left(r(x,a,X_1)+\alpha v_t(X_1)-Q^*(x,a)\right)\\
&=\one_{\{(x,a)=(X_t,a_t(X_t))\}}
\operatorname{Var}_{\mathbb P^{(k_t^*)}(x,a)}\left(Z_t(x,a,X_1)\right)\\
&\leq\operatorname{Var}_{\mathbb P^{(k_t^*)}(x,a)}\left(Z_t(x,a,X_1)\right),
\end{aligned}
\end{equation}
where the quantities determined by $\mathcal G_t$ are held fixed in the variance on the right-hand side.
For each $y\in\X$, an upper bound for this random variable is
\begin{align*}
Z_t(x,a,y)
&\leq C_r-\alpha v_t(y)+\alpha v^*(y)\\
&\leq C_r+\alpha|v^*(y)-v_t(y)|\\
&\leq C_r+\alpha\|\Delta_t\|_\infty=:M_t.
\end{align*}
For a lower bound, note that
\[
\max_{y\in\X}v^*(y)-\min_{z\in\X}v^*(z)
\leq2\|Q^*\|_\infty\leq\frac{2C_r}{1-\alpha}.
\]
Consequently,
\begin{align*}
Z_t(x,a,y)
&\geq-C_r+\alpha\bigl(v^*(y)-v_t(y)\bigr)
-\alpha\left(v^*(y)-\min_{z\in\X}v^*(z)\right)\\
&\geq-C_r-\alpha|v_t(y)-v^*(y)|
-\alpha\left(\max_{z\in\X}v^*(z)-\min_{z\in\X}v^*(z)\right)\\
&\geq-C_r-\alpha\|\Delta_t\|_\infty-\frac{2\alpha C_r}{1-\alpha}=:m_t.
\end{align*}
Thus $m_t\leq Z_t(x,a,y)\leq M_t$ for every successor state. Conditional on $\mathcal G_t$, these bounds are fixed, so Popoviciu's variance inequality applies. Indeed, for a random variable $Z$ taking values in $[m,M]$, the inequality $(Z-m)(M-Z)\geq0$ implies, writing $\mu=\mathbb E[Z]$,
\[
\operatorname{Var}(Z)
\leq(M-\mu)(\mu-m)
=\frac{(M-m)^2}{4}-\left(\mu-\frac{M+m}{2}\right)^2
\leq\frac{(M-m)^2}{4}.
\]
Applying this inequality to \eqref{eq_Var_ineq} gives
\begin{align*}
\operatorname{Var}_{\mathbb P_{x_0,\mathbf a}}(F_t(x,a)\mid\mathcal G_t)
&\leq\frac14(M_t-m_t)^2\\
&=\frac14\left(2C_r+2\alpha\|\Delta_t\|_\infty+\frac{2\alpha C_r}{1-\alpha}\right)^2\\
&=\left(\frac{C_r}{1-\alpha}+\alpha\|\Delta_t\|_\infty\right)^2\\
&\leq\frac{2C_r^2}{(1-\alpha)^2}+2\alpha^2\|\Delta_t\|_\infty^2\\
&\leq C\left(1+\|\Delta_t\|_\infty\right)^2,
\end{align*}
where the deterministic constant
\[
C:=1+\frac{2C_r^2}{(1-\alpha)^2}+2\alpha^2>0
\]
does not depend on $t$ or $(x,a)$. Taking the maximum over all pairs proves condition~(iii).

All assumptions of Lemma~\ref{lem_convergence} are therefore satisfied, and we conclude that
\[
\Delta_t(x,a)\longrightarrow0
\qquad\mathbb P_{x_0,\mathbf a}\text{-almost surely for every }(x,a)\in\X\times A.
\]
Since $\Delta_t=Q_t-Q^*$ and $\X\times A$ is finite, this also gives $\|Q_t-Q^*\|_\infty\to0$ almost surely, proving the assertion.
\end{proof}
\begin{proof}[Proof of Theorem~\ref{thm:finite_time_rate_alg1}]
\medskip\noindent\emph{Step 1: simultaneous confidence bounds for the empirical kernels.}
For each $(x,a)$ and $k$, realize the simulator calls as successive observations from an i.i.d. sequence with law $\mathbb P^{(k)}(x,a)$; write $\widehat{\mathbb P}^{(k)}_n(x,a)$ for the empirical law of its first $n$ observations. For $n\geq1$, define
\[
u_\delta(n):=\sqrt{\frac{2}{n}\log\left(\frac{\pi^2Nd(2^{|\X|}-2)n^2}{6\delta}\right)}.
\]
The concentration lemma gives
\begin{align*}
\mathbb P\left(
\|\widehat{\mathbb P}^{(k)}_n(x,a)-\mathbb P^{(k)}(x,a)\|_1>u_\delta(n)
\right)
&\leq(2^{|\X|}-2)\exp\left(-\frac{nu_\delta(n)^2}{2}\right)\\
&=(2^{|\X|}-2)
\exp\left[-\log\left(\frac{\pi^2Nd(2^{|\X|}-2)n^2}{6\delta}\right)\right]\\
&=\frac{6\delta}{\pi^2Nd n^2}.
\end{align*}
Define $\mathcal E_\delta$ to be the event that
\[
\|\widehat{\mathbb P}^{(k)}_n(x,a)-\mathbb P^{(k)}(x,a)\|_1\leq u_\delta(n)
\quad\text{for every }n\geq1,\ (x,a)\in\X\times A,\ k=1,\ldots,N.
\]
There are $Nd$ choices of $(x,a,k)$. The union bound over these choices and all $n\geq1$ gives
\begin{align*}
\mathbb P(\mathcal E_\delta^c)
&\leq\sum_{n=1}^\infty Nd\,\frac{6\delta}{\pi^2Nd n^2}\\
&=\frac{6\delta}{\pi^2}\sum_{n=1}^\infty\frac1{n^2}
=\frac{6\delta}{\pi^2}\frac{\pi^2}{6}
=\delta.
\end{align*}
This union bound does not require the events to be independent. Also, for any probability laws $p,q$,
\[
\|p-q\|_1=\sum_y|p(y)-q(y)|
\leq\sum_y(p(y)+q(y))=2.
\]
Thus, on $\mathcal E_\delta$, the estimation error at sample size $n\geq1$ is at most $\min\{2,u_\delta(n)\}=b_\delta(n)$. At $n=0$, both the true and initialized laws are probability measures, so their distance is at most $2=b_\delta(0)$.

At any visited pair, $\mathbb P_T^{(k)}(x,a)$ is precisely the empirical law after $N_T(x,a)$ observations. Because $\mathcal E_\delta$ bounds the empirical laws at every sample size, it also applies at these random, data-dependent visit counts. For an unvisited pair, the $n=0$ bound applies. Therefore, on the same event, simultaneously for every $T$,
\[
\max_{x,a,k}\|\mathbb P_T^{(k)}(x,a)-\mathbb P^{(k)}(x,a)\|_1
\leq\max_{x,a}b_\delta(N_T(x,a))=e_T(\delta).
\]

\medskip\noindent\emph{Step 2: error from the empirical Bellman residual.}
Fix $T$ and let $\widehat Q_T^*$ be the unique fixed point of the observable $\alpha$-contraction $\widehat{\mathcal H}_T$. Then
\begin{align*}
\|Q_T-\widehat Q_T^*\|_\infty
&\leq\|Q_T-\widehat{\mathcal H}_TQ_T\|_\infty
+\|\widehat{\mathcal H}_TQ_T-\widehat{\mathcal H}_T\widehat Q_T^*\|_\infty\\
&\leq R_T+\alpha\|Q_T-\widehat Q_T^*\|_\infty.
\end{align*}
Subtracting the last term and dividing by $1-\alpha>0$ gives
\[
\|Q_T-\widehat Q_T^*\|_\infty\leq\frac{R_T}{1-\alpha}.
\]

\medskip\noindent\emph{Step 3: error from the empirical reward--continuation expectations.}\par\noindent
The reward bound gives $\|Q^*\|_\infty\leq C_r/(1-\alpha)$. Set, for each $(x,a)$,
\[
g^*_{x,a}(y):=r(x,a,y)+\alpha\max_{b\in A}Q^*(y,b),
\qquad |g^*_{x,a}(y)|\leq\frac{C_r}{1-\alpha}.
\]
If $N_T(x,a)>0$, identity~\eqref{eq:empirical_reward_identity} and the inequality for minima give, on $\mathcal E_\delta$,
\begin{align*}
|\widehat{\mathcal H}_TQ^*(x,a)-\mathcal HQ^*(x,a)|
&\leq\max_k\left|\sum_y
\bigl(\mathbb P_T^{(k)}(x,a)(y)-\mathbb P^{(k)}(x,a)(y)\bigr)g^*_{x,a}(y)\right|\\
&\leq\frac{C_r}{1-\alpha}\max_k
\|\mathbb P_T^{(k)}(x,a)-\mathbb P^{(k)}(x,a)\|_1\\
&\leq\frac{C_r}{1-\alpha}\,b_\delta(N_T(x,a)).
\end{align*}
If $N_T(x,a)=0$, put $v^*(y):=\max_bQ^*(y,b)$. The initialization gives $\bar r_T^{(k)}(x,a)=0$ and $\mathbb P_T^{(k)}(x,a)=\nu(x,a)$ for every $k$. Consequently,
\begin{align*}
|\widehat{\mathcal H}_TQ^*(x,a)-\mathcal HQ^*(x,a)|
&\leq\max_k\left|\alpha\sum_y\nu(x,a)(y)v^*(y)
-\sum_y\mathbb P^{(k)}(x,a)(y)\bigl(r(x,a,y)+\alpha v^*(y)\bigr)\right|\\
&\leq C_r+2\alpha\|Q^*\|_\infty
\leq\frac{(1+\alpha)C_r}{1-\alpha}
\leq\frac{C_r}{1-\alpha}\,b_\delta(0).
\end{align*}
Thus, simultaneously for all $T$,
\[
\|\widehat{\mathcal H}_TQ^*-\mathcal HQ^*\|_\infty
\leq\frac{C_r}{1-\alpha}e_T(\delta).
\]
Since $\widehat Q_T^*=\widehat{\mathcal H}_T\widehat Q_T^*$ and $Q^*=\mathcal HQ^*$, contraction yields
\begin{align*}
\|\widehat Q_T^*-Q^*\|_\infty
&\leq\|\widehat{\mathcal H}_T\widehat Q_T^*-\widehat{\mathcal H}_TQ^*\|_\infty
+\|\widehat{\mathcal H}_TQ^*-\mathcal HQ^*\|_\infty\\
&\leq\alpha\|\widehat Q_T^*-Q^*\|_\infty
+\frac{C_r}{1-\alpha}e_T(\delta).
\end{align*}
Rearranging gives
\[
\|\widehat Q_T^*-Q^*\|_\infty
\leq\frac{C_r}{(1-\alpha)^2}e_T(\delta).
\]
Combining this with Step~2 proves
\[
\|Q_T-Q^*\|_\infty
\leq\|Q_T-\widehat Q_T^*\|_\infty+\|\widehat Q_T^*-Q^*\|_\infty
\leq\frac{R_T}{1-\alpha}+\frac{C_r}{(1-\alpha)^2}e_T(\delta).
\]
This holds for every $T$ on the same event $\mathcal E_\delta$, proving \eqref{eq:finite_sample_certificate}. In particular, the observed deterministic rewards are already included through the empirical transition laws, without a separate concentration inequality for reward means.

\medskip\noindent\emph{Step 4: visit-count bound and degenerate cases.}
Suppose $N_T^*\geq1$ and define
\[
K_\delta:=\frac{\pi^2Nd(2^{|\X|}-2)}{6\delta}>1.
\]
For each pair let $n=N_T(x,a)$. Since $N_T^*\leq n\leq T$,
\[
\frac{2}{n}\log(K_\delta n^2)
\leq\frac{2}{n}\log(K_\delta T^2)
\leq\frac{2}{N_T^*}\log(K_\delta T^2).
\]
The first inequality uses $n\leq T$ and monotonicity of $\log$; the second uses $n\geq N_T^*$ and $\log(K_\delta T^2)\geq0$. Taking square roots, truncating at $2$, and maximizing over the pairs yields
\[
e_T(\delta)\leq
\min\left\{2,\sqrt{\frac{2}{N_T^*}
\log\left(\frac{\pi^2Nd(2^{|\X|}-2)T^2}{6\delta}\right)}\right\}.
\]
Substituting this bound into \eqref{eq:finite_sample_certificate} proves \eqref{eq:finite_sample_visit_bound}. If $|\X|=1$, every candidate law equals the point mass. At a sampled pair the deterministic reward is then known exactly, so $\widehat{\mathcal H}_TQ^*=\mathcal HQ^*$ there. At an unvisited pair their difference is at most $C_r$, because only the initialized reward mean differs. Thus Steps~2--3 apply deterministically with $e_T(\delta)=2\one_{\{N_T^*=0\}}$. Finally, when both the kernels and their expected rewards are known and used exactly, the implemented Bellman operator equals $\mathcal H$, and Step~2 gives the certificate without an estimation term. Replacing $C_r$ by any known upper bound $\overline C_r\geq C_r$ preserves all inequalities.
\end{proof}

\subsection{Proof for the three-state comparison}\label{subsec:proof_three_state}

\begin{proof}[Proof of Proposition~\ref{prop:three_state}]
Write $M_n:=\min\{C_n^{(1)},C_n^{(2)}\}$. When the old counts differ, $M_n$ increases by one precisely when the selected candidate succeeds. When the old counts coincide, candidate~1 is selected and $M_n$ increases only if both candidates succeed. Consequently,
\[
S_n-M_n=S_{n-1}-M_{n-1}
+\one_{\{C_{n-1}^{(1)}=C_{n-1}^{(2)},\ B_n^{(1)}=1,\ B_n^{(2)}=0\}}.
\]
Since $S_0=M_0=0$, induction gives $S_n=M_n+T_n$, proving \eqref{eq:three_state_identity}.

For the strict comparison in expectation, the risky action is optimal and the regret of selecting the safe action is $\Delta$. Empirical robust planning selects the risky action precisely when $C_n^{(1)}>n\theta$ and $C_n^{(2)}>n\theta$. The two counts are independent binomial random variables, so
\[
\EE[L(a_n^{\rm P})]=\Delta\bigl(1-A_1(n)A_2(n)\bigr).
\]
Conditional on all observations before update $m$, the selected fresh success indicator $B_m^{(k_m)}$ has success probability $p_{k_m}\geq p_1$. To couple these indicators, use independent uniform random variables $U_m$ for the selected observations and another independent sequence for the unselected observations. Since $k_m$ is determined before the new draws, this construction has the prescribed joint law and
\[
B_m^{(k_m)}=\one_{\{U_m\leq p_{k_m}\}}
\geq\one_{\{U_m\leq p_1\}}.
\]
The sum of the indicators on the right has law $\operatorname{Bin}(n,p_1)$. Thus $\PP(S_n>n\theta)\geq A_1(n)$, and \eqref{eq:three_state_q} yields
\[
\EE[L(a_n^{\rm Q})]\leq\Delta\bigl(1-A_1(n)\bigr).
\]
Finally, $0<\theta<p_1\leq p_2<1$ implies $0<A_k(n)<1$ for every finite $n\geq1$. Hence $A_1(n)>A_1(n)A_2(n)$, which proves the strict inequality in \eqref{eq:three_state_regret}.
\end{proof}

\section*{Statements and Declarations}
\noindent\textbf{Funding.} J. Sester gratefully acknowledges support by the NUS Start-Up Grant \emph{Tackling model
uncertainty in Finance with machine learning}, by the MOE AcRF Tier 1 Grant 	 25-0428-P0001, and by the MOE AcRF Tier 2 Grant
T2EP20225-0030.

\medskip\noindent\textbf{Competing interests.} There are no competing interests.

\medskip\noindent\textbf{Data and code availability.} The  implementation of the numerical examples is accessible at \url{https://github.com/juliansester/FiniteQLearning}; the revised results use the corrected implementations supplied with this manuscript.

\medskip\noindent\textbf{Use of AI-assisted tools.} OpenAI's ChatGPT was used during preparation of the manuscript to assist with literature searches, mathematical revisions, computational implementation and checks, and language editing. Responsibility for the manuscript rests with the authors.

\bibliographystyle{plainnat} 
\bibliography{literature}

@article{neufeld2022robust,
  title = {Robust $ Q $-learning Algorithm for {Markov} Decision Processes under {W}asserstein Uncertainty},
  author = {Neufeld, Ariel and Sester, Julian},
  journal = {Automatica},
  volume = {168},
  pages = {111825},
  year = {2024},
  publisher = {Pergamon}
}

@article{azuma1967weighted,
  title = {Weighted sums of certain dependent random variables},
  author = {Azuma, Kazuoki},
  journal = {Tohoku Mathematical Journal, Second Series},
  volume = {19},
  number = {3},
  pages = {357--367},
  year = {1967},
  publisher = {Mathematical Institute, Tohoku University}
}

@article{lu2025distributionally,
  title = {Distributionally Robust Deep {Q}-Learning},
  author = {Lu, Chung I and Sester, Julian and Zhang, Aijia},
  journal = {arXiv preprint arXiv:2505.19058},
  year = {2025}
}

@article{schuermann2014stress,
  title = {Stress testing banks},
  author = {Schuermann, Til},
  journal = {International Journal of Forecasting},
  volume = {30},
  number = {3},
  pages = {717--728},
  year = {2014},
  publisher = {Elsevier}
}

@article{foglia2018stress,
  title = {Stress testing credit risk: a survey of authorities' approaches},
  author = {Foglia, Antonella},
  journal = {International Journal of Central Banking},
  year = {2009},
  volume = {5},
  number = {3},
  pages = {9--45}
}

@article{kapinos2018stress,
  title = {Stress testing banks: whence and whither?},
  author = {Kapinos, Pavel S and Martin, Christopher and Mitnik, Oscar A},
  journal = {Journal of Financial Perspectives},
  volume = {5},
  number = {1},
  year = {2018}
}

@techreport{weissman2003inequalities,
  title = {Inequalities for the {$L_1$} deviation of the empirical distribution},
  author = {Weissman, Tsachy and Ordentlich, Erik and Seroussi, Gadiel and Verd{\'u}, Sergio and Weinberger, Marcelo J.},
  year = {2003},
  institution = {Hewlett-Packard Laboratories},
  number = {HPL-2003-97 (R.1)},
  url = {https://shiftleft.com/mirrors/www.hpl.hp.com/techreports/2003/HPL-2003-97R1.html}
}

@incollection{mcdiarmid1989method,
  title = {On the method of bounded differences},
  author = {McDiarmid, Colin},
  volume = {141},
  pages = {148--188},
  year = {1989},
  publisher = {Cambridge University Press},
  booktitle = {Surveys in Combinatorics, 1989},
  series = {London Mathematical Society Lecture Note Series},
  doi = {10.1017/CBO9781107359949.008},
  editor = {Siemons, J.}
}

@incollection{mcdiarmid1998concentration,
  title = {Concentration},
  author = {McDiarmid, Colin},
  booktitle = {Probabilistic methods for algorithmic discrete mathematics},
  pages = {195--248},
  year = {1998},
  publisher = {Springer}
}

@inproceedings{foerster2017stabilising,
  title = {Stabilising experience replay for deep multi-agent reinforcement learning},
  author = {Foerster, Jakob and Nardelli, Nantas and Farquhar, Gregory and Afouras, Triantafyllos and Torr, Philip HS and Kohli, Pushmeet and Whiteson, Shimon},
  booktitle = {Proceedings of the 34th International Conference on Machine Learning},
  pages = {1146--1155},
  year = {2017},
  volume = {70},
  series = {Proceedings of Machine Learning Research},
  publisher = {PMLR},
  url = {https://proceedings.mlr.press/v70/foerster17b.html}
}

@article{lecun2015deep,
  title = {Deep learning},
  author = {LeCun, Yann and Bengio, Yoshua and Hinton, Geoffrey},
  journal = {Nature},
  volume = {521},
  number = {7553},
  pages = {436--444},
  year = {2015},
  publisher = {Nature Publishing Group UK London}
}

@inproceedings{van2016deep,
  title = {Deep reinforcement learning with double {Q}-learning},
  author = {van Hasselt, Hado and Guez, Arthur and Silver, David},
  booktitle = {Proceedings of the AAAI Conference on Artificial Intelligence},
  volume = {30},
  year = {2016},
  pages = {2094--2100},
  doi = {10.1609/aaai.v30i1.10295}
}

@article{mnih2015human,
  title = {Human-level control through deep reinforcement learning},
  author = {Mnih, Volodymyr and Kavukcuoglu, Koray and Silver, David and Rusu, Andrei A and Veness, Joel and Bellemare, Marc G and Graves, Alex and Riedmiller, Martin and Fidjeland, Andreas K and Ostrovski, Georg and others},
  journal = {Nature},
  volume = {518},
  number = {7540},
  pages = {529--533},
  year = {2015},
  publisher = {Nature Publishing Group UK London}
}

@inproceedings{melo2008analysis,
  title = {An analysis of reinforcement learning with function approximation},
  author = {Melo, Francisco S and Meyn, Sean P and Ribeiro, M Isabel},
  booktitle = {Proceedings of the 25th international conference on Machine learning},
  pages = {664--671},
  year = {2008},
  doi = {10.1145/1390156.1390240}
}

@incollection{baird1995residual,
  title = {Residual algorithms: Reinforcement learning with function approximation},
  author = {Baird, Leemon},
  booktitle = {Machine Learning Proceedings 1995},
  pages = {30--37},
  year = {1995},
  publisher = {Elsevier}
}

@inproceedings{liu2022distributionally,
  title = {Distributionally Robust $ Q $-Learning},
  author = {Liu, Zijian and Bai, Qinxun and Blanchet, Jose and Dong, Perry and Xu, Wei and Zhou, Zhengqing and Zhou, Zhengyuan},
  booktitle = {Proceedings of the 39th International Conference on Machine Learning},
  pages = {13623--13643},
  year = {2022},
  volume = {162},
  series = {Proceedings of Machine Learning Research},
  publisher = {PMLR},
  url = {https://proceedings.mlr.press/v162/liu22a.html}
}

@inproceedings{martin2017count,
  title = {Count-based exploration in feature space for reinforcement learning},
  author = {Martin, Jarryd and Sasikumar, Suraj Narayanan and Everitt, Tom and Hutter, Marcus},
  year = {2017},
  booktitle = {Proceedings of the Twenty-Sixth International Joint Conference on Artificial Intelligence},
  pages = {2471--2478},
  publisher = {International Joint Conferences on Artificial Intelligence Organization},
  doi = {10.24963/ijcai.2017/344},
  url = {https://www.ijcai.org/proceedings/2017/344}
}

@inproceedings{tarbouriech2020active,
  title = {Active model estimation in {M}arkov decision processes},
  author = {Tarbouriech, Jean and Shekhar, Shubhanshu and Pirotta, Matteo and Ghavamzadeh, Mohammad and Lazaric, Alessandro},
  booktitle = {Proceedings of the 36th Conference on Uncertainty in Artificial Intelligence},
  pages = {1019--1028},
  year = {2020},
  volume = {124},
  series = {Proceedings of Machine Learning Research},
  publisher = {PMLR},
  url = {https://proceedings.mlr.press/v124/tarbouriech20a.html}
}

@book{dixit1990optimization,
  title = {Optimization in economic theory},
  author = {Dixit, Avinash K},
  year = {1990},
  publisher = {Oxford University Press, USA}
}

@inproceedings{ccalicsir2019model,
  title = {Model-free reinforcement learning algorithms: A survey},
  author = {{\c{C}}al{\i}{\c{s}}{\i}r, Sinan and Pehlivano{\u{g}}lu, Meltem Kurt},
  booktitle = {2019 27th Signal Processing and Communications Applications Conference (SIU)},
  pages = {1--4},
  year = {2019},
  organization = {IEEE}
}

@article{aguirregabiria2002swapping,
  title = {Swapping the nested fixed point algorithm: A class of estimators for discrete {M}arkov decision models},
  author = {Aguirregabiria, Victor and Mira, Pedro},
  journal = {Econometrica},
  volume = {70},
  number = {4},
  pages = {1519--1543},
  year = {2002},
  publisher = {Wiley Online Library}
}

@incollection{rust1994structural,
  title = {Structural estimation of {M}arkov decision processes},
  author = {Rust, John},
  volume = {4},
  pages = {3081--3143},
  year = {1994},
  publisher = {Elsevier},
  booktitle = {Handbook of Econometrics},
  chapter = {51},
  doi = {10.1016/S1573-4412(05)80020-0}
}

@article{srisuma2012semiparametric,
  title = {Semiparametric estimation of {M}arkov decision processes with continuous state space},
  author = {Srisuma, Sorawoot and Linton, Oliver},
  journal = {Journal of Econometrics},
  volume = {166},
  number = {2},
  pages = {320--341},
  year = {2012},
  publisher = {Elsevier}
}

@book{knight1921risk,
  title = {Risk, uncertainty and profit},
  author = {Knight, Frank Hyneman},
  volume = {31},
  year = {1921},
  publisher = {Houghton Mifflin}
}

@inproceedings{wang2023finite,
  title = {A finite sample complexity bound for distributionally robust {Q}-learning},
  author = {Wang, Shengbo and Si, Nian and Blanchet, Jose and Zhou, Zhengyuan},
  booktitle = {Proceedings of the 26th International Conference on Artificial Intelligence and Statistics},
  pages = {3370--3398},
  year = {2023},
  volume = {206},
  series = {Proceedings of Machine Learning Research},
  publisher = {PMLR},
  url = {https://proceedings.mlr.press/v206/wang23b.html}
}

@book{Aliprantis,
  author = {Aliprantis, Charalambos D. and Border, Kim C.},
  title = {Infinite dimensional analysis},
  edition = {Third},
  note = {A hitchhiker's guide},
  publisher = {Springer, Berlin},
  year = {2006},
  pages = {xxii+703},
  isbn = {978-3-540-32696-0; 3-540-32696-0},
  mrclass = {46-01 (28D05 46N10 47-01 54-01 60B10 60J05)},
  mrnumber = {2378491}
}

@inproceedings{wang2022policy,
  title = {Policy Gradient Method For Robust Reinforcement Learning},
  author = {Wang, Yue and Zou, Shaofeng},
  year = {2022},
  booktitle = {Proceedings of the 39th International Conference on Machine Learning},
  series = {Proceedings of Machine Learning Research},
  volume = {162},
  pages = {23484--23526},
  publisher = {PMLR},
  url = {https://proceedings.mlr.press/v162/wang22at.html}
}

@article{iyengar2005robust,
  title = {Robust dynamic programming},
  author = {Iyengar, Garud N},
  journal = {Mathematics of Operations Research},
  volume = {30},
  number = {2},
  pages = {257--280},
  year = {2005},
  publisher = {INFORMS},
  doi = {10.1287/moor.1040.0129}
}

@article{naghibi2006application,
  title = {Application of {Q}-learning with temperature variation for bidding strategies in market based power systems},
  author = {Naghibi-Sistani, Mohammad Bagher and Akbarzadeh-Tootoonchi, MR and Bayaz, MH Javidi-Dashte and Rajabi-Mashhadi, Habib},
  journal = {Energy Conversion and Management},
  volume = {47},
  number = {11-12},
  pages = {1529--1538},
  year = {2006},
  publisher = {Elsevier}
}

@article{cao2021deep,
  title = {Deep hedging of derivatives using reinforcement learning},
  author = {Cao, Jay and Chen, Jacky and Hull, John and Poulos, Zissis},
  journal = {The Journal of Financial Data Science},
  volume = {3},
  number = {1},
  pages = {10--27},
  year = {2021},
  publisher = {Institutional Investor Journals Umbrella}
}

@incollection{angiuli2021reinforcement,
  title = {Reinforcement learning for mean field games, with applications to economics},
  author = {Angiuli, Andrea and Fouque, Jean-Pierre and Lauri{\`e}re, Mathieu},
  year = {2023},
  booktitle = {Machine Learning and Data Sciences for Financial Markets: A Guide to Contemporary Practices},
  editor = {Capponi, Agostino and Lehalle, Charles-Albert},
  publisher = {Cambridge University Press},
  pages = {393--425},
  doi = {10.1017/9781009028943.022}
}

@article{angiuli2022reinforcement,
  title = {Reinforcement Learning Algorithm for Mixed Mean Field Control Games},
  author = {Angiuli, Andrea and Detering, Nils and Fouque, Jean-Pierre and Lauri{\`e}re, Mathieu and Lin, Jimin},
  journal = {Journal of Machine Learning},
  year = {2023},
  volume = {2},
  number = {2},
  pages = {108--137},
  doi = {10.4208/jml.220915}
}

@article{xu2010distributionally,
  title = {Distributionally Robust {M}arkov Decision Processes},
  author = {Xu, Huan and Mannor, Shie},
  journal = {Mathematics of Operations Research},
  volume = {37},
  number = {2},
  pages = {288--300},
  year = {2012},
  publisher = {INFORMS}
}

@article{charpentier2021reinforcement,
  title = {Reinforcement learning in economics and finance},
  author = {Charpentier, Arthur and {\'E}lie, Romuald and Remlinger, Carl},
  journal = {Computational Economics},
  pages = {425--462},
  year = {2023},
  publisher = {Springer},
  volume = {62},
  number = {1},
  doi = {10.1007/s10614-021-10119-4}
}

@article{kolm2020modern,
  title = {Modern perspectives on reinforcement learning in finance},
  author = {Kolm, Petter N and Ritter, Gordon},
  journal = {The Journal of Machine Learning in Finance},
  volume = {1},
  number = {1},
  year = {2020}
}

@article{huang2020deep,
  title = {Deep learning in finance and banking: A literature review and classification},
  author = {Huang, Jian and Chai, Junyi and Cho, Stella},
  journal = {Frontiers of Business Research in China},
  volume = {14},
  pages = {13},
  year = {2020},
  publisher = {SpringerOpen},
  doi = {10.1186/s11782-020-00082-6}
}

@article{ning2021double,
  title = {Double deep {Q}-learning for optimal execution},
  author = {Ning, Brian and Lin, Franco Ho Ting and Jaimungal, Sebastian},
  journal = {Applied Mathematical Finance},
  volume = {28},
  number = {4},
  pages = {361--380},
  year = {2021},
  publisher = {Taylor \& Francis}
}

@article{wiesemann2013robust,
  title = {Robust {M}arkov decision processes},
  author = {Wiesemann, Wolfram and Kuhn, Daniel and Rustem, Ber{\c{c}}},
  journal = {Mathematics of Operations Research},
  volume = {38},
  number = {1},
  pages = {153--183},
  year = {2013},
  publisher = {INFORMS}
}

@article{bauerle2021distributionally_2,
  title = {Distributionally robust {M}arkov decision processes and their connection to risk measures},
  author = {B{\"a}uerle, Nicole and Glauner, Alexander},
  journal = {Mathematics of Operations Research},
  year = {2022},
  publisher = {INFORMS},
  volume = {47},
  number = {3},
  pages = {1757--1780},
  doi = {10.1287/moor.2021.1187}
}

@article{uugurlu2018robust,
  title = {Robust optimal control using conditional risk mappings in infinite horizon},
  author = {U{\u{g}}urlu, Kerem},
  journal = {Journal of Computational and Applied Mathematics},
  volume = {344},
  pages = {275--287},
  year = {2018},
  publisher = {Elsevier}
}

@incollection{bauerle2021q,
  title = {{Q}-learning for distributionally robust {M}arkov decision processes},
  author = {B{\"a}uerle, Nicole and Glauner, Alexander},
  booktitle = {Modern Trends in Controlled Stochastic Processes: Theory and Applications, Volume III},
  pages = {108--128},
  year = {2021},
  publisher = {Springer},
  doi = {10.1007/978-3-030-76928-4_6}
}

@inproceedings{tokic2011value,
  title = {Value-difference based exploration: adaptive control between epsilon-greedy and softmax},
  author = {Tokic, Michel and Palm, G{\"u}nther},
  booktitle = {Annual conference on artificial intelligence},
  pages = {335--346},
  year = {2011},
  organization = {Springer}
}

@article{si2020distributional,
  title = {Distributionally robust batch contextual bandits},
  author = {Si, Nian and Zhang, Fan and Zhou, Zhengyuan and Blanchet, Jose},
  journal = {Management Science},
  year = {2023},
  volume = {69},
  number = {10},
  pages = {5772--5793},
  doi = {10.1287/mnsc.2023.4678}
}

@inproceedings{si2020distributionally,
  title = {Distributionally robust policy evaluation and learning in offline contextual bandits},
  author = {Si, Nian and Zhang, Fan and Zhou, Zhengyuan and Blanchet, Jose},
  booktitle = {Proceedings of the 37th International Conference on Machine Learning},
  pages = {8884--8894},
  year = {2020},
  volume = {119},
  series = {Proceedings of Machine Learning Research},
  publisher = {PMLR},
  url = {https://proceedings.mlr.press/v119/si20a.html}
}

@inproceedings{zhou2021finite,
  title = {Finite-sample regret bound for distributionally robust offline tabular reinforcement learning},
  author = {Zhou, Zhengqing and Zhou, Zhengyuan and Bai, Qinxun and Qiu, Linhai and Blanchet, Jose and Glynn, Peter},
  booktitle = {Proceedings of the 24th International Conference on Artificial Intelligence and Statistics},
  pages = {3331--3339},
  year = {2021},
  volume = {130},
  series = {Proceedings of Machine Learning Research},
  publisher = {PMLR},
  url = {https://proceedings.mlr.press/v130/zhou21d.html}
}

@article{yang2022towards,
  title = {Toward theoretical understandings of robust {M}arkov decision processes: Sample complexity and asymptotics},
  author = {Yang, Wenhao and Zhang, Liangyu and Zhang, Zhihua},
  journal = {The Annals of Statistics},
  volume = {50},
  number = {6},
  pages = {3223--3248},
  year = {2022},
  publisher = {Institute of Mathematical Statistics}
}

@inproceedings{panaganti2022sample,
  title = {Sample Complexity of Robust Reinforcement Learning with a Generative Model},
  author = {Panaganti, Kishan and Kalathil, Dileep},
  booktitle = {Proceedings of the 25th International Conference on Artificial Intelligence and Statistics},
  pages = {9582--9602},
  year = {2022},
  volume = {151},
  series = {Proceedings of Machine Learning Research},
  publisher = {PMLR},
  url = {https://proceedings.mlr.press/v151/panaganti22a.html}
}

@book{bauerle2011markov,
  title = {{M}arkov decision processes with applications to finance},
  author = {B{\"a}uerle, Nicole and Rieder, Ulrich},
  year = {2011},
  publisher = {Springer Science \& Business Media}
}

@book{dixon2020machine,
  title = {Machine Learning in Finance: From Theory to Practice},
  author = {Dixon, Matthew F and Halperin, Igor and Bilokon, Paul},
  year = {2020},
  publisher = {Springer},
  doi = {10.1007/978-3-030-41068-1}
}

@inproceedings{dvoretzky1956stochastic,
  title = {On stochastic approximation},
  author = {Dvoretzky, Aryeh},
  year = {1956},
  publisher = {University of California Press},
  booktitle = {Proceedings of the Third Berkeley Symposium on Mathematical Statistics and Probability},
  volume = {1},
  pages = {39--55},
  editor = {Neyman, Jerzy},
  url = {https://projecteuclid.org/euclid.bsmsp/1200501645}
}

@article{jaakkola1994convergence,
  title = {On the convergence of stochastic iterative dynamic programming algorithms},
  author = {Jaakkola, Tommi and Jordan, Michael I and Singh, Satinder P},
  journal = {Neural Computation},
  volume = {6},
  number = {6},
  pages = {1185--1201},
  year = {1994},
  publisher = {MIT Press}
}

@article{neufeld2023markov,
  title = {{M}arkov decision processes under model uncertainty},
  author = {Neufeld, Ariel and Sester, Julian and {\v{S}}iki{\'c}, Mario},
  journal = {Mathematical Finance},
  volume = {33},
  number = {3},
  pages = {618--665},
  year = {2023},
  publisher = {Wiley Online Library}
}

@article{singh2000convergence,
  title = {Convergence results for single-step on-policy reinforcement-learning algorithms},
  author = {Singh, Satinder and Jaakkola, Tommi and Littman, Michael L and Szepesv{\'a}ri, Csaba},
  journal = {Machine Learning},
  volume = {38},
  number = {3},
  pages = {287--308},
  year = {2000},
  publisher = {Springer}
}

@techreport{szepesvari1996generalized,
  title = {Generalized {M}arkov decision processes: Dynamic-programming and reinforcement-learning algorithms},
  author = {Szepesv{\'a}ri, Csaba and Littman, Michael L},
  year = {1996},
  institution = {Brown University, Department of Computer Science},
  number = {CS-96-11},
  url = {https://cs.brown.edu/research/pubs/techreports/reports/CS-96-11.html}
}

@inproceedings{van2007convergence,
  title = {Convergence of model-based temporal difference learning for control},
  author = {van Hasselt, Hado and Wiering, Marco A},
  booktitle = {2007 IEEE International Symposium on Approximate Dynamic Programming and Reinforcement Learning},
  pages = {60--67},
  year = {2007},
  organization = {IEEE}
}

@book{villani2008optimal,
  title = {Optimal transport: old and new},
  author = {Villani, C{\'e}dric},
  volume = {338},
  year = {2009},
  publisher = {Springer},
  series = {Grundlehren der mathematischen Wissenschaften},
  doi = {10.1007/978-3-540-71050-9}
}

@phdthesis{watkins1989learning,
  title = {Learning from Delayed Rewards},
  author = {Watkins, Christopher John Cornish Hellaby},
  year = {1989},
  school = {University of Cambridge},
  url = {https://www.repository.cam.ac.uk/items/f1ce2aa9-2e01-4a08-92b7-1eec76b652ed}
}

@article{watkins1992q,
  title = {{Q}-learning},
  author = {Watkins, Christopher J. C. H. and Dayan, Peter},
  journal = {Machine Learning},
  volume = {8},
  number = {3--4},
  pages = {279--292},
  year = {1992},
  publisher = {Springer},
  doi = {10.1007/BF00992698}
}

@article{ElGhaouiNilim2005robust,
  author = {Nilim, Arnab and {El Ghaoui}, Laurent},
  title = {Robust Control of {Markov} Decision Processes
             with Uncertain Transition Matrices},
  journal = {Operations Research},
  volume = {53},
  number = {5},
  pages = {780--798},
  year = {2005},
  doi = {10.1287/opre.1050.0216}
}

@article{mannor2016robust,
  title = {Robust {MDPs} with k-rectangular uncertainty},
  author = {Mannor, Shie and Mebel, Ofir and Xu, Huan},
  journal = {Mathematics of Operations Research},
  volume = {41},
  number = {4},
  pages = {1484--1509},
  year = {2016},
  publisher = {INFORMS}
}

@article{clifton2020q,
  title = {{Q}-learning: Theory and applications},
  author = {Clifton, Jesse and Laber, Eric},
  journal = {Annual Review of Statistics and Its Application},
  volume = {7},
  pages = {279--301},
  year = {2020},
  publisher = {Annual Reviews}
}

@article{jang2019q,
  title = {{Q}-learning algorithms: A comprehensive classification and applications},
  author = {Jang, Beakcheol and Kim, Myeonghwi and Harerimana, Gaspard and Kim, Jong Wook},
  journal = {IEEE Access},
  volume = {7},
  pages = {133653--133667},
  year = {2019},
  publisher = {IEEE}
}

@inproceedings{clavier2024minimax,
  author = {Clavier, Pierre and {Le Pennec}, Erwan and Geist, Matthieu},
  title = {Towards Minimax Optimality of Model-based Robust Reinforcement Learning},
  booktitle = {Proceedings of the Fortieth Conference on Uncertainty in Artificial Intelligence},
  series = {Proceedings of Machine Learning Research},
  volume = {244},
  pages = {820--855},
  year = {2024},
  publisher = {PMLR},
  url = {https://proceedings.mlr.press/v244/clavier24a.html}
}

@inproceedings{he2025online,
  author = {He, Yiting and Liu, Zhishuai and Wang, Weixin and Xu, Pan},
  title = {Sample Complexity of Distributionally Robust Off-Dynamics Reinforcement Learning with Online Interaction},
  booktitle = {Proceedings of the 42nd International Conference on Machine Learning},
  series = {Proceedings of Machine Learning Research},
  volume = {267},
  pages = {22595--22646},
  year = {2025},
  publisher = {PMLR},
  url = {https://proceedings.mlr.press/v267/he25n.html}
}

@article{huang2021stochastic,
  author = {Huang, Wenjie and Haskell, William B.},
  title = {Stochastic Approximation for Risk-Aware {Markov} Decision Processes},
  journal = {IEEE Transactions on Automatic Control},
  volume = {66},
  number = {3},
  pages = {1314--1320},
  year = {2021},
  doi = {10.1109/TAC.2020.2989702}
}

@article{ou2025sequential,
  author = {Ou, Wenfan and Bi, Sheng},
  title = {Sequential decision-making under uncertainty: a robust {MDPs} review},
  journal = {Annals of Operations Research},
  volume = {353},
  pages = {1239--1285},
  year = {2025},
  doi = {10.1007/s10479-025-06738-x}
}

@article{ruszczynski2025risk,
  author = {Ruszczy{\'n}ski, Andrzej and Yang, Shangzhe},
  title = {Risk-averse control of {Markov} systems with value function learning},
  journal = {Annals of Operations Research},
  year = {2025},
  doi = {10.1007/s10479-025-06701-w}
}

@article{wang2024variance,
  author = {Wang, Shengbo and Si, Nian and Blanchet, Jose and Zhou, Zhengyuan},
  title = {Sample Complexity of Variance-Reduced Distributionally Robust {Q}-Learning},
  journal = {Journal of Machine Learning Research},
  volume = {25},
  number = {341},
  pages = {1--77},
  year = {2024},
  url = {https://www.jmlr.org/papers/v25/23-0526.html}
}

\end{document}